\documentclass[12pt]{article} 
\usepackage[utf8x]{inputenc}
\usepackage{tikz} %%% parallel arrows 
\usepackage[all]{xy}
\usepackage{authblk}
\usepackage[babel=true]{csquotes}%%% guillemets
\usepackage{amsmath,amsthm, amssymb,amsfonts}
\usepackage[hmargin=1in,vmargin=1in]{geometry}
\numberwithin{equation}{subsection} %%% ajoute une numérotation des équations avec les sections %%%
\xyoption{arc}
\usepackage{eucal}%%% mathcal change
\usepackage{indentfirst}%%% premier ligne indenter 
\usepackage{mathrsfs}
\usepackage{verbatim}
\usepackage{makeidx}%%%% index 
\usepackage{graphicx}
\usepackage{hyperref}
\usepackage{enumitem} % % % % 
\usepackage{extpfeil} %%%% sert à nommer les fleches avec deux tetes dans un sens : fibrations...
\usepackage{dsfont}
\usepackage[T1]{fontenc}
\newtheorem{thm}{Theorem}[section]
\newtheorem{prop}[thm]{Proposition}
\newtheorem{pdef}[thm]{Proposition-Definition}%%% proposition definition%%%
\newtheorem{lem}[thm]{Lemma}
\newtheorem{cor}[thm]{Corollary}

\newtheoremstyle{bidule}% name
{3pt}% Space above
{3pt}% Space below
{}% Body font
{}% Indent amount
{\scshape}% Theorem head font
{.}% Punctuation after theorem head
{.5em}% Space after theorem head
{}% Theorem head spec (can be left empty, meaning `normal')
\newtheorem{df}[thm]{Definition}
\theoremstyle{definition}
\newtheorem{rmk}[thm]{Remark}
\newtheorem{ex}[thm]{Example}

\newtheorem*{defsans}{Definition} %%% normal mode %%%%

\newtheorem*{note}{Note}

\newtheorem*{warn}{Warning}
\newtheorem*{claim}{Claim}
\newtheorem{nota}[thm]{Notation}

\newcommand{\E}{\mathscr{E}}
\newcommand{\Ev}{\tx{Ev}}
\newcommand{\C}{\mathcal{C}}
\newcommand{\Ub}{\mathcal{U}}%%% Ub comme Oublie%%% forgetful functor%%%%%
\newcommand{\Ca}{\mathcal{C}}
\newcommand{\F}{\mathcal{F}}

\newcommand{\K}{\mathcal{K}}%%%ù K %% genericl locally presentable category %%%%
\newcommand{\Ar}{\text{Arr}}

\newcommand{\D}{\mathcal{D}}
\newcommand{\Ba}{\mathcal{B}}
\newcommand{\Ga}{\mathcal{G}}

\newcommand{\A}{\mathscr{A}}

\newcommand{\B}{\mathscr{B}}
\newcommand{\M}{\mathscr{M}}
\newcommand{\V}{\mathbb{V}}
\newcommand{\Je}{\mathbf{J}_{easy}} %%acyclic cofibrations%%%
\newcommand{\Ie}{\mathbf{I}_{easy}} %%acyclic cofibrations%%%
\newcommand{\Ja}{\mathbf{J}} %%acyclic cofibrations%%%
\newcommand{\J}{\mathcal{J}} %% index category for limit and colimits%%%

\newcommand{\T}{\mathcal{T}}

\newcommand{\Ea}{\mathcal{E}} %%%
 
\newcommand{\W}{\mathscr{W}}
\newcommand{\Fa}{\mathcal{F}}
\newcommand{\I}{\mathbf{I}}
\newcommand{\Un}{\mathbb{I}} %%%% unity in mset

\newcommand{\G}{\mathcal{G}}
\newcommand{\Za}{\mathcal{Z}}

\newcommand{\R}{\mathbb{R}}

\renewcommand{\P}{\mathscr{P}}
\renewcommand{\S}{\mathcal{S}}

\newcommand{\Fb}{\mathbf{F}}%%%%%% Free functor 
\newcommand{\Sim}{\mathscr{S}}
\renewcommand{\to}{\longrightarrow}
\newcommand{\ol}{\overline}
\newcommand{\ul}{\underline}

\renewcommand{\bf}{\mathbf}
\newcommand{\U}{\mathbb{U}}

\newcommand{\Ob}{\text{Ob}}% set of objects  
\newcommand{\n}{\textbf{n}} %object of Delta
\newcommand{\m}{\textbf{m}} % object of Delta
\renewcommand{\k}{\textbf{k}} % object of Delta
\newcommand{\0}{\textbf{0}} % 0 of Delta
\renewcommand{\1}{\textbf{1}} %  1 of Delta
\renewcommand{\2}{\textbf{2}} %  2 of Delta
\newcommand{\tx}{\text}
\newcommand{\tld}{\widetilde}

\renewcommand{\to}{\longrightarrow}
\DeclareMathOperator\Id{Id}
\DeclareMathOperator\Hom{Hom}
\DeclareMathOperator\Homco{\mathcal{H}om}
\DeclareMathOperator\rhom{\R\mathcal{H}om} % hom dérivé
\DeclareMathOperator\HOM{\ul{Hom}}

\DeclareMathOperator\Set{\textbf{Set}} % Category of Set
\DeclareMathOperator\Lax{Lax}%% Lax functors 
\DeclareMathOperator\Cat{\mathbf{Cat}}%% Cat
\DeclareMathOperator\colim{colim}%% colim
\DeclareMathOperator\Graph{\textbf{Graph}}%% Graph
\DeclareMathOperator\msx{\M_{\S}(\tx{$X$})}%% Pre-Coseg-Cat
\DeclareMathOperator\msxsu{\M_{\S}(\tx{$X$})_{\tx{$su$}}}%% Pre-Coseg-Cat
\DeclareMathOperator\msysu{\M_{\S}(\tx{$Y$})_{\tx{$su$}}}%% Pre-Coseg-Cat
\DeclareMathOperator\msxsueasy{\M_{\S}(\tx{$X$})_{\tx{$su,easy$}}}%% Pre-Coseg-Cat
\DeclareMathOperator\msxsuplus{\M_{\S}(\tx{$X$})_{\tx{$su$}}^{+}}%% Pre-Coseg-Cat 
\DeclareMathOperator\msxsuc{\M_{\S}(\tx{$X$})_{\tx{$su$}}^{\mathbf{c}}}
\DeclareMathOperator\msysuc{\M_{\S}(\tx{$Y$})_{\tx{$su$}}^{\mathbf{c}}}
\DeclareMathOperator\msysuplus{\M_{\S}(\tx{$Y$})_{\tx{$su$}}^{+}}%% Pre-Coseg-Cat
\DeclareMathOperator\msxpt{\M_{\S}(\tx{$X$})_{\star}}
\DeclareMathOperator\msxpteasy{\M_{\S}(\tx{$X$})_{\star,\tx{$easy$}}}%% Pre-Coseg-Cat
\DeclareMathOperator\msxaubb{\M_{\S}(\tx{$[X,s,B,\sigma]$})_{\star}}%% Pre-Coseg-Cat
\DeclareMathOperator\msxaaub{\M_{\S}(\tx{$[X,A,s,\sigma]$})_{\star}}%% Pre-Coseg-Cat
\DeclareMathOperator\msy{\M_{\S}(\tx{$Y$})}%% Pre-Coseg-Cat Y= set of object

\DeclareMathOperator\mset{\M_{\S}(\Set)}%% Pre-Coseg-Cat %\mset_{su}
\DeclareMathOperator\msetsu{\mset_{\tx{$su$}}}

\DeclareMathOperator\msetsuc{\mset_{\tx{$su$}}^{\bf{c}}}

\DeclareMathOperator\kset{\K_{\Set}}%% K-Set
\DeclareMathOperator\kx{\K_{\tx{$X$}}}%% K-X
\DeclareMathOperator\kxproj{\kx_{-proj}}%% K-X
\DeclareMathOperator\kxeasy{\kx_{-easy}}%% K-X
\DeclareMathOperator\msxeasy{\msx_{\tx{$easy$}}}%% K-X
\newcommand{\fstar}{f^{\star}}%% f étoile en haut
\newcommand{\gstar}{g^{\star}}%% g étoile
\newcommand{\fex}{{f_!}}%% f exclamtion
\DeclareMathOperator\Le{\mathcal{L}}%% 
\DeclareMathOperator\mcatx{\M\tx{-}\Cat(\tx{$X$})}%%
\DeclareMathOperator\smcatx{\frac{1}{2}\M\tx{-}\Cat(\tx{$X$})}%%  
\DeclareMathOperator\mgraphx{\M\tx{-}\Graph(\tx{$X$})}%% 
\DeclareMathOperator\os{\otimes_{\S}}%% 
\DeclareMathOperator\cof{\mathbf{cof}}%% cofibration 
\DeclareMathOperator\kb{\mathbf{K}} % % % 
\DeclareMathOperator\cn{\mathbf{[n]}} %%% crochet n 
\DeclareMathOperator\msn{\M_{\S}(\cn)}%% Pre-Coseg-Cat n
\DeclareMathOperator\sn{{\S}_{\cn}}%% Pre-Coseg-Cat n
\DeclareMathOperator\kn{\K_{\cn}}%% 
\DeclareMathOperator\op{\tx{op}}%% op tout court
\DeclareMathOperator\msxinj{\M_{\S}(\tx{$X$})_{\tx{inj}}} % % % msx with the injective model structure.
\DeclareMathOperator\msxproj{\M_{\S}(\tx{$X$})_{\tx{proj}}} % % % msx with the projective model structure.
\DeclareMathOperator\sx{\S_{\ol{X}}} % % % %  sx
\DeclareMathOperator\px{\P_{\ol{X}}} % % % %  sx
\DeclareMathOperator\sy{\S_{\ol{Y}}} % % % %  sx
\DeclareMathOperator\Ho{\mathbf{ho}} % % % %  homotopy category
\DeclareMathOperator\sxop{(\S_{\ol{X}})^{2\tx{-$op$}}} % % % %  all  morphisms  % \cite{DHKS}
\DeclareMathOperator\sxopn{(\S_{\ol{X}_{\leq \n}})^{2\tx{-$op$}}} % % % %  all  morphisms  % \cite{DHKS}
\DeclareMathOperator\sxopnp{(\S_{\ol{X}_{\leq \n+\1}})^{2\tx{-$op$}}} % % % %  all  morphisms  % \cite{DHKS}
\DeclareMathOperator\sxopun{(\S_{\ol{X}_{\leq \1}})^{2\tx{-$op$}}} 
 
\DeclareMathOperator\syop{(\S_{\ol{Y}})^{2\tx{-op}}} % % % %  all  morphisms  % \cite{DHKS}%
\DeclareMathOperator\Depiop{\Delta_{\tx{epi}}^{op}}
\DeclareMathOperator\degb{\textbf{deg}} % % % %  degree
\DeclareMathOperator\We{\mathbf{W}_{easy}}
\DeclareMathOperator\cb{\mathbf{c}}  
 
\DeclareMathOperator\laxlatch{\mathbf{Latch}_{lax}}  %% laxlatch 
\DeclareMathOperator\Depi{\Delta_{\tx{epi}}} %%
\DeclareMathOperator\sk{\mathbf{sk}} %%
\DeclareMathOperator\xdiez{\ol{X}<1} %%
\DeclareMathOperator\mcat{\M-\Cat} %%
\DeclareMathOperator\sxdop{(\S_{\ol{X}<1})^{2\tx{-op}}} 
\DeclareMathOperator\Rpt{\mathscr{R}_{\ast}} %
\DeclareMathOperator\dfunc{2\tx{-Func}}%% Lax functors 
\DeclareMathOperator\Finf{\tx{$F$}_{\infty}}%% F-infini limit ou colimit
\DeclareMathOperator\Qs{\mathcal{Q}}%% Q functor colimit
\DeclareMathOperator\linf{\tx{$l$}_{\infty}}%% F-infini limit ou colimit
\DeclareMathOperator\Idinf{\tx{$I$}_{\infty}}%% F-infini limit ou colimit
\DeclareMathOperator\zn{\Za_{\leq \n}}%% F-infini limit ou 
\DeclareMathOperator\znplus{\Za_{\leq \n+ \1}}%% F-infini limit ou 
\DeclareMathOperator\hb{\mathbf{h}}
\DeclareMathOperator\armeasy{\M^{[1]}_{\tx{$easy$}}}
\DeclareMathOperator\av{\alpha_{\downarrow_{\Id_V}}}%% F-infini limit ou colimit
\DeclareMathOperator\elalpha{\ell_{\alpha}}%% F-infini 
\DeclareMathOperator\xialpha{\xi_{\alpha}}%% F-infini 
 
\DeclareMathOperator\coscan{\Sim_{can}} 
 
\DeclareMathOperator\End{\ul{End}} 
 
\title{Pursuing Lax Diagrams and Enrichment}
\author{Hugo V. Bacard 
}
 \affil{Western University}
\date{}
\begin{document}
\maketitle
\begin{abstract}
This paper is part of a project that aims to give a homotopy cousin of Kelly's treatment of enriched category theory. After introducing unital co-Segal $\M$-categories, we establish the unital version of a previous theorem that was proven for the nonunital ones; but this was done under strong hypothesis. We've removed here these assumptions and try to keep the hypothesis on $\M$ as minimal as possible. Our main result provides a sort of Bousfield localization of a model category that is not known to be left proper. In this model structure, every co-Segal $\M$-category is `canonically' equivalent to a strict $\M$-category with the same set of objects. We revisit some constructions of classical enriched category theory to set up the necessary material for our future applications.

%This work is the second piece of a project that aims to develop the category theory of co-Segal $\M$-categories. After introducing unital co-Segal categories, we establish a previous theorem that was proven for the nonunital ones, but this was done under strong hypothesis. We removed these assumptions and try to keep the hypothesis on $\M$ as minimal as possible. We revisit some constructions of classical enriched category theory to set up the necessary material for our upcoming applications. 
\end{abstract}
\setcounter{tocdepth}{1}
\tableofcontents 

\section{Introduction}
The theory of \emph{weakly enriched categories} is a subject of growing interest. The motivations for studying these categories are various and come from different areas of mathematics.  Behind the idea of a weakly enriched category, there is a mixture of two theories; namely \emph{Enriched category theory} and genuine \emph{Higher category theory}; the later being part of the former.

 Enriched category theory is a complex subject on its own; and its development requires different approaches than classical category theory. Some \emph{simple} constructions such as \emph{(co)limits, Kan extensions, comma categories} are harder to define if not impossible to get (see \cite{Ke}).

On the other side Higher category theory is famous for not having a canonical window of study; different formalisms exist and depending of what is intended to be done, one might prefer one approach over another one.\\ 

The theory that is being developed here used $2$-category theory and model categories; and the main idea behind is the notion of co-Segal enriched category \cite{COSEC1,These_Hugo,Bacard2013}. Another very interesting approach that uses the theory of \emph{quasicategories} and \emph{$\infty$-operad} has been recently considered by Haugseng and Gepner \cite{Gep_Hausg}. An obvious question is to determine how one goes from one theory to another, but this is difficult in general. However if $\M$ is either the category of simplicial sets or the category of chain complexes; it seems much easier to do it. \\

When we've started the theory of \emph{co-Segal enriched categories} over a symmetric monoidal model category $\M=(\ul{M},\otimes, I)$,  we had to assume strong hypothesis such as `all objects of $\M$ are cofibrant'. Our goal in this paper and the upcoming ones, is to work with hypothesis that are as minimal as possible. But the price for this, is a long technical theory.  It seems though  that if we want a \emph{Dwyer-Kan model structure}, it will be hard to avoid notions like \emph{interval objects}. We were unable to get this model structure with our actual restrictions for not using many ingredients.\\

The major modification is the definition of weak equivalences we adopt for co-Segal precategories. To understand this let's recall very briefly the structure of a co-Segal precategory. A co-Segal precategory with set of objects $X$, is a lax functor of $2$-categories, $$\C: \sxop \to \M,$$ where $\sx$ is a strict $2$-category with $X$ as objects, and whose $1$-morphisms are finite sequences $(A_0,...,A_n)$ of elements of $X$. It's a \emph{decorated version} of $(\Delta_{epi}^{+},+,\0)$. The data for $\C$ include a family of diagrams $$\C_{AB}:\sx(A,B)^{op} \to \M,$$ 
together with \emph{laxity maps} 
$$ \varphi: \C(A,\cdots,B) \otimes   \C(B,\cdots,C) \to  \C(A,\cdots,B, \cdots,C),$$
that are subjected to a coherence axiom for lax functors.\\

For example we have the following configuration that displays the \emph{algebraic part} and the \emph{simplicial part} of $\C$.      
\[
\xy
(-15,0)*+{\C(A,B) \otimes \C(B,C)}="X";
(30,0)*+{\C(A,B,C)}="Y";
(30,18)*+{\C(A,C)}="E";
{\ar@{->}^-{\varphi}"X";"Y"};
{\ar@{->}^-{}_{}"E";"Y"};
%{\ar@{.>}^-{}"X";"E"};
\endxy
\]

The algebraic part is the laxity map $\varphi:\C(A,B) \otimes \C(B,C)\to \C(A,B,C)$; and the simplicial part is the vertical map $\C(A,C) \to \C(A,B,C)$ which is a piece of the diagram $\C_{AC}:\sx(A,C)^{op} \to \M$.  Being a co-Segal category is demanding that each component $$\C_{AC}:\sx(A,C)^{op} \to \M,$$ lands in the subcategory of weak equivalences in $\M$. 

A key observation is that if every component $\C_{AB}$ is a constant diagram, then we get a strict $\M$-category and this is precisely how we get the inclusion
$$ \tx{$\M$-categories} \hookrightarrow \tx{co-Segal $\M$-categories}.$$

For a general co-Segal category, if we post-compose by the  localization functor $$L: \M \to \Ho(\M),$$ which is a monoidal functor, we find that $L\circ\C$ is isomorphic to a locally  constant lax functor diagram, that is a strict category enriched over $\Ho(\M)$.\\ 

In  \cite{COSEC1} we defined a map $(\sigma,f): \C \to \D$ as a weak equivalence of precategories if each natural transformation 
$\sigma_{AB}: \C_{AB} \to \D_{fA fB}$ is a level-wise weak equivalence in $\M$. But this turns out to be difficult to work with. For example it's not known whether or not we have left properness for precategories if $\M$ is left proper; even when we fix the set of objects. In fact even for strict $\M$-categories left properness is not easy to establish without some assumptions.\\ 

The new definition of weak equivalences we adopt here is:
\begin{defsans}
Say that a map $(\sigma,f): \C \to \D$ is an \emph{easy weak equivalence} if for all $(A,B)$ the component
$\C(A,B) \to \D(fA,fB)$ is a weak equivalence. 
\end{defsans}
This means that we are only imposing a weak equivalence between the corresponding \emph{initial entries} of each diagram. The reason we call them initial entries is because in each category $\sx(A,B)^{op}$, the $1$-morphism corresponding to $(A,B)$ is the initial object, just like $\1$ is initial in $(\Delta_{epi}^{+})^{op}$ (without $\0$).\\

 Now observe that if $\C$ and $\D$ are co-Segal categories, then having  a weak equivalence at one entry is equivalent, by $3$-for-$2$, to having a weak equivalence everywhere. That is to say,  for co-Segal categories, an easy weak equivalence is the same thing as a level-wise weak equivalence. In particular for strict $\M$-categories, an easy weak equivalence is the same thing as a local weak equivalence. This means that our new notion of weak equivalence is still correct for co-Segal categories.\\

The first definition of co-Segal categories that appears in \cite{These_Hugo, COSEC1} is in fact a definition of a non unital weak $\M$-category. We correct this with our notion of unital precategories. 
The main result of this paper is the existence of a model structure on precategories such that fibrant objects are unital co-Segal categories: this is Theorem \ref{fibred_localized}. This theorem appears to be an example of an \emph{implicit Bousfield localization}, in the sense that we are able to produce a model category that behaves like the Bousfield localization of a model category whose Bousfield localization is not guaranteed to exist by classical methods. In fact our previous model structure with level-wise weak equivalence is not left proper (or at least not known to be). \\
% localizes a set of map but there is a model structure with the level-wise weak equivalence that cannot be localized as we cannot guarantee left properness. But we are able to immediately produce this localization\\

This theorem has a better improvement but requires an entire generalization of precategories themselves. This is due to the fact that when we vary the set of objects it's not guaranteed that we have a global \emph{left properness} with the DK-equivalences.\\

Another important result is Theorem \ref{strict-thm} that says that in the new model structure, every co-Segal category is equivalent to a strict category with the same set of objects. This result requires a careful analysis of the structures and is a bit technical.

\paragraph{Plan of the paper} The paper is organized in $5$ sections.\\

In Section \ref{sec-unital}, we define unital precategories and study their properties.  We consider a \emph{unitalization functor} and establish the necessary results that will later ensure that this functor preserves the homotopy type. There is a \emph{hierarchy} in the definition being unital, each of them is \emph{algebraic}. By algebraic we mean that the unity is subject to satisfy some equations as if it was a strict category. There is a much flexible notion of unital category that was first given in \cite{Bacard2013} for example; but this one is not suitable to have a model structure.\\

It's in Section \ref{sec-easy-model} that we establish the first model structure on precategories that we call the \emph{easy model structure}. We also outline a class of precategories that are called \emph{$2$-constant} (Definition \ref{def-deux-const}).  The subcategory of $2$-constant co-Segal categories are the most natural examples of co-Segal categories.\\

We localize this model structure in Section \ref{sec-fibre-loc}; we call this \emph{fibred localization}. We localize the fibers of a Grothendieck bifibration where each fiber is a combinatorial model category which is left proper. What we would like is to have a Bousfield localization of the total category of such bifibration. We will address this issue separately in another work.

In Section \ref{sec-monoidal-coseg}, Section \ref{sec-intern-hom} and Section \ref{sec-distrib} we consider the corresponding notions of monoidal co-Segal category, natural transformation between co-Segal functors and the co-Segal category of functors. Finally we consider the notion of \emph{distributor} also called bimodules. These objects will be used in the upcoming papers.\\ 

In \cite{Bacard_cosdg} we leave the full generality and consider the specific case where $\M$ is the category of chain complexes. In fact the main motivation for all this was precisely the theory of linear co-Segal categories. Thanks to the existence of the model structure we are able to get `for free' classical notions such as \emph{dg-quotient, Serre functors, triangulated categories}, etc.  Notions like \emph{dg-nerve} can be easily adapted to our categories as it was done for $A_\infty$-categories (see \cite{Faonte_dg_nerve}). co-Segal dg-categories are supposed to play the same role as for $A_\infty$-categories and the constructions that exist for $A_\infty$-categories can be defined for co-Segal categories. However it seems difficult to define the \emph{product} of $A_\infty$-categories whereas there is a canonical product for co-Segal categories that is defined here. It's in \cite{Bacard_cosdg} that we will give the first applications of the material developed so far.

\begin{warn}
In this paper all the set theoretical size issues have been left aside \footnote{We can work with universes $\U \subsetneq \V \subsetneq  \cdots$ }. Some of the material provided here are well known facts and we make no claim of inventing or introducing them. Unless otherwise specified when we say `lax functor' we will mean the ones called \emph{normal lax functors} or \emph{normalized lax functor}. These are lax functors $\Fa$ such that the maps `$\Id \to \Fa(\Id)$' are identities and all the laxity maps $\Fa(\Id) \otimes \Fa(f) \to \Fa( \Id \otimes f)$ are natural isomorphisms.
\end{warn} 
\section{Unital co-Segal categories}\label{sec-unital}
\subsection{An important adjunction}
For a set $X$ we have a category $\msx$ of normal lax functors indexed by some strict $2$-category $\sxop$. Below we recall very briefly the structure of $\sx$, the reader can find details in  \cite{These_Hugo,COSEC1}. The objects of $\msx$ are called \emph{co-Segal precategories}.\\

Denote by $\mcatx$ and $\smcatx$, respectively, the category of $\M$-categories and semi-$\M$-categories (= without identities)  for a fixed set of object $X$. The following lemma is of great importance; it can be found in a general version in \cite{toward_strict}. As said in the introduction, co-Segal precategories that are locally constant are precisely semi-$\M$-categories. 
\begin{lem}
For a complete and cocomplete monoidal closed category $\M$, the inclusion:

$$ \smcatx \hookrightarrow \msx $$ 
has a left adjoint
$$|-|: \msx \to \smcatx.$$
\end{lem}

\begin{proof}[Sketch of proof]
The lemma can be proved by some ``abstract nonsense'' using the adjoint functor theorem but there is a direct proof.\\
If $\F: \sxop \to \M$ is a  (normal) lax functor, take the colimit of each component: $$\F_{AB}: \sx^{op}(A,B) \to \ul{M}.$$

As $\M$ is monoidal closed, colimits distribute over $\otimes$ and we get a semi-$\M$ category $|\F|$ where the hom-object is $|\F|(A,B):=\colim \F_{AB}$.
\end{proof}
We have a canonical map of precategories $\delta:\F \to |\F|$. 
\subsection{Pointed and unital precategories }
\subsubsection{Pointed precategories}
We will denote by $\emptyset$ the initial object of $\M=(\ul{M},\otimes, I)$. For a set $X$, we will denote by $I_X$ the $\M$-category defined as follows.

The set of objects of $I_X$ is $X$ and the hom-objects are $I_X(A,B)= I$ if $A=B$; and $I_X(A,B)=\emptyset$ if $A \neq B$. The composition in $I_X$ is the obvious one i.e, either $I \otimes I \cong I$ or $\emptyset \to I$. Usually we call $I_X$ the \emph{discrete category} associated to the set $X$.\\ 

We will denote again $I_X$ its image in $\msx$ by the inclusion functor $$\mcatx \hookrightarrow \msx.$$  
\begin{df}
A \emph{pointed co-Segal precategory} is an object of the under category $$I_X \downarrow \msx.$$
We will write for simplicity $\msxpt$ the category  $I_X \downarrow \msx$ and will denote by $$\Ub: \msxpt \to \msx $$ the forgetful functor.
\end{df}
\begin{prop}
The functor $\Ub$ has a left adjoint $\Rpt: \msx \to \msxpt$.
\end{prop}

\begin{proof}
Define $\Rpt(\F)$ to be the the coproduct $I_X \coprod \F$  in $\msx$.  
\end{proof}
\begin{rmk}
By definition of $\Rpt$, for any pair of objects we have
$$\Rpt(\F)(A,B)= \F(A,B) \quad \tx{if  $A \neq B$}$$
and 
$$\Rpt(\F)(A,A)= \F(A,A) \coprod I \quad \tx{if $A=B$}.$$

Let $\sigma: \F \to \G$ be a morphism in $\msx$ such that $\sigma_{(A,B)}: \F(A,B) \to \G(A,B)$ is a weak equivalence in $\M$. Then by the above formulas, it's easy to see that the map $\Rpt(\sigma)_{(A,B)}: \Rpt(\F)(A,B) \to \Rpt(\G)(A,B) $ is also a weak equivalence if $\M$ satisfies one of the conditions listed below.
\begin{enumerate}
\item The unity $I$ is cofibrant and $\M$ is left proper.
\item For every weak equivalence $f:A \to B$ in $\M$ then  $f \coprod \Id_I: A \coprod I \to B \coprod I$ is also a weak equivalence.
\end{enumerate}
\end{rmk}

Since $\msxpt$ is a comma category, it inherits most of the properties that $\msx$ has. And clearly one has the following proposition.
\begin{prop}
\begin{enumerate}
\item If $\msx$ is complete (resp. cocomplete, locally presentable)  then so is $\msxpt$ respectively.
\item If $\msx$ carries a model structure then so does $\msxpt$ the under model structure.
\end{enumerate}
\end{prop}
\subsubsection{Unital precategories}
Recall that for each $X$ the $2$-category $\sx$ is a sub-$2$-category of a $2$-category $\px$ that we are going to describe very briefly. The $2$-category $\px$ is  characterized by the fact that we have a functorial isomorphism of sets
$$ \Lax(\ol{X}, \B) \cong \dfunc(\px, \B),$$
for any $2$-category $\B$. This is a small version of an adjunction due to Bénabou (see \cite{Chiche} for details). There is a $2$-functor $\degb: \px \to (\Delta^+, +,0)$ which is locally a cofibred category. And $\sx$ is the pullback of that functor along the inclusion $(\Depi^+,+, 0) \hookrightarrow (\Delta^+, +,0)$.\\

We will use small letters $s$ to represent the $1$-morphisms of $\sx$ which are chains of elements of $X$ e.g $s=(A,B,C)$. If $s=(A,...,B)$ and $t=(A, ...., B)$ are two $1$-morphisms of $\sxop$, we will denote by $\sigma: s \to t$, a generic $2$-morphism in $\sxop$. The composition in $\sxop$ will be denoted by $\otimes$. This composition is just the concatenation chains, which is a generalization of the ordinal addition in $\Delta^+$.

\begin{df}\label{unital-precat}
\begin{enumerate}
\item A precategory $\F$ is said to be \emph{relatively unital} if for every $A$, there is a map $I \to \F(A,A)$ such that the composite $I \to \F(A,A) \to |\F|(A,A)$ turns the semi-category $|\F|$ to a category (with identities).
\item  A precategory $\F$ is said to be \emph{strongly unital} if for every $A \in X$ there is a map $I_A: I \to \F(A,A)$ such that for \ul{any} $s= (A,...,B)$ and \ul{any} $\sigma: s \to (A,A) \otimes s$ (resp. any $\sigma: s \to s\otimes (B,B)$), the following diagrams commute.
\[
\xy
(0,18)*+{I \otimes \F(s)}="W";
(0,0)*+{\F(A,A) \otimes \F(s) }="X";
(40,0)*+{\F[(A,A) \otimes s]}="Y";
(40,18)*+{\F(s)}="E";
{\ar@{->}^-{\varphi}"X";"Y"};
{\ar@{->}^-{I_A \otimes \Id}"W";"X"};
{\ar@{->}^-{\cong}"W";"E"};
{\ar@{->}^-{\F(\sigma)}"E";"Y"};
\endxy
\xy
(0,18)*+{\F(s) \otimes I}="W";
(0,0)*+{\F(s)  \otimes  \F(B,B)}="X";
(40,0)*+{\F[s\otimes (B,B)]}="Y";
(40,18)*+{\F(s)}="E";
{\ar@{->}^-{\varphi}"X";"Y"};
{\ar@{->}^-{\Id \otimes I_B}"W";"X"};
{\ar@{->}^-{\cong}"W";"E"};
{\ar@{->}^-{\F(\sigma)}"E";"Y"};
\endxy
\]
\item Similarly we will say that \emph{$\F$ is $n$-strongly unital} if the above diagrams commutes only for $s$ of length $n$.
\end{enumerate}

A morphism $\sigma: \F \to \G$ of strongly (resp. relative)  unital  precategories is a morphism in $\msx$ that takes unity to unity.\\

We will denote by  $\msxsu$ the category of strongly unital precategories.

\end{df}

The following remark is crucial for the rest of the paper. 
\begin{rmk}
It's important to observe that being strongly unital doesn't impose any constraint on the endomorphism-object $\F(A,A)$ nor on any $\F(A,B)$ because they don't receive algebraic data. Indeed, there is no laxity map going to any $\F(A,B)$.

This is important for us because, as we shall see below, there is a \emph{unitalization process} that takes a precategory to a strongly unital one. And this process won't change the value of $\F(A,A)$ nor of $\F(A,B)$! In fact the very basic idea of imposing a single unity condition is to take the coequalizer of these maps:
$$I \otimes \F(s) \to \F(A,A) \otimes \F(s) \xrightarrow{\varphi_j} \F[(A,A) \otimes s];$$ 
$$I \otimes \F(s) \to \F(s) \xrightarrow{\F(\sigma)} \F[(A,A) \otimes s].$$ 

This will change $\F[(A,A) \otimes s]$ but not $\F(A,A)$ nor $\F(A,B)$. 

\begin{note}
For the rest of the paper, we will give most of our results on $\msxsu$, but they hold also for the other categories.
\end{note}
\end{rmk}
\subsection{Unitalization of precategories}
We assume that $\M=(\ul{M},\otimes,I)$ is locally presentable. This is only to make the proof easier and shorter. Under this assumption we proved in \cite{COSEC1}, that $\msx$ is also locally presentable. 

\begin{prop}\label{limit-msxsu}
For a complete and cocomplete monoidal category $\M$ the following hold.
\begin{enumerate}
\item The category $\msxsu$ is complete.
\item Let   $F: \J \to \msxsu$ be a diagram such that the colimit $F_\infty$ exists in $\msx$. Assume that  for each $1$-morphism $s \in \sx$, the object $F_\infty(s)$ together with the  canonical maps attached to $F_\infty$ is a colimit of the diagram $\Ev_s \circ F: \J \to \M$.

Then $F_\infty$ lifts to an object of $\msxsu$ and is the colimit in $\msxsu$ of the diagram $F$.
\item The category $\msxsu$ has coequalizers of reflexive pairs and they are preserved by the forgetful functor.
\item The category $\msxsu$ and is closed under directed colimits. and they are preserved by the forgetful functor. 
\end{enumerate}
\end{prop}
For the proof we give hereafter we simply check that the left invariance diagram holds for $F_\infty$. The argumentation is the same for the right invariance diagram. 
\begin{proof}
The proof is actually an easy exercise in category theory; and the idea is to play with the various commutative diagrams. We expose the proof in Appendix \ref{ap-limit-msxsu}.
\end{proof}

\begin{prop}
Let $\M$ be a symmetric monoidal closed category which is locally presentable and let $\Ub:\msxsu \to \msxpt$ be the forgetful functor. Then the following hold.
\begin{enumerate}
\item The functor $\Ub$ has a left adjoint $\Phi: \msxpt \to \msxsu$.
\item The induced adjunction is a monadic adjunction.
\item The category $\msxsu$ is locally presentable.
\end{enumerate}
\end{prop}

\begin{proof}

Assertion $(1)$ is an application of  the main theorem in \cite{Adamek-Rosicky-reflection} which asserts that for a locally presentable category, any full subcategory that is complete and closed under directed colimits is reflective.

For Assertion $(2)$ we use Beck monadicity theorem (see for example \cite{joy_cat}). The only thing that we need to check is that the $\msxsu$ has coequalizers of reflexive pairs and that $\Ub : \msxsu \to \msxpt$ preserves them. But this is given by the previous proposition.\\

Finally Assertion $(3)$ follows from the fact that the induced monad is finitary i.e, preserves directed colimits. Indeed both $\Ub$ and $\Phi$ preserves directed colimits. Now as finitary monad defined on the locally presentable category $\msxpt$, we conclude by \cite[Remark 2.78]{Adamek-Rosicky-loc-pres} that $\msxsu$, which is equivalent to the category of algebras of the monad, is also locally presentable.
\end{proof}
Unfortunately the above proof doesn't tell us much about $\Phi$. One of the important properties of the adjoint that will be needed is the following:
\begin{thm}
Let $\F \in \msxpt$ be a pointed precategory and  let $(A,B)$ be a pair of objects. Then the left adjoint $\Phi$ doesn't change, up-to an isomorphism, the value of $\F$ at the $1$-morphism $(A,B)$ i.e,
$$\Phi(\F)(A,B)\cong \F(A,B).$$ And if $\sigma : \F \to \G$ is a morphism in $\msxpt$ then the component 
$$\Phi(\sigma)_{(A,B)}: \Phi(\F)(A,B)  \to \Phi(\G)(A,B),  $$
is isomorphic to the component $\sigma_{(A,B)}$, in the category $\M^{[1]}$ (of morphisms of $\M$). 
\end{thm}
The proof of the theorem requires writing down explicitly the left adjoint $\Phi$. Basically $\Phi$ is obtained as a sequential directed colimit of pushouts in $\msxpt$. The construction is tedious but not hard. We will only outline the different steps. But in order to do this properly we need some definitions. 
\subsubsection{Colimits in the category of unital precategories}

Given $\n \in \Depi$ we've considered in \cite{COSEC1}, truncation of the $2$-category $\sx$, that are denoted by $\sx_{\leq \n}$. This means that we only consider $1$-morphisms that are of length $\leq \n$. What we get is not a $2$-category but an \emph{almost $2$-category}, in the sense that the composition is partially defined but remains associative.\\

This is a $2$-dimensional case of what Bonnin \cite{Bonnin} called \emph{groupements}. We have a truncation functor
$$\msx= \Lax[\sxop, \M] \xrightarrow{\tau_n} \Lax[\sxopn, \M].$$
 
We show in \cite{COSEC1} that this functor has a left adjoint $\sk_n: \Lax[\sxopn, \M] \to \Lax[\sxop, \M]$  and that $\tau_n$ \emph{creates colimits}. In fact we showed that each truncation 
$$\tau_n: \Lax[\sxopnp, \M]  \to \Lax[\sxopn, \M] $$ creates colimits. This provides a direct method to prove that the category $\Lax[\sxop, \M]$ is cocomplete. \\

Our purpose is to establish that the truncation for unital precategories also creates colimits.
\begin{nota}
\begin{enumerate}
\item Denote by $\Lax[\sxopn, \M]_{\ast}$ the category of pointed (and truncated) normal lax functors. 
\item For $\n\geq \1$, denote by  $\Lax[\sxopn, \M]_{su}$, the category of pointed (and truncated) normal lax functors that satisfy the unity conditions. Note that for $\n=\1$, there are no unity conditions.
\item We will still denote by $\Ub:\Lax[\sxopn, \M]_{su} \to \Lax[\sxopn, \M]_{\ast}$ the forgetful functor.
\end{enumerate}
\end{nota}
 
\begin{prop}\label{lem-creat-colimit-trois}
With the above notation, the following hold.
\begin{enumerate}
\item The functor $\tau_n: \msxsu \to \Lax[\sxopn, \M]_{su}.$
creates colimits. 
\item Colimits in the category $\msxsu$ are computed level-wise at the $1$-morphisms $(A,B)$. 
\end{enumerate}
\end{prop} 
%%%%%%%%

\begin{proof}
See Appendix \ref{ap-lem-creat-colimit-trois}
\end{proof}
%%%%%%%%
%%%%%%%%%%%%%comment %%%
%%%%%%%%%%%%%%%%%% end comment 
\subsubsection{Gluing adjunctions}
Let $\B$ be a category and let $\lambda$ and $\kappa$ be two infinite regular cardinals with $\lambda < \kappa$. Assume that $\B$ has all $\kappa$-small colimits \footnote{ In most cases we will assume also that $\B$ is locally $\lambda$-presentable (hence locally $\kappa$-presentable)}. Let's start with the following lemma which is a tautology. We mention it because it appears many times in the upcoming constructions.
\begin{lem}[\emph{Crossing lemma}]\label{cross-lem}
Let $C: \lambda \to \B$ and $D: \lambda \to \B$ be two directed diagrams in $\B$. Assume that for every $k \in \lambda$ there exists two maps $\eta_k: C_k \to D_k$ and  $\varepsilon_k: D_k \to C_{k+1}$ such that the structure maps of $C$ and $D$ are respectively the composite below.
 $$C_k \to C_{k+1}= C_k \xrightarrow{\eta_k} D_k \xrightarrow{\varepsilon_k} C_{k+1}$$ 
 $$D_k \to D_{k+1}= D_k \xrightarrow{\varepsilon_k} C_{k+1} \xrightarrow{\eta_{k+1}} D_{k+1}$$

Then $C$ and $D$ have the same colimit and the maps between the colimits that are induced by $\varepsilon_k$ and $\eta_k$ are inverse each other.
\end{lem}
\begin{proof}
Clear.
\end{proof}

Let  $S$ be a $\kappa$ small set and let $N= \{(\A_i, P_i, \eta_i )\}_{i \in S}$ be a family of triples consisting of a \ul{full reflective} subcategory $\A_i \subset \B$, with  $P_i: \B \to A_i$ the reflection functor i.e, the left adjoint to the inclusion $\tau_i:\A_i \hookrightarrow B$; and  $ \eta_i: \Id \to P_i$ the unit of the adjunction. Denote by $\A= \cap_{i \in S} \A_i$ the \emph{intersection category} which is the full subcategory of $\B$ whose set of objects is $\Ob(\A)=\cap_{i \in S} \Ob(\A_i) $.\\  

For every $F \in \B$, consider $F_k$ the $\lambda$-directed diagram defined as follows. 

\begin{enumerate}
\item $F_0= F$
\item $F_{k+1}$ is the object obtained by taking the wide pushout $\{ F_k \xrightarrow{\eta_i} P_i(F_k) \}_{i \in S}$.
\item The structure map $\delta_k: F_k \to F_{k+1}$ is the canonical map going into the pushout. 
\item Denote by $\varepsilon_k^i: P_i(F_k) \to F_{k+1}$ the canonical map going also into the pushout object. In particular $\delta_k = \varepsilon_k^i \circ \eta_i$:
$$F_k \xrightarrow{\delta_k} F_{k+1} =  F_k \xrightarrow{\eta_i} P_i(F_k)  \xrightarrow{\varepsilon_k^i} F_{k+1}.$$
\end{enumerate} 
Define $Q(F)= \colim_{k \in \lambda} F_k$ and let $\eta: F \to Q(F) $ be the canonical map. 
\begin{prop}\label{glue-adj}
With the above notation, assume that each $\A_i$ is closed under isomorphisms and that the inclusion $\tau_i: \A_i \to \B$ preserves directed colimits. Then the functor $Q$ is a left adjoint to the inclusion $\tau: \A \to \B$ and the map $\eta: \Id \to Q$ is the unit of the adjunction.
\end{prop}
\begin{proof}
We need to establish two things: that $Q(F)$ is indeed in $\A$ and that we have a functorial isomorphism:
$$\Hom_{\A}(Q(F), G) \cong \Hom_{\B}(F, \tau(G)),$$
for any $F \in \B$ and any $G \in \A$.\\

For every $i$ we have by construction, a directed 
diagram formed by the $P_i(F_k)$ ($i$ fixed),  whose structure map is the following composite.
$$ P_i(F_k)  \xrightarrow{\varepsilon_k^i} F_{k+1} \xrightarrow{\eta_i} P_i(F_{k+1}) $$

The two directed diagrams of the $F_k$ and the one of the $P_i(F_k)$ are \emph{crossing each other} and by the crossing lemma (Lemma \ref{cross-lem}) above, they have the same colimit  object which is $Q(F)$.\\

Now in one hand, we have from the definition that for every $i$ and every $k$, $P_i(F_k)$ is an object in $\A_i$. On the other hand since $\A_i$ is \ul{full}, the structure map $ P_i(F_k)  \to  P_i(F_{k+1}) $ is also a map in $\A_i$ so that the directed diagram formed by the $ P_i(F_k)$ lands in $\A_i$. \\
As we assumed that $\A_i$ is closed under directed colimits, then there is a colimit \ul{inside} $\A_i$ of that sequence, which is also the colimit in $\B$ since the inclusion $\tau_i: \A_i \to \B$ preserves directed colimits by hypothesis.

If we assemble all the previous, we get that $Q(F)$ is isomorphic to an object of $\A_i$ (the colimit of $ P_i(F_k)$) and thanks to the fact that $\A_i$ is closed under isomorphisms we deduce that $Q(F) \in \Ob(\A_i)$. The previous being true for all $i$ we have $Q(F) \in \cap_{i \in S} \Ob(\A_i)= \Ob(\A)$.\\

Let $F \in \B$ and $G \in \A$ be two objects and  $h: F \to \tau(G)$ be a map in $\B$. For simplicity, we will drop the $\tau$ and will simply write $\tau(G)= G$. Note that by definition we have  $\tau(G)= \tau_i(G)= G $ for every $i$.\\ 

Thanks to adjunction $P_i \dashv \tau_i$ and since $G \in \A_i$ there is a unique map $ \epsilon_i: P_i(F) \to G$ such that $h : F \to G$ is the following composite.
$$ F \xrightarrow{\eta_i} P_i(F)  \xrightarrow{\epsilon_i} G$$

Since this is true for all $i$, we have from the construction of $F_1$ (which is a pushout), that there is a unique map $h_1: F_1 \to G$ that makes everything compatible. In particular $h$ is the composite:
$$ F \xrightarrow{\delta_0} F_1 \xrightarrow{h_1} G$$

Proceeding like that starting with $h_1$, we get by induction a compatible cocone ending $G$ whose (transfinite) composite is $h$. Therefore by the universal property of the colimit $Q(F)$, there is a unique map  $h_\infty: Q(F) \to G$ such that $h$ is the composite
$$ F \xrightarrow{\eta} Q(F) \xrightarrow{h_\infty} G.$$
\end{proof}
\subsubsection{Gluing monadic projections}
The following definition can be found in \cite[Section 9.2.1]{Simpson_HTHC}.
\begin{df}
Let $\B$ be a category and $\A \subset \B$ be a full subcategory. A \emph{monadic projection} from $\B$ to $\A$ is an endofunctor $P: \B \to \B$ together with a natural transformation $\eta_F: F \to P(F)$ such that: 
\begin{enumerate}
\item $P(F) \in \A$ for all $F \in \B$;
\item for any $F \in \A$ $\eta_F$ is an isomorphism; and
\item for any $F \in \B$, the map $P(\eta_F): P(F) \to P(P(A))$ is an isomorphism. 
\end{enumerate}
\end{df}
Let  $S$ be a $\kappa$ small set and let $N= \{(\A_i, P_i, \eta_i )\}_{i \in S}$ be a family of triples consisting of a full subcategory $\A_i \subset \B$ with a monadic projection $(P_i, \eta_i)$ from $\B$ to $\A_i$. Denote by $\A= \cap_{i \in S} \A_i$ the intersection category defined previously.\\

Apply the same construction as in the previous case to get a directed diagrams of $F_k$.\\
 Define again $Q(F)= \colim_{k \in \lambda} F_k$ and let $\eta: F \to Q(F) $ be the canonical map.

\begin{prop}\label{glue-monad}
With the above notation, assume that each $\A_i$ is closed under isomorphisms, and that each $P_i$ preserves directed colimits. Then the following hold. 
\begin{enumerate}
\item For every $i \in S$ the map $\eta_i: Q(F) \to P_i(Q(F))$ is an isomorphism, so that $Q(F)$ is in $\A= \cap_{i \in S} \A_i$.
\item The pair $(Q, \eta)$ is a monadic projection from $\B$ to $\A$.
\end{enumerate}
\end{prop}

\begin{proof}[Sketch of proof]
By construction we have 
$$Q(F_k) \cong Q(F_{k+1}) \cong  Q(F_{k+2})\cong  \cdots \cong Q(F).$$

Consider the commutative diagram below. The dashed arrows are universal maps between colimits and are induced by universal property of the colimit. The dotted ones are canonical maps going to the colimit. 

\begin{align*}
\xy
(-60,0)+(-12,0)*+{F_k}="X";
(-60,0)+(12,0)*+{F_{k+1}}="A";
(-5,0)+(5,0)*+{Q(F_k)}="C";
{\ar@{.>}^-{}"A";"C"};
{\ar@{>}^-{\delta_k}"X";"A"};
%%% lower square%%%
(-72,-15)*+{P_i(F_k)}="P";
(-15,0)+(-33,-15)*+{P_i(F_{k+1})}="Q";
(20,0)+(0,0)+(-20,-15)*+{P_i(Q(F_k))}="R";
{\ar@{->}^-{}"X";"P"};
{\ar@{->}^-{}"C";"R"};
{\ar@{->}^-{}"A";"Q"};
{\ar@{->}^-{P_i(\delta_k)}"P";"Q"};
{\ar@{.>}^-{}"Q";"R"};
%%%%%% dot dot dot 
(-72,-30)*+{F_{k+1}}="S";
(-15,0)+(-33,-30)*+{F_{k+2}}="T";
(20,0)+(0,0)+(-20,-30)*+{Q(F_k)_1}="V";
{\ar@{->}^-{}"P";"S"};
{\ar@{->}^-{}"Q";"T"};
{\ar@{-->}^-{}"S";"T"};
{\ar@{->}^-{}"R";"V"};
{\ar@{-->}^-{}"T";"V"};
%%%%%%%%%% rajouts diag%%%
(-72,-40)+(0,-5)*+{Q(F_k)}="L";
(-15,0)+(-33,-40)+(0,-5)*+{Q(F_{k+1})}="M";
(20,0)+(0,0)+(-20,-40)+(0,-5)*+{Q(Q(F_k))}="N";
%%%%%%%%%%%%% morphisms fictifs
{\ar@{.>}^-{}"S";"L"};
{\ar@{.>}^-{}"T";"M"};
{\ar@{-->}^-{}"L";"M"};
{\ar@{-->}^-{}"M";"N"};
{\ar@{.>}^-{}"V";"N"};
%%%%%%%% fleches courbées%%%%%%
{\ar@/_1.8pc/@{.>}_{Q(\eta_{F_k})}"L"; "N"};
{\ar@/_1.8pc/@{.>}_{\eta_{F_k}}"X"; "L"};
{\ar@/^1.8pc/@{.>}^{\eta_{F_k}}"X"+(0,1); "C"};
{\ar@/^1.8pc/@{.>}^{\eta_{Q(F_k)}}"C"+(5,-1); "N"};
\endxy
%%%%%%%%%%%%%%%%%%%%% dernier morphisms verticaux
\end{align*}

Denote by $\B^{[1]}= \Ar(\B)$ be the category of morphisms (or arrows) of $\B$. Then we have two directed diagrams   $\{\delta_k: F_k \to F_{k+1} \}$ and $\{P_i(\delta_k): P_i(F_k) \to P_i(F_{k+1}) \}$ in $\B^{[1]}$, which are crossing by construction. They are connected by the maps $(\eta_i, \eta_i): \delta_k \to  P_i(\delta_k)$ and $(\varepsilon_k^i, \varepsilon_{k+1}^i) :P_i(\delta_k) \to \delta_{k+1}$.  Thanks to the crossing lemma (Lemma \ref{cross-lem}) we know that they have the same colimit.\\

The colimit of $\delta_k$ is just $\Id_{Q(F)}$ and since $P_i$ preserves directed colimits by hypothesis, the colimit of $P_i(\delta_k)$ is $ P_i(\Id_{Q(F)})= \Id_{P_i(Q(F))}$. The two being isomorphic, we get that the canonical maps hereafter are inverse each other and the assertion follows.
$$\eta_i: Q(F) \to P_i(Q(F)),$$
$$\varepsilon_{\infty}^i:P_i(Q(F)) \to Q(F).$$

Note that there is a slight abuse of notation, because the universal map $Q(F) \to P_i(Q(F))$ is not a priori the component at $Q(F)$ of the natural transformation $\eta_i$; but it's isomorphic to $\eta_i$ in $\B^{[1]}$ (because $P_i$ preserved directed colimits and in $\B^{[1]}$ colimits are computed level-wise).\\

For Assertion $(2)$ we proceed as follows. First if $F \in \A$, then $F$ is in each $\A_i$ and, therefore each map $\eta_i$ is an isomorphism by hypothesis. It follows that the wide pushout defining $F_1$ in Step $2$ gives an isomorphism $\delta_0: F \to F_1$ so the process stops i.e, $\delta_k$ is an isomorphism for all $k$ so that $\eta: F \to   Q(F)$ is an isomorphism.\\

It remains to show that $Q(\eta)$ is also an isomorphism. To see this take $k=0$ in the above big diagram, so that $\eta_{F_0}$ is just $\eta$; and then observe that we have also a directed diagram of dashed arrows $\{F_{k+1} \dashrightarrow F_{k+2}\}$ which is \ul{different} (a priori) from $\delta_{k+1}$. It's a map induced by universal property of the pushout. 

 Now by construction we have a crossing between that diagram and the diagram of $\delta_k$; thus they have the same colimit which is $\Id_{Q(F)}$.

From this we get that each induced dashed arrow $Q(F) \dashrightarrow Q(F)$ is an isomorphism; consequently the whole horizontal map on the bottom is an isomorphism. This completes the proof of the proposition.  
\end{proof}

\subsubsection{Enforcing a single unity constraint: \emph{Dévissage}} Given a triple $(A,s, \sigma)$ where $A\in X$ and $\sigma: s \to (A,A) \otimes s$ is a $2$-morphism in $\sxop$, we will denote by  
 $$\msxaaub \subset \msxpt,$$ the full subcategory of $\msxpt$ consisting of pointed precategory $\F$ such that the \emph{left invariance diagram} hereafter commutes.
\[
\xy
(0,18)*+{I \otimes \F(s)}="W";
(0,0)*+{\F(A,A) \otimes \F(s) }="X";
(40,0)*+{\F[(A,A)\otimes s]}="Y";
(40,18)*+{\F(s)}="E";
{\ar@{->}^-{\varphi}"X";"Y"};
{\ar@{->}^-{I_A \otimes \Id}"W";"X"};
{\ar@{->}^-{\cong}"W";"E"};
{\ar@{->}^-{\F(\sigma)}"E";"Y"};
\endxy
\]
\\
Our goal is to construct a left adjoint $(P_{s}^l, \eta)$ from $\msxpt$ to $\msxaaub$ (the superscript $l$ on $P^l$ represents \emph{left} as in \emph{left unity}). Similarly we have a full subcategory $\msxaubb$ consisting of the pointed diagrams such that the \emph{right invariance diagram} commutes. 

\begin{warn}
Below we only give the construction for $\msxaaub$, but obviously the same results hold for $\msxaubb$ and we will have also a left adjoint $(P_{s}^r, \eta)$ from $\msxpt$ to $\msxaubb$.
\end{warn}

 We will use the following notation.
\begin{nota}
\begin{enumerate}
\item We will write for short $\kx$ the category 
$$ \prod_{(A,B) \in X^2}\Hom[\sx(A,B)^{op}, \M] \cong \Hom[\coprod_{(A,B)\in X^2}\sx(A,B)^{op}, \M].$$
\item We will write $\Gamma: \kx \to \msx$ the left adjoint of the forgetful functor:
$$\Ub: \msx \to \kx.$$
\item For each pair $(A,B)$ of objects of $\sxop$  let $p_{AB}$ be the projection functor:
$$p_{AB}: \kx \to \Hom[\sx(A,B)^{op}, \M].$$ 
\item $p_{AB}$ has a left adjoint $\delta_{AB}$ which is the `Dirac mass' (see \cite{COSEC1}). Basically $\delta(\F)$ is the product diagrams where all other factors are constant diagram with value the initial object of $\M$ except for the
$(A,B)$-factor which is $\F$.
\end{enumerate}
\end{nota}

\paragraph*{About the left adjoint $\Gamma$}
Recall very briefly that the left adjoint $\Gamma: \kx \to \msx$ is given as follows. If $\F \in \kx$ and $z=(A_0,\cdots,A_n)$ is a $1$-morphism of $\sxop$, then $\Gamma \F$ is given by
$$[\Gamma \F] z= \F(z) \bigsqcup( \coprod_{(s_1,...,s_l) \in \otimes^{-1}(z)} \F(s_1) \otimes \cdots \otimes \F(s_l)).$$ 
Here $\otimes^{-1}(z)$ is the set of all subdivisions of $z$ i.e, all $l$-tuples of composable chain $s_i$ whose composite (concatenation) is $z$ and such that none of the $s_i$ is a chain with a single element ($\degb(s_i) >0$ for all $i$).\\

In particular for $z=(A,B)$, $[\Gamma \F] (A,B)= \F(A,B)$ and similarly on morphisms.  
\begin{rmk}
In each category  $\sx(A,B)^{op}$ the object $(A,B)$ is an initial object and there is no nonidentity morphism whose target is $(A,B)$ (because there is no surjective map $\1 \to \n$ in $\Depi$ for $\n \geq \2$). This property is important for us in the upcoming sections.
\end{rmk}

Recall that the evaluation functor at  $[(A,A)\otimes s]$
 $$\Ev_{[(A,A)\otimes s]}: \Hom[\sx(A,B)^{op}, \M] \to \M,$$ has a left adjoint $\Fb^{[(A,A)\otimes s]}$ that takes an object $m\in \M$ to the diagram $\Fb^{[(A,A)\otimes s]}_m$ defined by the following formula.
$$\Fb^{[(A,A)\otimes s]}_m(z)= m \otimes \Hom[[(A,A)\otimes s],z]:=  \coprod_{h: [(A,A)\otimes s] \to s} m. $$

With the above remark it's easy to see that:

\begin{lem}\label{lem-eval-invar}
For any object $m$ of $\M$ the value of $\Fb^{[(A,A)\otimes s]}_m$ at $(A,B)$ is the initial object $\emptyset$ of $\M$.
\end{lem}
\begin{proof}
 Since there is no morphism $h: [(A,A)\otimes s] \to (A,B)$ in $\sx(A,B)^{op}$, the value at $s=(A,B)$ it's the empty coproduct which is the initial object of $\M$.
\end{proof}

\begin{prop}\label{eval-invar-ab}
Let $\Upsilon_{[(A,A)\otimes s]} : \M \to \msxpt$ be the composite functor below, which is the left adjoint of the evaluation $\Ev_{[(A,A)\otimes s]}$ at the $1$-morphism  $[(A,A)\otimes s]$. 
$$\M  \xrightarrow{\Fb^{[(A,A)\otimes s]}}  \Hom[\sx(A,B)^{op}, \M]  \xrightarrow{\delta_{AB}} \kx \xrightarrow{\Gamma} \msx \xrightarrow{\Rpt} \msxpt.$$

Then for every object $m$ of $\M$ we have:
$$\Upsilon_{[(A,A)\otimes s]}(m)(A,B)= \emptyset \quad \tx{if $A\neq B$,}  $$
$$\Upsilon_{[(A,A)\otimes s]}(m)(A,A)= I. $$
\end{prop} 
\begin{proof}
Indeed the left adjoints $\delta_{AB}$ and $\Gamma$ don't change the value at $(A,B)$ (see \cite{COSEC1} for details). 
By the previous lemma the value of $\Fb^{[(A,A)\otimes s]}_m $ at $(A,B)$ is the initial object and remains the same until we arrive in $\msx$. From there, the functor $\Rpt$ brings the unit $I$ when $A=B$.
\end{proof}

\paragraph{The process} Let $\F$ be an object of $\msxpt$. We are going to construct a directed diagram $\F_k \to \F_{k+1} \to \cdots$.\begin{enumerate}
\item Set $\F_0= \F$.
\item Introduce $m_k$ the object obtained by taking the coequalizer in $\M$ of the two parallel maps below.
$$I \otimes \F_k(s) \xrightarrow{\cong} \F_k(s)\xrightarrow{\F_k(\sigma)} \F_k[(A,A)\otimes s],$$
$$I \otimes \F_k(s) \xrightarrow{I_A \otimes \Id} \F_k(A,A) \otimes \F_k(s) \xrightarrow{\varphi} \F_k[(A,A)\otimes s].$$
Let $j_k: \F_k[(A,A)\otimes s]  \to m_k$ be the canonical map going to the coequalizer object.
\item Using the unit of the adjunction $ \Upsilon_{[(A,A)\otimes s]} \dashv \Ev_{[(A,A)\otimes s]}$, the above diagram gives by adjoint transpose, the following diagram in $\msxpt$.
\[
\xy
%%%%%%%%%%%face arriere%%%%%
(-10,20)*+{\Upsilon_{[(A,A)\otimes s]}(\F_k[(A,A)\otimes s])}="A";
(10,20)+(15,0)*+{\F_k}="B";
(-10,0)*+{\Upsilon_{[(A,A)\otimes s]}(m_k)}="C";
%(10,0)+(15,0)*+{\F_k[(A,A)\otimes s]}="D";
{\ar@{->}^-{}"A";"B"};
{\ar@{->}_-{\Upsilon_{[(A,A)\otimes s]}(j_k)}"A"+(0,-3);"C"};
\endxy
\]
\item Define $\F_{k+1}$ to be the pushout in $\msxpt$ of the previous pushout diagram. We have a canonical map $\delta_k: \F_k \to \F_{k+1}$. The reason we do this is that we want to force an equality between $\F_{k}[(A,A)\otimes s]$ and $m_k$. The directed diagrams of the $(m_k)$ and $ \F_k[(A,A)\otimes s] $ will be crossing and thus have the same colimit in the sense that $j_k$ will induce an isomorphism between the  colimits (Lemma \ref{cross-lem}).
\begin{warn}
It's important to differentiate the \emph{construction pushout} from the object we get since we are doing an induction. So we will write sometimes $push(m_k, \F_k)=\F_{k+1}$
\end{warn}

\item The following diagram  commutes and this is important for us.

\[
\xy
%%%%%%%%%%%face arriere%%%%%
%(-10,20)*+{\F_k(s)}="A";
(10,10)+(3,0)*+{m_k}="B";
(-10,0)*+{\F_k[(A,A)\otimes s]}="C";
(10,0)+(20,0)*+{\F_{k+1}[(A,A)\otimes s]}="D";
%{\ar@{->}^-{\alpha_k}"A";"B"};
%{\ar@{->}_-{\F_k(\sigma)}"A"+(0,-3);"C"};
{\ar@{->}^-{\xi_k}"B";"D"};
{\ar@{->}_-{\delta_k}"C";"D"};
{\ar@{->}_-{j_k}"C";"B"};
\endxy
\]
\item Define $P_{s}^l(\F)= \colim \F_k$ and let $\eta: \F \to P_{s}^l(\F)$ be the canonical map.
\end{enumerate}

\begin{lem}\label{lem-mon-proj-left}
With the previous notation, the following hold.
\begin{enumerate}
\item For every $k \in \lambda$ the component $\delta_k: \F_k(A,B)\to \F_{k+1}(A,B)$ of $\delta_k$ at the $1$-morphism $(A,B)$ is an isomorphism in $\M$.
\item The component $\eta_{(A,B)}: \F(A,B) \to P_{s}^l(\F)(A,B)$ is an isomorphism.
\item For every $\F$, $P_{s}^l(\F)$ is in $\msxaaub$ and the functor 
$$P_{s}^l: \msxpt \to \msxaaub,$$ is a left adjoint of the forgetful functor. 
\item The pair $(P_{s}^l, \eta)$ is a monadic projection from $\msxpt$ to $\msxaaub$. Moreover $P_{s}^l$ preserves directed colimits.
\end{enumerate}
\end{lem}
%%%%%%%%%%%%%
\begin{proof}
See Appendix \ref{ap-lem-mon-proj-left}.
\end{proof}
%%%%%%%%%%

\subsection{The global unitalization}
As mentioned previously we have the same results for $\msxaubb$ i.e, a left adjoint $(P_{s}^r, \eta)$ which is a monadic projection and having the same property as $P_{s}^l$. In particular the component $\eta:  \F(A,B) \to P_{s}^r(\F)(A,B)$ is an isomorphism in $\M$. To make the discussion clear let's fix some notation.
\begin{nota}
\begin{enumerate}
\item Denote $S_{AB}^l$ the following set of triples $$ \{(A,s,\sigma); s \in \sx(A,B)^{op}; \sigma:s \to (A,A)\otimes s \}$$ 
\item Similarly let $S_{AB}^r=\{(s,B,\sigma); s \in \sx(A,B)^{op}; \sigma:s \to s \otimes (B,B) \}$.
\item $S= \coprod_{(A,B)\in X^2} S_{AB}^l \cup S_{AB}^r.$
\item If $i= (A,s,\sigma) \in S$, define $\A_i= \msxaaub$.
\item If $i=(s,B,\sigma) \in S$, define $\A_i= \msxaubb$
\end{enumerate}
\end{nota}

Denote by $(P_i, \eta_i)$ the corresponding left adjoint which is also a monadic projection.\\
\\
Clearly one has an equivalence of categories between $\A= \cap_{i \in S} \A_i$ and $\msxsu$.\\
\\
Denote by $(\Phi,\eta)$ the pair $(Q,\eta)$ constructed in the previous subsection; its the \emph{gluing} of the $(P_i,\eta_i)$. Then it's not hard to see that:
\begin{thm}\label{thm-unitalization-preservation}
With the previous notation, the following hold.
\begin{enumerate}
\item For every $\F$, $\Phi(\F)$ is in $\msxsu$ and $(\Phi, \eta)$ is a left adjoint of the forgetful functor $\Ub: \msxsu \to \msxpt$; and $\eta$ is the unity of the adjunction.
\item The component $\eta_{(A,B)}: \F(A,B) \to \Phi(\F)(A,B)$ is an isomorphism.
\item The pair $(\Phi, \eta)$ is a monadic projection from $\msxpt$ to $\msxsu$.
\end{enumerate}
\end{thm}
\begin{proof}
The three assertions follow by combining Proposition \ref{glue-adj} and Proposition \ref{glue-monad} together with the fact that  in $\msxsu$, colimits at each $1$-morphism $(A,B)$ are taken level-wise (Proposition \ref{lem-creat-colimit-trois}).\\
\end{proof}

%%%%%%%%%%%%%%%%%%%%%%%%%%
\section{The easy model structure}\label{sec-easy-model}
 We provided in \cite{COSEC1} two model structures on $\msx$, denoted $\msxinj$ and $\msxproj$, in which the weak equivalences are the level-wise ones. These model structures lack of \emph{left properness}, a property that is needed to guarantee the existence of a left Bousfield localization.
\begin{note}
We assume that all our categories are $\kappa$-small for a sufficiently large regular cardinal $\kappa$. 
\end{note}
 
\begin{df}
Let $\M$ be a model category and $\J$ be a small category with an initial object $e$. A natural transformation $\eta: \F \to \G$ in $\Hom(\J,\M)$ is called an \emph{easy weak equivalence} if the component $\eta_e: \F(e) \to \G(e)$ at the initial object is a weak equivalence in $\M$.
\end{df}

Clearly one has that:
\begin{prop}
\begin{enumerate}
\item Every level-wise weak equivalence is an easy weak equivalence.
\item If $\F$ and $\G$ take their value in the subcategory of weak equivalences, then $\eta: \F \to \G$ is a level-wise weak equivalence if and only if it is an easy weak equivalence. 
\end{enumerate}
\end{prop}
We shall call any diagram indexed by a category with an initial object \emph{conical diagram}. Recall that for any object $i \in \J$, the evaluation functor at $i$ has a left adjoint:
$$\Ev_i: \Hom(\J,\M) \rightleftarrows \M: \Fb^{i},$$
described as follows.
For $U \in \M$ we have $$ \Fb_{U}^{i}(j)= U \otimes \Hom(i, j)= \coprod_{i \to j} U$$
and similarly on morphism (see \cite{Hirsch-model-loc}).  
The following proposition is the general form of Lemma \ref{lem-eval-invar}.
\begin{prop}\label{cof-are-easy-weak}
Let $\J$ be a category with an initial object $e$ such that there is no morphism whose target is $e$ except $\Id_e$.
Then for every $i \neq e$ and every morphism $f: U\to V$ in $\M$, the natural transformation $\Fb_{f}^{i}: \Fb_{U}^{i} \to \Fb_{V}^{i}$ is an easy weak equivalence.  
\end{prop}  

\begin{rmk}\label{rmk-cofib}
Note that if $\J$ is a \ul{direct} category an initial object then necessarily $\J \downarrow e$ is the one point category. This last situation is precisely what we're looking at.
\end{rmk}

\begin{proof}
Indeed $\Fb_{U}^{i}(e)=\Fb_{V}^{i}(e)=\emptyset$ is the initial object of $\M$, the component of $f$ at $e$ is the unique endomorphism of the initial object: the identity which is a weak equivalence.
\end{proof}
\subsection{Easy projective model structure}
Recall that for a cofibrantly generated model category $\M$, there is a \emph{projective model structure} on $\Hom(\J, \M)$ in which the fibrations and weak equivalences are object wise. This is a cofibrantly generated model category and we can explicitly say what are the generating set cofibrations and trivial cofibrations. \\

Indeed if $\I$ and $\Ja$ are respectively the set of generating cofibrations and trivial cofibrations for $\M$, then the set 
$$\I_{proj}=\coprod_{i\in \J}\{ \Fb_{f}^{i}; f \in \I \} $$ 
$$\Ja_{proj}=\coprod_{i\in \J}\{ \Fb_{f}^{i}; f \in \Ja \} $$ 
are the ones for $\Hom(\J, \M)_{proj}$. 

\begin{nota}
We will denote by: 
\begin{enumerate}
\item $\We=$ the class of easy weak equivalences;
\item $\Ie= \{ \Fb_{f}^{e}; f \in \I \}$;
\item $\Je= \{ \Fb_{f}^{e}; f \in \Ja \} $. 
\end{enumerate}
\end{nota}

For simplicity we assume $\M$ to be locally presentable i.e, a combinatorial model category. We derive (or localize) that projective model structure to get a new one as follows.
\begin{thm}\label{thm-local-proj}
Let $\M$ be a combinatorial model category and $\J$ be a category 
with an initial object $e$ such that the comma category $\J \downarrow e$ is reduced to $\Id_e$. Then there is a combinatorial model structure on $\Hom(\J, \M)$, denoted $\Hom(\J, \M)_{easy}$ where:
\begin{enumerate}
\item $\We$ is the class of weak equivalences;
\item $\Ie$ is the set of generating cofibrations;
\item $\Je$ is the set of generating trivial cofibrations;
\end{enumerate}
The identity functor $\Id: \Hom(\J, \M)_{easy} \to \Hom(\J, \M)_{proj}$ is a left Quillen functor.
\end{thm}

To prove the theorem we will use the recognition theorem that we recall hereafter as stated in \cite{Hov-model}. 

\begin{thm}\label{thm-recognition}
Suppose $\C$ is a category with all small colimits and limits. Suppose $\W$ is a subcategory of $\C$, and $I$ and $J$ are sets of maps of $\C$. Then there is a cofibrantly generated model structure on $\C$ with $I$ as the set of generating cofibrations, $J$ as the set of generating trivial cofibrations, and $\W$ as the subcategory of weak equivalences if and only if the following conditions are satisfied.
\begin{enumerate}
\item  The subcategory $\W$ has the two out of three property and is closed under retracts.
\item  The domains of I are small relative to I-cell.
\item  The domains of J are small relative to J-cell.
\item  $J$-cell $\subseteq \W \cap I$-cof.
\item  $I$-inj $\subseteq \W \cap J$-inj.
\item  Either $\W \cap I$-cof $\subseteq J$-cof or $\W \cap J$-inj $\subseteq I$-inj.
\end{enumerate}
\end{thm}

\begin{proof}[Proof of Theorem \ref{thm-local-proj}]
The first three properties are fulfilled as the reader can check.\\

Property $(4)$ is also easy to check. Indeed maps in $\Je$-cell are old projective trivial cofibrations: this gives the inclusion  $\Je$-cell 
$\subseteq I$-cof $\cap \We$. \\

Property $(5)$ is also easy, since any map $\sigma \in \Ie$-inj , is such that $\sigma_e$ is a trivial fibration and the  inclusion $\Ie$-inj $\subseteq \We \cap \Je$-inj follows.\\

Property $(6)$ is clear. Indeed if $\alpha \in \We \cap \Je$-inj, then $\sigma \in \Ja_{proj}$-inj, which means that $\alpha_e$ is a fibration. And since $\alpha \in  \We$, the component $\sigma_e$ is a weak equivalence, thus a trivial fibration which means precisely that $\sigma$ is in $\I_{easy}$-inj.
\end{proof}
\begin{rmk}
Note that by definition a fibrant object in $\Hom(\J, \M)_{easy}$ is just a diagram $\F$ such that the object $\F(e)$ is fibrant in $\M$.
We can increase the set of generating trivial cofibration by simply taking $\Ja_{proj}$. This will give the usual fibrant objects which are object-wise fibrant diagrams.
\end{rmk}
A direct consequence of the previous theorem is:
\begin{cor}\label{model-kx-easy}
For any pair $(A,B)$ of objects of $\sxop$, there is a combinatorial model structure $\Hom[\sx(A,B)^{op}, \M]_{easy}$ on $\Hom[\sx(A,B)^{op}, \M]$, in which the weak equivalences are the easy weak equivalences and the cofibrations are the easy projective cofibrations.\\

The fibrations are the maps $\sigma$ such that the component $\sigma_{(A,B)}$ is a fibration.
 The generating set of trivial cofibrations is $\Je$ and the generating set of cofibration is $\Ie$.\\

The identity functor $\Id: \Hom[\sx(A,B)^{op}, \M]_{easy}\to  \Hom[\sx(A,B)^{op}, \M]_{proj}$ is a left Quillen functor.
\end{cor}

We have another corollary obtained when $\J=[1]=\{0 \to 1 \}$ is the \emph{walking-morphism category}. Note that by definition $\M^{[1]}= \Hom[\J, \M]$ and that the left adjoint $\Fb^{0}: \M \to \M^{[1]}$ is the natural embedding that takes an object $m$ to $\Id_{m}$; it takes a morphism $f: m \to m'$ to the morphism $[f]: \Id_m \to \Id_{m'}$ whose components are both $f$.  
\begin{cor}
There is combinatorial model structure on the category $\M^{[1]}$ of morphisms of $\M$ where a weak equivalence (resp. fibration) is a map $\sigma: \F \to \G$ such that the component $\sigma_0: \F(0) \to \G(0)$ is a weak equivalence (resp. fibration).\\

The set of generating cofibrations is 
$$\{[\alpha]: \Id_U \to \Id_V \}_{\alpha: U \to V \in \I}.$$
The set of generating trivial cofibrations is 
$$\{[\alpha]: \Id_U \to \Id_V \}_{\alpha: U \to V \in \Ja}.$$

We will denote by $\armeasy$ this model structure.
\end{cor}
\begin{prop}\label{prop-av}
Let $s=(A,...,B)$ be an object of $\sx(A,B)^{op}$, and let 
$u_s: (A,B) \to s$ be the unique morphism therein. Then the following hold.
\begin{enumerate}
\item The evaluation $\Ev_{u_s}:\Hom[\sx(A,B)^{op}] \to \M^{[1]}$ has a left adjoint
$$\Psi_s: \M^{[1]} \to \Hom[\sx(A,B)^{op}].$$ 
\item This adjunction is moreover a Quillen adjunction between the respective easy model structures.
\item If $\sigma: \F \to \G$ is a morphism in $\M^{[1]}$ such that the component $\sigma_0: \F(0) \to \G(0)$ is an isomorphism (resp. cofibration), then the component $\Psi_s(\sigma)_{(A,B)}: \Psi_s(\F)(A,B) \to \Psi_s(\G)(A,B)$ is also an isomorphism (resp cofibration).
\item For any generating cofibration $\alpha: U \to V$, the canonical map $\xi_\alpha=(\alpha,i_0): \alpha \to i_\alpha$, obtained by the pushout of $\alpha$ along itself, is a cofibration in $\armeasy$.
\[
\xy
(0,18)*+{U}="W";
(0,0)*+{V}="X";
(30,0)*+{V\cup^{U}V }="Y";
(30,18)*+{V}="E";
{\ar@{.>}^-{i_\alpha= i_1}"X";"Y"};
{\ar@{->}^-{\alpha}"W";"X"};
{\ar@{->}^-{\alpha}"W";"E"};
{\ar@{.>}^-{i_0}"E";"Y"};
\endxy
\]
\item The image $\Psi_s(\xi_\alpha)$ is a cofibration in $\Hom[\sx(A,B)^{op}]_{easy}$
\end{enumerate}
\end{prop}
\begin{proof}
$\Psi_s$ is the left Kan extension along the functor $[1] \xrightarrow{u_s} \sx(A,B)^{op}$ that picks $u_s$; this gives Assertion $(1)$. 

The evaluation sends (trivial) fibrations to (trivial) fibrations. In fact it's easily seen that  $\Psi_s$ sends generating (trivial) cofibrations to generating (trivial) cofibrations; we get Assertion $(2)$.

Assertion $(3)$ follows from the formula of the left Kan extension as the reader can check. In fact
$\Psi_s(\F)(A,B) \to \Psi_s(\G)(A,B)$ is isomorphic  to $\sigma_0: \F(0) \to \G(0)$ as objects in $\M^{[1]}$.\\

For Assertion $(4)$ it suffices to observe that we obtained $\xialpha$ as a pushout of the generating cofibration $[\alpha]: \Id_U \to \Id_V$. We illustrate the pushout data in the diagram hereafter.
\[
\xy
%%%%%%%%%%%face arriere%%%%%
(-10,20)*+{U}="A";
(10,20)+(10,0)*+{U}="B";
(-10,0)*+{U}="C";
(10,0)+(10,0)*+{V}="D";
{\ar@{->}^-{\Id}"A";"B"};
{\ar@{->}_-{\Id}"A"+(0,-3);"C"};
{\ar@{->}_-{\alpha}"B";"D"};
{\ar@{->}_-{\alpha}"C";"D"};
%%%%%%%%%%%%%%%%%%%%face avant%%%%
(-10,20)+(-25,-15)*+{V}="E";
(-10,0)+(-25,-15)*+{V}="G";
{\ar@{->}_-{\Id}"E"+(0,-3);"G"};
%%%%%%%%%%% fleches sortantes arrieres vers avant %%%%%%%
{\ar@{->}_-{\alpha}"A";"E"};
{\ar@{->}_-{\alpha}"C";"G"};
(-10,20)+(-25,-15)+(30,0)*+{V}="J";
(-10,0)+(-25,-15)+(30,0)*+{V\cup^{U}V }="K";
{\ar@{.>}^-{\alpha}"B";"J"};
{\ar@{.>}^-{i_0}"D";"K"};
{\ar@{=}^-{}"E";"J"};
{\ar@{.>}_-{i_\alpha}"G";"K"};
{\ar@{.>}^-{i_\alpha}"J";"K"};
\endxy
\]
Assertion $(5)$ is obvious since left Quillen functors transport cofibrations. 
\end{proof}

\begin{nota}
\begin{enumerate}
\item We have two product model structures $\kxproj$ and $\kxeasy$ and the identity $$\Id: \kxproj \to \kxeasy$$ is a left Quillen functor.
\item We will use the same notation and write $\I_{AB}$ (resp. $\Ja_{AB}$) for set of generating cofibrations (resp. trivial cofibrations) for either $\Hom[\sx(A,B)^{op}, \M]_{proj}$ or  $ \Hom[\sx(A,B)^{op}, \M]_{easy}$.
\end{enumerate}
\end{nota}

By lifting properties and adjunction one can clearly have:
\begin{lem}\label{lem-generation}
The sets 
$$ \coprod_{(A,B)} \{ \delta_{AB}(\alpha); \alpha \in \I_{AB} \}$$
$$ \coprod_{(A,B)} \{ \delta_{AB}(\alpha); \alpha \in \Ja_{AB} \} $$
constitute a set of generating cofibrations (resp. trivial cofibrations) of $\kx$.
\end{lem}
The left adjoint $\Gamma$ has a nice property as outlined in the following
\begin{lem}\label{Gamma-no-change-easy}
Let $\alpha: \F \to \G$ be a morphism in $\kx$ and $\Gamma \alpha: \Gamma(\F) \to \Gamma(\G) $ be the image  in $\msx$. Then the following hold.
\begin{enumerate}
\item If $\alpha$ is an easy weak equivalence i.e, $\alpha_{(A,B)}: \F(A,B) \to \G(A,B)$ is a weak equivalence in $\M$,  then so is $\Gamma \alpha$
\item If the component $\alpha_{(A,B)}: \F(A,B) \to \G(A,B)$ is an isomorphism, then the component $[\Gamma \alpha]_{(A,B)}$ is also an isomorphism.
\end{enumerate} 
\end{lem}
\begin{proof}
Recall again that for any $\F \in \prod_{A,B} \Hom[\sx(A,B)^{\tx{op}}, \ul{M}]$,  the lax functor $\Gamma \F$ is given by the following formula at a generic $1$-morphism $z$ of $\sxop$ (see \cite{COSEC1}).
$$\Gamma\F (z)= \F(z) \sqcup (\coprod_{(s_1,..., s_l);   \otimes(s_i)=z ; s_i\neq z} \F(s_1) \otimes \cdots \otimes \F(s_l)).$$

Similarly on morphisms one has
$$\Gamma (\alpha)_{z} =\alpha_z \sqcup (\coprod_{(s_1,..., s_l);   \otimes(s_i)=z ; s_i\neq z} \alpha_{s_1} \otimes \cdots \otimes \alpha_{s_l}).$$

Now it suffices to observe that $(A,B)$ is \emph{indecomposable} in $\sxop$ i.e, there is no $l$-tuple $(s_1,...s_l)$ such that:
\begin{itemize}
\item $l >1$,
\item   $s_i \neq (A,B)$  $\forall i$ and
\item  $\otimes(s_i)=(A,B)$.
\end{itemize}

Therefore one has:
$ \Gamma (\F)(A,B)= \F(A,B) $
and  $ \Gamma (\alpha)_{(A,B)}= \alpha_{(A,B)}$
and the lemma follows.
\end{proof}

\begin{warn}
We are going to give a lemma that holds for the three left adjoints $$\Gamma: \kx \to \msx,$$ 
$$\Rpt \Gamma : \kx \to \msxpt,$$
$$\Phi\Rpt\Gamma: \kx \to \msxsu.$$  

Therefore we will use a generic notation $\Qs$ for the three functors for simplicity. Similarly we will write $\Ub$ for the various forgetful functor which are the corresponding right adjoint.
\end{warn}
\begin{lem}\label{lem-q-no-change}
Let $\Qs$ be the functor $\Gamma$ (resp. $\Rpt \Gamma$, resp. $\Phi\Rpt\Gamma$).\\
Let $\alpha: \F \to \G$ be a morphism in $\kx$ and $\Qs \alpha: \Qs(\F) \to \Qs(\G) $ be the image by $\Qs$.\\

Then if the component $\alpha_{AB}: \F(A,B) \to \G(A,B)$ is a (trivial) cofibration (resp. an isomorphism) in $\M$ then so is $\Ub(\Qs \alpha)_{(A,B)}$ respectively.  

\end{lem}

\begin{proof}
From the proof of the previous lemma and thanks to Theorem \ref{thm-unitalization-preservation} one has the following.
\begin{enumerate}
\item For $\Qs= \Gamma$, $\Ub(\Qs \alpha)_{(A,B)}= \alpha_{(A,B)}$, and the statement is obvious.
\item For $\Qs= \Rpt\Gamma$, we know from the construction of $\Rpt$ that 
$$\Ub(\Qs \alpha)_{(A,B)}= \alpha_{(A,B)} \hspace*{0.2in} \tx{if $A\neq B$};$$
$$\Ub(\Qs \alpha)_{(A,B)}= \alpha_{(A,A)} \sqcup \Id_I \hspace*{0.2in} \tx{if $A= B$}.$$
Since $\Id_I$ is a trivial cofibration and an isomorphism and thanks to the fact that  (trivial) cofibrations (resp. isomorphisms) are closed under coproduct, we see that in both cases, $\Ub(\Qs \alpha)_{(A,B)}$ is a cofibration (resp. isomorphism) which is trivial if $\alpha_{(A,B)}$ is;   and the statement holds also.
\item Finally if $\Qs=\Phi \Rpt \Gamma$, we now that  the unit of the adjunction $$\Phi: \msxpt \leftrightarrows \msxsu: \Ub, $$ gives an an isomorphism in the arrow category $\M^{[1]}$,
$$\Rpt \Gamma(\alpha)_{(A,B)} \cong   \Phi [\Rpt \Gamma(\alpha)]_{(A,B)},$$
and we conclude by the previous cases.
\end{enumerate}
\end{proof}

\begin{prop}\label{easy-pushout}
Let $\Qs$ be the functor $\Gamma$ (resp. $\Rpt \Gamma$, resp. $\Phi\Rpt\Gamma$).
Given a pushout square in $\msx$ (resp. $\msxpt$, resp. $\msxsu$)
\[
\xy
(0,18)*+{\Qs \F}="W";
(0,0)*+{\Qs \G}="X";
(30,0)*+{\Qs\G \cup^{\Qs \F}\Ba}="Y";
(30,18)*+{\Ba}="E";
{\ar@{->}^-{\ol{j}}"X";"Y"};
{\ar@{->}^-{\Qs\alpha}"W";"X"};
{\ar@{->}^-{j}"W";"E"};
{\ar@{->}^-{\ol{\Qs\alpha}}"E";"Y"};
\endxy
\]
where $\alpha : \F \to \G$ is a morphism in $\kx$ the following hold. 
\begin{enumerate}
\item If the component $\alpha_{(A,B)}: \F(A,B) \to \G(A,B)$ is a trivial cofibration in $\M$ then so is the component
$$ \ol{\Qs\alpha}_{(A,B)}: \Ba(A,B) \to [\Qs\G \cup^{\Qs \F}\Ba] (A,B).$$
\item If $\M$ is left proper  and the component $\alpha_{(A,B)}: \F(A,B) \to \G(A,B)$ is a cofibration, and if $j_{(A,B)} :(\Qs\F)(A,B) \to \Ba(A,B) $ is a weak equivalence then 
$$ \ol{j}_{(A,B)} :(\Qs\G)(A,B) \to [\Qs\G \cup^{\Qs \F}\Ba] (A,B) $$ 
is a weak equivalence. 

\item If $\F$, $\G$ and $\Ba$ are objects of $\msx$ (resp. $\msxpt$, $\msxsu$) and if $\sigma: \F \to \G$ is a morphism of precategories (resp. pointed , unital) such that the component $$\sigma_{(A,B)}: \F(A,B) \to \G(A,B),$$ is an isomorphism, then the component 
$$\ol{\sigma}_{(A,B)}: \Ba(A,B) \to [\G \cup^{\F}\Ba] (A,B)$$
in the pushout square below, is an isomorphism; in particular $\ol{\sigma}$ is an easy weak equivalence.
\[
\xy
(0,18)*+{\F}="W";
(0,0)*+{\G}="X";
(30,0)*+{\G \cup^{\F}\Ba}="Y";
(30,18)*+{\Ba}="E";
{\ar@{->}^-{\ol{j}}"X";"Y"};
{\ar@{->}^-{\sigma}"W";"X"};
{\ar@{->}^-{j}"W";"E"};
{\ar@{->}^-{\ol{\sigma}}"E";"Y"};
\endxy
\]
\end{enumerate} 
\end{prop}
\begin{proof}
This is a direct consequence of Assertion $(2)$ of  Proposition \ref{lem-creat-colimit-trois}, which says that colimits in $\msxsu$ are computed level-wise at the $1$-morphism $(A,B)$. 
\end{proof}

\begin{rmk}
To have an intuition of why the proposition holds, it suffices to observe that  all $1$-morphisms of the form $(A,B)$ are initial and there is no morphism in $\sx$ whose target is $(A,B)$ except the identity; and as outlined earlier in the proof of Lemma \ref{Gamma-no-change-easy}, $(A,B)$ is indecomposable. It follows that 
for a lax diagram $\F$, the object $\F(A,B)$ doesn't receive (non trivial) laxity maps; which means that $\F(A,B)$ is not subject to algebraic constraints. Therefore colimits at the level $(A,B)$ are computed level-wise.
\end{rmk}

\subsection{The model structure for a fixed set of objects}
Let $\kxeasy$ be the model category obtained from Corollary \ref{model-kx-easy}. Denote by $\I_{\kxeasy}$ and $\Ja_{\kxeasy}$ the respective generating sets of cofibrations and trivial fibrations therein.  Let $\Qs$ be the functor $\Gamma$ (resp. $\Rpt \Gamma$, resp. $\Phi\Rpt\Gamma$). 
\begin{thm}\label{easy-model-msx}
There is a combinatorial model structure on $\msx$ (resp. $\msxpt$, resp. $\msxsu$)  in which:
\begin{itemize}[label=$-$]
\item the weak equivalences are the easy weak equivalences; 
\item the fibrations are the maps $\sigma$ such that the component  $\sigma_{(A,B)}$ is a fibration;
\item the trivial fibrations are the maps $\sigma$ such that the component $\sigma_{(A,B)}$ is a trivial fibration. 
\end{itemize}
This model structure if furthermore left proper if $\M$ is.\\

The sets $\Qs\I_{\kxeasy}$ and $\Qs \Ja_{\kxeasy}$ are respectively the sets of generating cofibrations and trivial cofibrations.\\

We will denote by $\msxeasy$ (resp. $\msxpteasy$, resp. $\msxsueasy$) this model structure. The respective monadic adjunctions 
$$\Qs: \kxeasy \rightleftarrows  \msxeasy: \Ub,$$
$$ \Qs: \msxeasy \rightleftarrows  \msxpteasy: \Ub,$$
$$ \Qs: \msxpteasy \rightleftarrows  \msxsueasy: \Ub,$$
are Quillen adjunctions where $\Qs$ is left Quillen and $\Ub$ is right Quillen.
\end{thm}
\begin{proof}
Thanks to Proposition \ref{easy-pushout}, we know that the pushout of a generating trivial cofibration is an easy weak equivalence and this the key condition for the transfer lemma of Schwede-Shipley \cite{Sch-Sh-Algebra-module} with the respect to the monadic adjunction $ \Qs \dashv \Ub$.

The left properness is given by the second assertion of Proposition \ref{easy-pushout}. \end{proof}
\begin{note}
One can avoid the use of the transfer lemma of \cite{Sch-Sh-Algebra-module} and do everything directly by Theorem \ref{thm-recognition}. With both methods, the proof boils down to check that the pushout along a generating trivial cofibration is a weak equivalence. 
\end{note}
\subsubsection{Hierarchy of precategories}
We take a moment to outline an important class of precategories that will be needed later. If $A,B$ are objects of $\sx$, we will denote by $\sx(A,B)_{\geq 2} \subset \sx(A,B)$ the full  subcategory of $1$-morphisms of degree $\geq 2$. This simply means that we \ul{remove} the $1$-morphism $(A,B)$ which is the co-initial object in $\sx(A,B)$. Similarly we have the opposite category $\sx(A,B)^{op}_{\geq 2}$.\\

As both $\sx(A,B)$ and $\sx(A,B)^{op}$ are Reedy $1$-categories, it's not hard to see that 
$\sx(A,B)$ is isomorphic to the \emph{latching category} of $\sx(A,B)$ at $(A,B)$ and dually $\sx(A,B)^{op}_{\geq 2}$ is isomorphic to the \emph{matching category} of $\sx(A,B)^{op}$ at $(A,B)$.
\begin{df}\label{def-deux-const}
Say that a precategory $\F: \sxop \to \M$ is \ul{$2$-constant} if for every pair $(A,B)$ of objects of $\sx$, the restriction to $\sx(A,B)^{op}_{\geq 2}$  of the component 
$$\F_{AB}: \sx(A,B)^{op} \to \ul{M},$$
is a constant functor.
\end{df}
\begin{ex}\label{ex-deux-constant}
The $2$-constant precategories are the most natural type of co-Segal categories. They appear, for example, when we do homotopy transfer. In fact given any classical $\M$-category $\C$, if we choose (randomly) a weak equivalence $\tld{\C}(A,B) \xrightarrow[f]{\sim}  \C(A,B)$ for each pair $(A,B)$ of objects, e.g a cofibrant replacement functor;  then there is a $2$-constant co-Segal category $\tld{\C}$ which is canonically weakly equivalent to $\C$. We describe very briefly $\tld{\C}$ as follows.\\

The component $\tld{\C}_{AB}:\sx(A,B)^{op} \to \ul{M}$, takes $(A,B)$ to $\tld{\C}(A,B)$; and is constant of value $\C(A,B)$ on 
$\sx(A,B)^{op}_{\geq 2}$. The unique structure map $(A,B) \to (A,...,B)$ is sends to the chosen weak equivalence $\tld{\C}(A,B) \xrightarrow{\sim} \C(A,B).$\\
The laxity map correspond to either one of the following composites.
$$ \C(A,B) \otimes \C(B,C) \to \C(A,C) \hspace*{0.2in} \tx{the composition in $\C$};$$
$$\tld{\C}(A,B) \otimes \tld{\C}(B,C) \xrightarrow{f\otimes f} \C(A,B) \otimes \C(B,C) \to \C(A,C);$$
$$\C(A,B) \otimes \tld{\C}(B,C) \xrightarrow{\Id \otimes f} \C(A,B) \otimes \C(B,C) \to \C(A,C);$$
$$\tld{\C}(A,B) \otimes \C(B,C) \xrightarrow{f\otimes  \Id} \C(A,B) \otimes \C(B,C) \to \C(A,C).$$
The coherence axiom follows from the fact that the composition in $\C$ is associative.  
\end{ex}
The above example has a unital version which we give as a lemma.

\begin{lem}\label{lem-deux-constant}
Let $\C$ be a strict $\M$-category in the usual sense. Assume that for every $(A,B)$ we have a morphism $f: \tld{\C}A,B) \to \C(A,B)$ and that if $A=B$, the unity  $I_A: I \to \C(A,A)$ factorizes through $f: \tld{\C}(A,A) \to \C(A,A)$ as $I_A= f\circ I'_A$, for some map $I'_A: I \to \tld{\C}(A,A).$\\

Then the precategory $\tld{\C}$ constructed previously is a unital precategory. 
\end{lem}

\begin{proof}
Using the fact that we have a category $\C$, one checks easily that the following composite is the same as the natural isomorphism $ I \otimes \C(A,B) \xrightarrow[\cong]{l} \C(A,B).$
$$I \otimes \C(A,B) \xrightarrow{I_A'\otimes \Id}\tld{\C}(A,A)\otimes \C(A,B) \xrightarrow{f \otimes \Id}\C(A,A) \otimes \C(A,B) \to \C(A,B).$$

In fact by the bifunctoriality of $\otimes$, the previous composite map is just the same as the following one.
$$I \otimes \C(A,B) \xrightarrow{I_A\otimes \Id} \C(A,A) \otimes \C(A,B) \to \C(A,B).$$

This gives the left invariance diagram for the unity, the right invariance is treated in the same way.
\end{proof}

In the beginning of the paper we mentioned the inclusion $\mcatx \hookrightarrow \msx$ that has a left adjoint $ |-|: \msx \to \mcatx$. Below we outline simply that this adjunction restrict to the inclusion $\mcatx \hookrightarrow \msxsu$.

%%%%%%
\begin{lem}
The inclusion functor $\mcatx \hookrightarrow \msxsu$ has a left adjoint denoted again $|-|: \msxsu \to \mcatx.$
\end{lem}
\begin{proof}
We can use the adjoint functor theorem for locally presentable categories since the inclusion $\mcatx \hookrightarrow \msxsu$ preserves directed colimits and limits. Indeed, in both categories, limits and directed colimits are computed level-wise. \\

Just like in the old adjunction, there is a direct proof that doesn't even requires $\M$ to be locally presentable, but monoidal closed. And the formula is the same i.e, given a pair $(A,B)$ of elements of $X$ define the $\M$-category with hom-object
$$|\F|(A,B)= \colim \F_{AB}.$$

Using exactly the same method as in the proof of Proposition \ref{limit-msxsu} one proves the following claim. 
\begin{claim}
The two maps hereafter are equal.
$$I \otimes |\F|(A,B) \to |\F|(A,A) \otimes |\F|(A,B) \to |\F|(A,B),$$
$$I \otimes |\F|(A,B) \xrightarrow{\cong}|\F|(A,B).$$
\end{claim}
Recall that the argument is the uniqueness of the universal map out of the colimit. More precisely, one uses this argument with respect to the two cocones starting  $ I \otimes \F(s)$ and ending at $\colim \F =|\F |(A,B)$ whose respective component is the following map (while $s$ runs through $\sx(A,B)^{op}$).  

$$I \otimes \F(s) \xrightarrow{\cong} \F(s) \xrightarrow{can} |\F |(A,B); $$
$$I \otimes \F(s) \xrightarrow{\Id_I \otimes can} I \otimes |\F|(A,B) \xrightarrow{I_A \otimes \Id} |\F |(A,A) \otimes  |\F |(A,B) \xrightarrow{|\varphi|} |\F |(A,B).$$

Using the commutativity of the diagrams involving the unity $I_A$ one finds that these two components are the same i.e, we have the same cocone, thus there is a unique map out of the colimit of $I \otimes \F(s)$ that gives the obvious factorizations. But $\colim I \otimes \F(s) \cong  I \otimes |\F|(A,B)$ and we get the left invariance. Proceeding in the same way we get the right invariance. 
\end{proof}
\begin{cor}
If a local model structure on $\mcatx$ exists then the adjunction
$$|-|: \msxsueasy \leftrightarrows \mcatx: \iota,$$
is a Quillen adjunction. 
\end{cor}

\begin{proof}
Indeed a local (trivial) fibration in $\mcatx$ is also a (trivial) fibration in $\msxsueasy$.
\end{proof}

\begin{rmk}\label{rmk-associated-deux-constant}
Let $\eta: \F \to |\F|$ be the unit of the adjunction. In particular  we have a map $\eta_{(A,B)}: \F(A,B) \to |\F|(A,B)$ which is just the canonical map going to the colimit. Just like in Example \ref{ex-deux-constant} we can build a $2$-constant precategory $\tld{|\F|}$ using the maps $\eta_{(A,B)}$; and one has by definition $\tld{|\F|}(A,B):= \F(A,B)$. This gives a canonical map $\rho: \tld{|\F|} \to |\F|$ whose component at $(A,B)$ is the identity.  We also have a canonical map $\epsilon: \tld{|\F|} \to |\F|$ and we have a factorization of $\F \to |\F|$ as follows. 
$$\F \xrightarrow{\rho} \tld{|\F|} \xrightarrow{\epsilon} |\F|.$$
\end{rmk}

\begin{df}
Define the $2$-constant precategory associated to $\F$ to be the precategory $\tld{|\F|}$.
\end{df}

\begin{prop}\label{prop-equiv-deux-constant}
\begin{enumerate}
\item The map $\rho: \F \to \tld{|\F|}$ is an easy weak equivalence. 
\item Let $ L: \msxsu \to \Ba$ be a functor that takes easy weak equivalences to isomorphisms in $\Ba$. Then $L(\F \xrightarrow{\eta} |\F|)$ is an isomorphism in $\Ba$ if and only if $ L( \tld{|\F|} \xrightarrow{\epsilon} |\F|)$ is an isomorphism in $\Ba$
\end{enumerate}
\end{prop}
\begin{proof}
From the previous remark we know that the component of $\rho$ at $(A,B)$ is the identity and Assertion $(1)$ follows. Assertion $(2)$ is a consequence of Assertion $(1)$ together with the fact that isomorphisms in any category $\Ba$ have the $3$-for-$2$ property.
\end{proof}
\begin{rmk}
Note that the proposition holds also for functors 
from $\msx$ (resp. $\msxpt$) to $\Ba$ that takes easy weak equivalences to isomorphisms.
\end{rmk}
\subsubsection{Some pushouts and lifting problems in $\msxsu$, $\msxpt$ and $\msx$}
As usual we will limit our discussion to the category $\msxsu$ since the methods are the same for the two other categories.

\begin{nota}\label{notation-set-localization}
\begin{enumerate}
\item If $\alpha: U \to V$ is a morphism of $\M$, we will denote by 
$$\alpha_{\downarrow_{\Id_V}}: \alpha \to \Id_V,$$ 
the morphism in the arrow category $\M^{[1]}$ which is identified with the following commutative square. 
\[
\xy
(0,18)*+{U}="W";
(0,0)*+{V}="X";
(30,0)*+{V}="Y";
(30,18)*+{V}="E";
{\ar@{->}^-{\Id}"X";"Y"};
{\ar@{->}^-{\alpha}"W";"X"};
{\ar@{->}^-{\alpha}"W";"E"};
{\ar@{->}^-{\Id}"E";"Y"};
\endxy
\]
\item If $u$ is a $2$-morphism in $\sxop$, denote by $\Ev_u$ the evaluation at $u$
$$\Ev_u:\msxsu \to \M^{[1]},$$
that takes $\F$ to $\F(u)$.
\item The functor $\Ev_u$ has a left adjoint that will be denoted by $\Psi_{u}:  \M^{[1]} \to \msxsu$ or simply $\Psi$ if there is no potential confusion. It follows that if $\alpha$ is a morphism in $\M$ (=object of  $\M^{[1]}$) and $\F \in \msxsu$, we have functorial isomorphism of sets
$$\Hom_{\M^{[1]}}(\alpha, \F(u)) \cong \Hom_{\msxsu}(\Psi(\alpha), \F).$$
For the record if $u \in \sx(A,B)$, $\Psi$ is obtained as a composite of left adjoints as follows. 
$$\M^{1} \to \Hom[\sx(A,B)^{op}, \M] \xrightarrow{\delta_{AB}} \kx \xrightarrow{\Gamma} \msx \xrightarrow{\Rpt} \msxpt \xrightarrow{\Phi} \msxsu.$$
\item We will denote by $\Psi(\av): \Psi(\alpha) \to \Psi(\Id_V)$ the image of $\av$ by $\Psi$.
\item For every $1$-morphism $s=(A,...,B)$ in $\sx(A,B)^{op}$, we have a unique $2$-morphism 
$$u_s: (A,B) \to s.$$ 
For simplicity we will denote again like in Proposition \ref{prop-av} by 
$$\Psi_{s}: \M^{[1]} \to \msxsu$$ the previous left adjoint 
when $u=u_s$.
\end{enumerate}
\end{nota}
By the universal property of the pushout of $\alpha$ along itself, we find a unique map $i : V \cup^{U}V \to V$  that makes everything commutative in the following diagram.
\[
\xy
(0,20)*+{U}="W";
(0,0)*+{V}="X";
(40,0)*+{V}="Y";
(40,20)*+{V}="E";
{\ar@{->}_-{\Id}"X";"Y"};
{\ar@{->}^-{\alpha}"W";"X"};
{\ar@{->}^-{\alpha}"W";"E"};
{\ar@{->}^-{\Id}"E";"Y"};
(20,8)*+{V \cup^{U}V}="Q";
{\ar@{.>}^-{q}"Q";"Y"};
{\ar@{.>}^-{i_1}"X";"Q"};
{\ar@{.>}^-{i_0}"E";"Q"};
\endxy
\]
In Proposition \ref{prop-av} we write $i_{\alpha}$ for $i_1$. The inner square is a morphism in $\M^{[1]}$, that was denoted in that proposition by $\xi_{\alpha}: \alpha \to i_1$. We also have a map $\ell_\alpha: i_1 \to \Id_V$ which we display as the commutative square hereafter.
\[
\xy
(0,18)*+{V}="W";
(0,0)*+{V \cup^{U}V}="X";
(30,0)*+{V}="Y";
(30,18)*+{V}="E";
{\ar@{->}^-{q}"X";"Y"};
{\ar@{->}_-{i_\alpha=i_1}"W";"X"};
{\ar@{->}^-{\Id}"W";"E"};
{\ar@{->}^-{\Id}"E";"Y"};
\endxy
\]

It's easy to see that we have a factorization of $\av= \elalpha \circ \xialpha$. Applying $\Psi_s$ we get the equality
$$\Psi_s(\av)= \Psi_s(\elalpha) \circ \Psi_s(\xialpha).$$ 

\paragraph{A fundamental lemma} The following lemma is important since we shall use it to establish a \emph{strictification theorem}. We use the language of \emph{cell-complex} and we refer the reader to \cite{Hov-model} for a definition.
\begin{lem}\label{lem-cell-complex-av}
Let $L: \msxsu \to \Ba$ be a functor that sends easy weak equivalences to isomorphisms in $\Ba$. Then the following hold. 
\begin{enumerate}
\item $ L[\Psi_s(\av)]$ is an isomorphism in $\Ba$ if and only if $L[\Psi_s(\xialpha)]$ is an isomorphism in $\Ba$.
\item The functor $L$ sends any $\{\Psi_s(\av)\}$-cell complex to an isomorphism if and only if it sends any $\{\Psi_s(\xialpha)\}$-cell complex to an isomorphism.
\end{enumerate}
\end{lem}

\begin{proof}
The map $\elalpha$ has the key property that the top component is the identity $\Id_V$ which is a (wonderful) isomorphism.  Now thanks to Assertion $(3)$ of Proposition \ref{prop-av}, we know that it's image in $\Hom[\sx(A,B)^{op}, \M]$ by the left adjoint\footnote{This is a big abuse of notation} to  $\Ev_{u_s}$  is a morphism $\sigma$ with the property that the component $\sigma_{(A,B)}$ is an isomorphism in $\M$. 

In fact $\sigma_{(A,B)}$ is isomorphic to $\Id_V$. Applying $\delta_{AB}$ we have a morphism in $\kx$ which is an \ul{easy weak equivalence} and with the property that the component at every $(A',B')$ is an isomorphism (not only $(A,B)$). This follows from the definition of $\delta_{AB}$; these component are simply the identity $\Id_{\emptyset}$ of the initial object of $\M$. 

Now thanks to Lemma \ref{Gamma-no-change-easy} and Lemma \ref{lem-q-no-change} we know that the image in $\msxsu$ (obtained by $\Psi_s$) has the same property i.e, that the component $$ \Psi_s(\elalpha)_{(A',B')}: [\Psi_s(i_1)](A',B')  \to  [\Psi_s(\Id_V)](A',B') ,$$ 
is an isomorphism in $\M$ (in particular a trivial cofibration). It follows that $\Psi_s(\elalpha)$ is an easy weak equivalence so that it's image by $L$ is an isomorphism in $\Ba$.  Assertion $(1)$ follows by $3$-for-$2$ of isomorphisms in $\Ba$ applied to the equality
$$L[\Psi_s(\av)]= L[\Psi_s(\elalpha)] \circ L[\Psi_s(\xialpha)].$$
\\
\\
For Assertion $(2)$ it's enough to show that a pushout of $\Psi_s(\av)$ along some morphism $\sigma: \Psi_s(\alpha) \to \Ea$ is sent to an isomorphism if and only if the pushout of $\Psi_s(\xialpha)$ along the same $\sigma$ is sent to an isomorphism. To establish that, we're going the following facts. 
\begin{enumerate}
\item We will use the well known fact that \emph{a pushout followed by a pushout is a pushout}; we will refer it as the `concatenation of pushouts'.
\item  Every component of the map $\Psi_s(\elalpha)_{(A',B')}$ is an isomorphism therefore, thanks to Assertion $(3)$ of Proposition \ref{easy-pushout} we know that the pushout of $\Psi_s(\elalpha)$ is alway an easy  weak equivalence. 
\end{enumerate} 

Now consider $\theta_1: \Ea \to \F$, the pushout of $\Psi_s(\xialpha)$ along $\sigma$; it's the canonical map going to the pushout object. Denote by $\ol{\sigma}: \Psi_s(i_1) \to \F$ the other canonical map.\\

Let $\theta_2: \F \to \G$ be the pushout of $\Psi_s(\elalpha)$ along $\ol{\sigma}: \Psi_s(i_1) \to \F$.  By concatenation of pushouts the map  $\eta =\theta_2 \circ \theta_1$ is the pushout of $\Psi_s(\av)$ along $\sigma$.\\
Thanks to the above facts, we know already that $L(\theta_2)$ is an isomorphism since $\theta_2$ is an easy weak equivalence. Then Assertion $(2)$ follows also by $3$-for-$2$ of isomorphisms in $\Ba$ applied to the equality 
$$L(\theta)= L(\theta_2) \circ L(\theta_1).$$ 

\end{proof}

\subsubsection{\emph{co-Segalification} for $2$-constant precategories}
If we want to define a functor $\Sim$ that takes a $2$-constant precategory $\F$ to a precategory that satisfies the co-Segal conditions the natural thing to do is to factorize the map $\F(A,B) \to |\F|(A,B)$ as a cofibration followed by a trivial fibration as follows.
$$\F(A,B) \hookrightarrow m \xtwoheadrightarrow{\sim} |\F|(A,B).$$
After this we would want to set $\Sim(\F)(A,B)=m$ and $|\Sim(\F) |= |\F|$. This gives a $2$-constant diagram that satisfies the co-Segal conditions. The purpose of the following discussion is to show that this is done as $\kb$-injective replacement in $\msxsu$ where $\kb$ is some set of maps that will be defined in a moment.\\  
 The first ingredient we need is the following lemma. 

\begin{lem}\label{lem-deux-fondamental}
Let  $s=(s)$ and $\Psi_{s}: \M^{[1]} \to \msxsu$ be the left adjoint to the evaluation at $u_s: (A,B) \to (s)$. Let $\alpha: U \to V$  be a morphism of $\M$  and $\F$ be a \ul{$2$-constant} precategory.\\

Let $\sigma : \Psi_s(\av) \to \F$ be a morphism in $\msxsu$ and let  $\Ea=\Psi_s(\Id_V) \cup^{\Psi_s(\alpha)} \F$ be the object obtained by the following pushout diagram in $\msxsu$. 

\[
\xy
(0,18)*+{\Psi_s(\alpha)}="W";
(0,0)*+{\Psi_s(\Id_V) }="X";
(40,0)*+{\Psi_s(\Id_V) \cup^{\Psi_s(\alpha)} \F}="Y";
(40,18)*+{\F}="E";
{\ar@{->}^-{}"X";"Y"};
{\ar@{->}^-{\Psi_s(\av)}"W";"X"};
{\ar@{->}^-{\sigma}"W";"E"};
{\ar@{->}^-{\varepsilon}"E";"Y"};
\endxy
\]

Then $\Ea=\Psi_s(\Id_V) \cup^{\Psi_s(\alpha)} \F$ is also a \ul{$2$-constant} precategory and has the following properties.
\begin{enumerate}
\item If $(A',B')\neq (A,B)$ then the natural transformation $\varepsilon: \F_{A'B'} \to \Ea_{A'B'}$ is an isomorphism.
\item The natural transformation $\varepsilon_{\geq 2}:  \F_{A'B',\geq 2} \to \Ea_{A'B', \geq 2}$ in $\Hom[\sx(A',B')^{op}_{\geq 2}, \M]$ is a isomorphism for all $(A',B')$ (including $(A,B)$).
\end{enumerate}
\end{lem}
Before giving the proof of the lemma, we give the following remark that will help in the proof. 
\begin{rmk}\label{rmk-important-deux-constant}
The reader can check, using the adjunction, that for any $\G\in \msxsu$, we have an equivalence between the following data.
\begin{enumerate}
\item A commutative square as
\[
\xy
(0,18)*+{\Psi_s(\alpha)}="W";
(0,0)*+{\Psi_s(\Id_V) }="X";
(40,0)*+{\G}="Y";
(40,18)*+{\F}="E";
{\ar@{->}^-{}"X";"Y"};
{\ar@{->}^-{\Psi_s(\av)}"W";"X"};
{\ar@{->}^-{\sigma}"W";"E"};
{\ar@{->}^-{\theta}"E";"Y"};
\endxy
\]
\item  A commutative diagram in $\M$ with a lifting as follows. 
\[
\xy
(0,20)*+{U }="A";
(25,20)*+{\F(A,B)}="B";
(0,0)*+{V}="C";
(25,0)*+{|\F|(A,B)}="D";
(50,20)*+{\G(A,B)}="E";
(50,0)*+{\G(s)}="F";
{\ar@{->}^-{}"A";"B"};
{\ar@{->}_-{\alpha}"A"+(0,-3);"C"};
{\ar@{->}_-{}"C"+(3,0);"D"};
{\ar@{->}^-{}_{}"B"+(0,-3);"D"};
{\ar@{->}^-{}"B";"E"};
{\ar@{->}_-{}^-{}"D";"F"};
{\ar@{->}^-{}"E";"F"};
%%%%% lifting%%%%
{\ar@{.>}^-{}"C";"E"};
\endxy
\] 

\end{enumerate}
\end{rmk}
\begin{proof}
We will construct the $2$-constant diagram $\Ea$ and show that it satisfies the universal property of the pushout.\\ 

If $(A',B') \neq (A,B)$ we set  $\Ea_{AB}= \F_{AB}$. And if $(A',B')=(A,B)$ we define 
$$ \Ea_{A'B', \geq 2}= \F_{A'B',\geq 2},$$ which is then a constant diagram of value $|\F|(A,B)$.  It remains to define $\Ea(A,B)$.\\

Note that the map $\F(u_s):  \F(A,B) \to \F(s)$ is  $\F(A,B) \to |\F|(A,B)$  since $\F$ is $2$-constant. Now by adjunction the map $\sigma: \Psi_s(\alpha) \to \F$ corresponds to a unique commutative square (a morphism in $\M^{[1]}$) as follows. 

\[
\xy
(0,18)*+{U}="W";
(0,0)*+{V}="X";
(30,0)*+{\F(A,B)}="Y";
(30,18)*+{|\F|(A,B)}="E";
{\ar@{->}^-{p}"X";"Y"};
{\ar@{->}^-{\alpha}"W";"X"};
{\ar@{->}^-{q}"W";"E"};
{\ar@{->}^-{\F(u_s)}"E";"Y"};
\endxy
\]

Define $\Ea(A,B)$ to be the object we get by forming the pushout of 
$$V \xleftarrow{\alpha} U \xrightarrow{q} \F(A,B).$$

Let $\varepsilon:\F(A,B) \to \Ea(A,B)$ and $i_V: V \to \Ea(A,B)$ be the canonical maps.\\
 
The universal property of the pushout gives a unique map  $\gamma: \Ea(A,B) \to |\F|(A,B)$ such that the factorizations below hold.
$$\F(u_s)=\gamma \circ \varepsilon; \quad p= \gamma \circ i_V.$$

And we obviously have $ \varepsilon \circ q = i_V \circ \alpha$, and the following square commutes.
\[
\xy
(0,18)*+{U}="W";
(0,0)*+{V}="X";
(40,0)*+{\underbrace{|\F|(A,B)}_{=\Ea(s)}}="Y";
(40,18)*+{\Ea(A,B)}="E";
{\ar@{->}^-{p}"X";"Y"};
{\ar@{->}^-{\alpha}"W";"X"};
{\ar@{->}^-{\varepsilon \circ q}"W";"E"};
{\ar@{->}^-{\gamma}"E";"Y"};
{\ar@{.>}^-{i_V}"X";"E"};
\endxy
\]

As shown above we have a (tautological)  lifting given by $i_V$.\\

If we put together the previous observations, we find maps 
$$\gamma: \Ea(A',B') \to |\F|(A',B'); \quad \tx{and} \quad  \varepsilon: \F(A',B') \to \Ea(A',B'),$$
for all $(A',B')$ such that every canonical $\F(A',B') \to |\F|(A',B')$ is the composite $\gamma \circ \varepsilon$. 

In particular for $A'=B'$ we get a map $I_A' : I \to \F(A',A') \to \Ea(A',A')$ which gives by composition with $\gamma: \Ea(A',A') \to |\F|(A',A')$ the unity of the strict $\M$-category $|\F|$.

We find that we are in the situation of Lemma \ref{lem-deux-constant} therefore $\Ea$ is a unital $2$-constant diagram.\\

The reader can check that we have a morphism $\varepsilon: \F \to \Ea$ and $\gamma: \Ea \to |\F|$ and that $\varepsilon$ is as in the statement of the lemma. Note that the map $\Ea(u_s)$ is just $\gamma: \Ea(A,B) \to |\F|(A,B)$. \\

If we expand the above commutative square that has a lifting, we see that we are in the situation of Remark \ref{rmk-important-deux-constant} and therefore, we get a canonical commutative square as shown below.
\[
\xy
(0,18)*+{\Psi_s(\alpha)}="W";
(0,0)*+{\Psi_s(\Id_V) }="X";
(40,0)*+{\Ea}="Y";
(40,18)*+{\F}="E";
{\ar@{->}^-{}"X";"Y"};
{\ar@{->}^-{\Psi_s(\av)}"W";"X"};
{\ar@{->}^-{\sigma}"W";"E"};
{\ar@{->}^-{\varepsilon}"E";"Y"};
\endxy
\]

It remains to show that this square is the universal one i.e, that $\Ea$ equipped with the appropriate maps satisfies the universal property of the pushout. \\

Let $\G$ be an arbitrary unital diagram such that we have commutative diagram as in Remark \ref{rmk-important-deux-constant}. Then as said in that remark, this is equivalent to having the commutative square below that possesses  a lifting. 
\[
\xy
(0,20)*+{U }="A";
(25,20)*+{\F(A,B)}="B";
(0,0)*+{V}="C";
(25,0)*+{|\F|(A,B)}="D";
(50,20)*+{\G(A,B)}="E";
(50,0)*+{\G(s)}="F";
{\ar@{->}^-{}"A";"B"};
{\ar@{->}_-{\alpha}"A"+(0,-3);"C"};
{\ar@{->}_-{}"C"+(3,0);"D"};
{\ar@{->}^-{}_{}"B"+(0,-3);"D"};
{\ar@{->}^-{}"B";"E"};
{\ar@{->}_-{}^-{}"D";"F"};
{\ar@{->}^-{}"E";"F"};
%%%%% lifting%%%%
{\ar@{.>}^-{}"C";"E"};
\endxy
\] 

Since the upper half triangle ending at $\G(A,B)$ is commutative, the universal property of the pushout gives a \ul{unique map} $\zeta: \Ea(A,B) \to \G(A,B)$ 
with the obvious factorizations. Another application of the universal property of the pushout with respect to the whole commutative square ending at  $\G(s)$ gives by uniqueness and equality 
$$ \theta \circ \gamma= \G(u_s) \circ \zeta.$$ 

In terms of commutative square we get the following.
\[
\xy
(0,20)*+{\F(A,B) }="A";
(25,20)*+{\Ea(A,B)}="B";
(0,0)*+{|\F|(A,B)}="C";
(25,0)*+{|\F|(A,B)}="D";
(50,20)*+{\G(A,B)}="E";
(50,0)*+{\G(s)}="F";
{\ar@{->}^-{}"A";"B"};
{\ar@{->}_-{\F(u_s)}"A"+(0,-3);"C"};
{\ar@{->}_-{\Id}"C";"D"};
{\ar@{->}^-{\Ea(u_s)}_{}"B"+(0,-3);"D"};
{\ar@{->}^-{}"B";"E"};
{\ar@{->}_-{}^-{}"D";"F"};
{\ar@{->}^-{}"E";"F"};
%%%%% lifting%%%%
%{\ar@{.>}^-{}"C";"E"};
\endxy
\] 

Note that the lifting $V \to \G(A,B)$ is $\zeta \circ i_V$ and $i_V$ is the lifting for the (universal) square associated to $\Ea$. Since $\Ea=\F$ everywhere except for $\F(A,B)$, we see that if we assemble the map $\zeta$ and the data for the map $\theta : \F \to \G$ we find a unique map  $ \zeta: \Ea \to \G$ such that
$$\theta= \zeta \circ \varepsilon.$$

Using the relation between the two liftings $\zeta \circ i_V$ and $i_V$, we find by uniqueness of the adjoint transpose that the given map $\Psi_s(\Id_V) \to \G$ is the composite of the canonical map $\Psi_s(\Id_V) \to \Ea$ and $\zeta$. 
This completes the proof.  
%%%%%%%%% continuer ici la preuve %%%%%%%%
\end{proof}

We also need the following lemma. 
\begin{lem}\label{lem-wide pushouts-deux-constants}
Let $\{\varepsilon_i: \F \to \Ea_i \}_{i \in S}$ be a small family of morphisms between $2$-constant precategories in $\msxsu$. Assume that each morphism 
$\varepsilon_i: \F \to \Ea_i$ is such that the induced $\M$-functor $|\varepsilon_i|: |\F| \to |\Ea_i| $ is an isomorphism.\\

Let $\Ea_{\infty}$ be the wide pushout in $\msxsu$ of the maps $\varepsilon_i$. Then the following hold.
\begin{enumerate}
\item $\Ea_\infty$ is also a $2$-constant precategory.
\item The canonical maps  $\Ea_i \to \Ea_\infty$  and $ \F \to \Ea_\infty$ induce isomorphisms  between the respective categories.  
$$|\Ea_i|\xrightarrow{\cong} |\Ea_\infty|, \quad |\F| \xrightarrow{\cong} |\Ea_\infty|.$$
\end{enumerate}
 \end{lem}
\begin{proof}
Take $|\Ea_\infty|$ to be the pushout of isomorphisms $|\varepsilon_i|: |\F| \to |\Ea_i|$. Clearly the canonical maps $ |\Ea_i|\xrightarrow{\cong} |\Ea_\infty|$ and $|\F| \xrightarrow{\cong} |\Ea_\infty|$ are isomorphisms. 

One gets the objects 
$\Ea_{\infty}(A,B)$ together with the unique map $\Ea_{\infty}(A,B) \to |\Ea_\infty|(A,B)$ by taking the wide pushout of $\F(u_s) \to \Ea_i(u_s)$ in the arrow category $\M^{[1]}$. Just like before $\Ea_\infty$ is also a unital precategory. The canonical maps $|\Ea_i|\xrightarrow{\cong} |\Ea_\infty|$ and $|\F| \xrightarrow{\cong} |\Ea_\infty|$ extend to morphisms in $\msxsu$ 
$$\Ea_i \to \Ea_\infty ; \quad \F \to \Ea_\infty.$$

Moreover the canonical map  $\F \to \Ea_\infty$ is the composite of  $\F \to \Ea_i$ and $\Ea_i \to \Ea_\infty$; thus we have a natural cocone (ending at $\Ea_\infty$). The reader can easily check that this cocone is the universal one i.e, $\Ea_\infty$ equipped with this cocone satisfies the universal property of the wide pushout.
\end{proof}
Finally we have a classical lemma in category theory whose proof will be left as an exercise. 

\begin{lem}\label{coproduc-pushout}
Let  $\{ C_i \xleftarrow{p_i} A_i \xrightarrow{h_i} B  \}_{i \in S}$ be a family of pushout data in a cocomplete category $\B$. Let $D_i=C_i \cup^{A_i}B $
be the pushout object for each diagram and denote by $\eta_i: B \to D_i$ the canonical map.\\

Let $\coprod p_i: \coprod A_i \to \coprod C_i$ and $\sum h_i: \coprod A_i \to B$ be the canonical maps induced by universal property of the coproduct. Let $E$  be the pushout object of the diagram $$ \coprod C_i \xleftarrow{\coprod p_i} \coprod A_i \xrightarrow{\sum h_i} B, $$
and denote by $\gamma: B \to E$ the canonical map. Let $\ol{D}$ be the object obtained by taking the wide pushout $\{B \xrightarrow{\eta_i} D_i \}$ and let $\delta: B \to  \ol{D}$ be the canonical map.\\

Then the two maps $\gamma: B \to E$ and $\delta: B \to  \ol{D}$ are isomorphic in the category $\B^{[1]}$ of morphism of $\B$. In particular $E \cong \ol{D}$.
\end{lem}

 Define the \emph{minimal localizing set} for $2$-constant precategories as:
$$\kb_2=\bigsqcup_{(A,B)\in X^2} \bigsqcup_{s \in \sx(A,B)^{op}| \degb(s)=2} \{\Psi_s(\av); \quad \alpha \in \I  \}.$$

Recall that we've also introduced maps $\Psi_s(\elalpha)$ and $\Psi_s(\xialpha) $ and we have an equality $\Psi_s(\av)= \Psi_s(\elalpha) \circ \Psi_s(\xialpha).$ Below we use the \emph{Gluing construction} and the \emph{Small object argument}. There are numerous references in the literature (see for example \cite{Dwyer_Spalinski,Hov-model}).

\begin{pdef}\label{pdef-deux-constant}
Let $\Sim_2: \msxsu \to \msxsu$ be the $\kb_2$-injective replacement functor obtained by applying the gluing construction and the small object argument. Denote by $\eta: Id \to \Sim$ the induced natural transformation. Then the following hold.

\begin{enumerate}
\item If $\F$ is a $2$-constant precategory then so is $\Sim(\F)$; and the map $|\eta|: |\F| \to |\Sim(\F)|$ is an isomorphism of $\M$-categories.
\item $\Sim(\F)$ satisfies the co-Segal conditions for every $2$-constant precategory $\F$. And the canonical map $\Sim(\F) \to |\Sim(\F)|$ is an easy weak equivalence.
\item If $L: \msxsu\to  \Ba$ is a functor that sends easy weak equivalences to isomorphisms and takes any pushout of $\Psi_s(\xialpha) $ to an isomorphism, then for all $\F \in \msxsu$, the image of $\F \to |\F|$ by $L$ is an isomorphism in $\Ba$. 
\end{enumerate}
The functor $\Sim$ will be called the \emph{ $2$-constant co-Segalification functor}. 
\end{pdef}

\begin{proof}
The full subcategory of $2$-constant precategories is closed under directed colimits (and limits) since they are computed level-wise. From Lemma \ref{lem-deux-fondamental}, we know that if $\F$ is $2$-constant, then the pushout of any $\Psi_s(\av)$ along any  $\Psi_s(\alpha) \to \F$ is a morphism of $2$-constant precategories which is moreover an isomorphism on the \emph{underlying categories} (obtained by the left adjoint $|-|$).

Using Lemma \ref{lem-wide pushouts-deux-constants} and Lemma \ref{coproduc-pushout} with respect to any pushout data 
$$(\coprod_{(A,B)\in X^2} \coprod_{s \in \sx(A,B)^{op}| \degb(s)=2} \Psi_s(\Id_V)) \xleftarrow{\coprod \Psi_s(\av)} (\coprod_{(A,B)\in X^2} \coprod_{s \in \sx(A,B)^{op}| \degb(s)=2} \Psi_s(\alpha)) \to \F, $$
we find that the canonical map $\F \to \F_1$ going to the pushout-object is again a map of $2$-constant precategories. Moreover the induce map $|\F| \to |\F_1|$ is an isomorphism of $\M$-categories.  But these pushouts are precisely the one we use to construct $\Sim(\F)$. 
It follows that $\Sim(\F)$ is a $2$-constant precategory as a directed colimit therein;  and the map 
$ \eta: \F \to \Sim(\F)$ is a  $\kb_2$-cell complex with the property that the induce map $|\eta|: |\F| \to |\Sim(\F)|$ is an isomorphism of $\M$-categories. This proves Assertion $(1)$. \\

Assertion $(2)$ follows from the fact that $\Sim(\F)$ in $\kb_2$-injective, and by adjunction we find that the unique map $\Sim(A,B) \to |\Sim|(A,B)$, viewed as an object of $\M^{[1]}$, is $\av$-injective for all generating cofibration $\alpha$. But this in turn simply means that we have a lifting to any problem defined by $\alpha$ and  $\Sim(A,B) \to |\Sim|(A,B)$ (see Proposition \ref{lem-lifting-inject} below). Consequently $\Sim(A,B) \to |\Sim|(A,B)$ is a trivial fibration, in particular a weak equivalence, therefore $\Sim(\F)$ is a co-Segal category. This also proves at the same time, in a tautological way, that the canonical map $\Sim(F) \to |\Sim(\F)|$, whose components are precisely the maps $\Sim(A,B) \to |\Sim|(A,B)$ is an easy weak equivalence.\\
 
Now for Assertion $(3)$  is suffices to use the factorization $\F \to |\F|$ given in Remark \ref{rmk-associated-deux-constant} and observe that we can factorize again as follows.

$$\F \xrightarrow{\rho} \tld{\F} \xrightarrow{\eta} \Sim(\tld{\F}) \xrightarrow{\sim} |\tld{\F}|= |\F|.$$

Under the assumptions of the proposition, we get from the second assertion of Lemma  \ref{lem-cell-complex-av}, that every $\kb_2$-cell complex is also send to isomorphism, in particular $L(\eta)$ is an isomorphism. The third map is an easy weak equivalence by the previous assertion we've just proved. Now the map $\rho$ is also an easy weak equivalence by Proposition \ref{prop-equiv-deux-constant} (actually it's by construction of $\tld{\F}$). In the end we find that the image of  $\F \to |\F|$ by $L$ is also an isomorphism.
\end{proof}

%%%%%%%%%%
\subsubsection{Changing the set of objects}
Given a function  $f: X \to Y$, we have a pullback functor $\fstar: \msy \to \msx$. This functor preserve any level-wise property e.g, limit directed limits, etc. 

If $\M$ is locally presentable then so are $\msxsu$ and $\msysu$; and thanks to the fact that $\fstar$ preserves directed colimits, we get from the adjoint functor of locally presentable (see \cite{Adamek-Rosicky-loc-pres}) the following. 
\begin{prop}
The functor $\fstar$ has a left adjoint $\fex: \msxsu \to \msysu$, called  \emph{the push forward}.
\end{prop} 
Because (trivial) fibrations in $\msxsueasy$ are characterized level-wise, it follows that:
\begin{prop}\label{quillen-pair-fstar-fex}
The functor $\fstar$ preserves fibrations and trivial fibration, therefore the adjunction $\fex \dashv \fstar$ is a Quillen adjunction.
\end{prop}
\subsection{Local \emph{easy} model structure for all precategories}
In what follows we are going to show that the category $\mset$ (resp. $\mset_{\ast}$, resp. $\mset_{su}$)  of all precategories (resp. pointed, resp. unital) has a model structure in which the weak equivalences and fibration are respectively the \emph{local easy  weak equivalences} and \emph{local easy fibrations} (see below).\\
\begin{warn}
Every definition, lemma, proposition and Theorem will hold for the three categories $\mset$, $\mset_{\ast}$, $\mset_{su}$. But in order to avoid saying each time ``respectively'' we will simply give every statement for $\mset_{su}$ which is the category we are much interested in.
\end{warn}

Recall for $\F \in \msxsu$ and $\G \in \msysu$, a map $\sigma: \F \to \G$ in $\mset_{su}$ is a pair $(f,\sigma)$ where $f: X \to Y$ is a function and $\sigma: \F \to \fstar \G$ is a morphism in $\msxsu$. 
\begin{df}
Let $\F \in \msxsu$ and $\G \in \msysu$ be objects of $\mset_{su}$. Say that a map $\sigma=(f,\sigma): \Fa \to \Ga$ in $\mset_{su}$ is:
\begin{enumerate}
\item a \emph{local easy fibration} if  the induced map $\F \to \fstar \G$  is a fibration in $\msxsueasy$.
\item a \emph{local easy weak equivalence} if the map 
$\F \to \fstar \G$ is a weak equivalence in $\msxsueasy$ i.e, if the component $\sigma: \F(A,B) \to \G(fA,fB)$
is a weak equivalence for all pair of objects $(A,B)$ of $\sx$.
\end{enumerate}
\end{df}  
\subsubsection{The generating sets}
\paragraph{Some natural $\S$-diagrams} 
The discussion we present here follows Simpson's considerations in \cite[13.2]{Simpson_HTHC}.\ \\

Let $\cn$ be the indiscrete category associated to the set $\{0,...,n \}$. In the $2$-category $\S_{\cn}$, there is a special $1$-morphism from $0$ to $n$  corresponding to the $n+1$-tuple $(0,...,n)$. It is the maximal nondegenerate simplex in the nerve of $\cn$.  We will denote this $1$-morphism by $s_n$. Let $\Fb^{s_n}_{-}: \M \to \Hom[\S_{\cn}(0,n)^{op}, \M] $ be the left adjoint of the evaluation at $s_n$. \ \\

We have as usual the categories $\msn, \msn_{\ast}, \msn_{su}$ and $\kn$ with the monadic adjunction  $\Qs \dashv \Ub$;
where $\Qs \in \{\Gamma, \Rpt\Gamma, \Phi \Rpt \Gamma \}$.\\

Each adjunction is moreover a Quillen adjunction with the respective easy model structure. For the record $\kn= \prod_{(i,j) \in \Ob(\cn)^2} \Hom[\S_{\cn}(i,j)^{\op}, \M] $. \\

We will use the following notation.
\begin{nota}
\begin{enumerate}
\item For $V \in \Ob(\M)$ we will denote by $\delta(s_n, V)$ the object of $\kn$ given by:
 
 \begin{equation*}
\delta(s_n, V)_{ij} =
  \begin{cases}
   \Fb^{s_n}_{V}  & \text{if $i=0, j=n$} \\
    (\emptyset, \Id_{\emptyset})  & \text{the constant functor otherwise}.
  \end{cases}
\end{equation*}
\item For $V \in \Ob(\M)$ and $\Qs= \Gamma$ (resp. $\Rpt\Gamma$, resp. $\Phi \Rpt \Gamma$),   define  $$\hb(\cn; V)=\Qs \delta(s_n, V)$$
be the corresponding adjoint in  $\msn$ (resp. $ \msn_{\ast}$, resp. $\msn_{su}$) 
\end{enumerate}
\end{nota}

\begin{lem}\label{gen-set}
For any $V \in \Ob(\M)$ and $\Fa \in \msysu$ the following are equivalent.
\begin{enumerate}
\item A morphism $\sigma : \hb(\cn; V) \to \Fa$ in $\mset_{su}$.
\item A sequence of elements  $(A_0, ..., A_n)$ of $Y$ together with a morphism $V \to \Fa(A_0,....,A_n)$ in $\M$.
\end{enumerate}
\end{lem}
The proof is simply a consequence of the various adjunctions mentioned previously. We include it here for completeness.
\begin{proof}[Sketch of proof]

A morphism $\sigma=(f,\sigma) : \hb(\cn; V) \to \Fa$ is by definition a function $$f : \{0,...,n\} \to Y$$ together with a morphism $\sigma: \hb(\cn; V) \to \fstar \Fa$ in $\msn_{su}$. Setting $A_i= f(i)$ we get $f s_n= (A_0,...,A_n)$ and  by adjunction we have:

\begin{equation*}
\begin{split}
\Hom_{\mset_{su}}[(\hb(\cn; V), \Fa]
&=\Hom_{\msn_{su}}[(\hb(\cn; V),\fstar \Fa ]\\
&= \Hom_{\msn_{su}}[\Qs \delta(s_n, V),\fstar \Fa ]  \\
&\cong \Hom_{\kn}[\delta(s_n, V), \Ub (\fstar \Fa) ]   \\
&\cong \Hom[\Fb^{s_n}_{V},\fstar \Fa_{A_0 A_n} ] \\
&\cong \Hom[V, \Fa_{A_0 A_n} (f s_n) ]\\
&= \Hom[V, \Fa(A_0,....,A_n)].\\
\end{split}
\end{equation*}
\end{proof}

Recall that we assume that $\M$ is cofibrantly generated with a set $\I$ (resp. $\Ja$) of  generating cofibrations (resp. trivial cofibrations). In fact we also ask $\M$ to be a locally presentable category.

As a corollary of the lemma one has the following proposition. 
\begin{prop}\label{lifting-mset}
Let $z=(A_0,...,A_n)$ be a $1$-morphism of  $\sx$ of degree $n$ and  $ \F \to \G$ be a morphism in $\mset_{su}$ with $\F \in \msxsu$ and $\G \in \msysu$. Then the following hold.
\begin{enumerate}
\item The component $\sigma_z: \F(z) \to \G(f(z))$ is a fibration in $\M$ if and only if $\sigma$ has the RLP with respect to all maps in the following set.
$$\{ \hb(\cn; U) \xrightarrow{\hb(\cn; \alpha)} \hb(\cn; V)  \}_{\alpha:U \to V \in \Ja}.$$  
\item Similarly, $\sigma_z: \F(z) \to \G(f(z))$ is a trivial fibration in $\M$ if and only if $\sigma$ has the RLP with respect to all maps in the following set.
$$\{ \hb(\cn; U) \xrightarrow{\hb(\cn; \alpha)} \hb(\cn; V)  \}_{\alpha:U \to V \in \I}.$$  
\end{enumerate} 
\end{prop}
\begin{proof}
Using Lemma \ref{gen-set}, we have that any lifting problem in $\mset_{su}$ defined by 
$$\hb(\cn; U) \xrightarrow{\hb(\cn; \alpha)}  \hb(\cn; V)\quad \tx{and}\quad \sigma: \F \to \G,$$
is equivalent to a lifting problem in $\M$ defined by 
$$ \alpha: U \to V \quad \tx{and} \quad \sigma_z: \F(z) \to \G(f(z)).$$

Therefore a solution to one problem is equivalent to a solution for the other one.
\end{proof}
\subsubsection{The main theorem}
\begin{warn}
In the following we will use the previous material with $\n=\1$. Therefore we have the corresponding objects $\delta(s_1, \alpha), \hb([\1]; \alpha)$, etc.
\end{warn}
\begin{nota}
We will use the following notation.
\begin{enumerate}
\item $\W_{\mset_{su}}=$ the class of local easy weak equivalences.
\item $\I_{\mset_{su}}= \{ \hb([\1]; \alpha): \hb([\1]; U) \to \hb([\1]; V) \}_{\alpha :U \to V \in \I}$.
\item $\Ja_{\mset_{su}}= \{ \hb([\1]; \alpha): \hb([\1]; U) \to \hb([\1]; V) \}_{\alpha :U \to V \in \Ja}.$
%\item $\I_{\kset}= \coprod_{n\geq 1} \{ \delta(s_n, \alpha): \delta(s_n, A)\to \delta(s_n, V) \}_{\alpha :U \to V \in \I}$.
%\item $\Ja_{\kset}= \coprod_{n\geq 1} \{ \delta(s_n, \alpha): \delta(s_n, A)\to \delta(s_n, V) \}_{\alpha :U \to V \in \Ja} $.
\end{enumerate}
We also have similar sets for $\mset$ and $\mset_{\ast}$.
\end{nota}

In virtue of Theorem \ref{thm-recognition} we have the following.
\begin{thm}\label{first-model-cat-mset}
For a combinatorial monoidal model category $\M$, there is a combinatorial model structure on the category $\mset_{su}$ (resp. $\mset$, resp. $\mset_{\ast}$) in which the weak equivalences and fibrations are precisely the local easy weak equivalences and local easy fibrations, respectively. \\
\\
The generating set of cofibrations is $\I_{\mset_{su}}$ (resp.$\I_{\mset}$ , resp. $\I_{\mset_{\ast}}$).\\
The generating set of trivial cofibrations is $\Ja_{\mset_{su}}$ (resp. $\Ja_{\mset}$, resp. $\Ja_{\mset_{\ast}}$).\\
\\
We will denote by $\mset_{su,easy}$ (resp. $\mset_{easy}$, resp. $\mset_{\ast}$) this model category. 
\end{thm}

\begin{proof}
The proof is \textbf{exactly} the same as of Theorem \ref{thm-local-proj}. We will only give the proof for $\mset_{su}$, the method remains the same for $\mset$ and $\mset_{\ast}$. We shall verify briefly that all conditions of Theorem \ref{thm-recognition} are fulfilled.
\begin{itemize}[label=$-$]
\item Conditions $1,2,3$ are clear as the reader can check.
\item Condition $5$ is given by the second assertion of Proposition \ref{lifting-mset}.
\item Condition $4$ i.e, the inclusion $\Ja_{\mset_{su}}\tx{-cell} \subseteq \W_{\mset_{su}} \cap \I_{\mset_{su}}\tx{-cof}$, follows from two steps. First the inclusion $\Ja_{\mset_{su}}\tx{-cell} \subseteq \I_{\mset_{su}}\tx{-cof}$ is clear because in $\M$ we have $\Ja \tx{-cell} \subseteq \I\tx{-cof}$ and everything is transported by adjunction.

Given a pushout datum $$\hb([\1]; V) \xleftarrow{\hb([\1]; \alpha)} \hb([\1]; U) \xrightarrow{(f,\sigma)} \F,$$
where $\hb([\1]; \alpha) \in \Ja_{\mset_{su}}$ and $\F \in \msxsu$, one computes the pushout in $\msxsu$ by pushing forward through $\fex$. Since $\hb([\1]; \alpha)$ is a trivial cofibration in $\msn_{su,easy}$  we find that  $\fex \hb([\1]; \alpha) $ is a trivial cofibration in $\msxsueasy$ because $\fex \dashv \fstar$ is a Quillen pair with $\fex$ left Quillen (Proposition \ref{quillen-pair-fstar-fex}).

Therefore the pushout in $\msxsueasy$ along $\fex \hb(\cn; \alpha)$ is a trivial cofibration in $\msxsueasy$, in particular a weak equivalence and the inclusion $\Ja_{\mset_{su}}\tx{-cell} \subseteq \W_{\mset_{su}}$ follows.
\item To get Condition $(6)$ we simply show that we have an inclusion 
$$\Ja_{\mset_{su}}\tx{-inj} \cap \W_{\mset_{su}} \subseteq \I_{\mset_{su}}\tx{-inj}.$$

By definition maps in $\I_{\mset_{su}}\tx{-inj}$ are the maps $\sigma$ whose component $\sigma_{(A,B)}$ is a   trivial fibrations: this is the second assertion of Proposition \ref{lifting-mset}. 

It follows that if $\sigma: \F \to \G$ is in $ \Ja_{\mset_{su}}\tx{-inj} \cap\W_{\mset_{su}}$, then in one hand $\sigma_{(A,B)}$  is fibration since $\sigma$ is $\hb([\1] \alpha)$-injective, for all $\alpha \in \Ja$ (Proposition \ref{lifting-mset}). On the other hand $\sigma_{(A,B)}$ is also a weak equivalence since $\sigma$ is an easy weak equivalence by hypothesis. The combination of the two facts tells us that $\sigma \in \I_{\mset_{su}}\tx{-inj}$.     
\end{itemize} 
\end{proof}
\begin{rmk}
One can show that we have a similar theorem for $\kset$, thus a model structure $\kset_{\tx{-}easy}$. And the following adjunction that can be found in \cite{COSEC1} is a Quillen adjunction.
$$\Gamma: \kset_{\tx{-}easy} \leftrightarrows \mset_{su,easy}: \Ub.$$
\end{rmk}
\begin{cor}
If a local model structure on $\mcat$ exists, we have a Quillen adjunction
$$|-|: \mset_{su,easy} \leftrightarrows \mcat: \iota.$$
\end{cor}
\begin{proof}
From Proposition \ref{lifting-mset} it's clear that a local (trivial) fibration of $\M$-categories is also a (trivial) fibration in $\msetsu$.
\end{proof}

\section{Fibred localization}\label{sec-fibre-loc}
\subsection{Localizing sets and \emph{co-Segalification}}
If $s=(A,...,B)$ is a $1$-morphism in $\sx(A,B)^{op}$, we mentioned in Notation \ref{notation-set-localization} that there is a unique morphism $u_s: (A,B) \to s$ in $\sx(A,B)^{op}$. We also mentioned that the evaluation at $u_s$ has a left adjoint $\Psi_s: \M^{[1]} \to \msxsu$.

Let $\alpha: U \to V$ be a cofibration in $\M$ and let $\Psi_s(\av): \Psi_s(\alpha) \to \Psi_s(\Id_V)$ be the image of the morphism $\av: \alpha \to \Id_V$ of $\M^{[1]}$ that has been introduced in Notation \ref{notation-set-localization}.

\begin{df}
\begin{enumerate}
\item Define the \emph{localizing set} for $\msxsu$ as
$$\kb_{\msxsu}:= \{\coprod_{s \in 1-Mor(\sx)} \{\Psi_s(\av); \quad \alpha \in \I  \} \}$$
\item Define the \emph{localizing set} for $\msetsu$ as
$$\kb_{\msetsu}:= \{  \coprod_{\n \geq 1} (\coprod_{s \in 1-Mor(\sn)} \{\Psi_s(\av); \quad \alpha \in \I  \}) \}.$$
\item Let $\ast$ be the coinitial (or terminal) object of $\msetsu$. If $\sigma$ is a map in $\msetsu$, say that an object $\F \in \msetsu$ is $\sigma$-injective if the unique map $\F \to \ast$ has the RLP with respect to $\sigma$. 
\item Similarly say that $\F$ is $\kb_{\msetsu}$-injective (resp. $\kb_{\msxsu}$-injective) if it's $\sigma$-injective for all $\sigma \in \kb_{\msetsu}$ (resp. $\kb_{\msxsu}$).
\end{enumerate}
\end{df}

One can easily establish the following proposition.
\begin{prop}\label{lem-lifting-inject}
Let $\theta=(f,g): \alpha \to p$ be a morphism in $\M^{[\1]}$ which is represented by the following commutative square.

\[
\xy
(0,20)*+{U}="A";
(20,20)*+{X}="B";
(0,0)*+{V}="C";
(20,0)*+{Y}="D";
{\ar@{->}^{f}"A";"B"};
{\ar@{->}_{\alpha}"A";"C"};
{\ar@{->}^{p}"B";"D"};
{\ar@{->}^{g}"C";"D"};
\endxy
\]  

Then the following are equivalent. 
\begin{itemize}[label=$-$]
\item There is a lifting in the commutative square above i.e there exists $k: V \to X$  such that: $k \circ \alpha =f$, $p \circ k=g$.
\item There is a lifting in the following square of $\M^{[\1]}$.
\[
\xy
(0,20)*+{\alpha}="A";
(20,20)*+{p}="B";
(0,0)*+{\Id_V}="C";
(20,0)*+{\ast}="D";
{\ar@{->}^{\theta}"A";"B"};
{\ar@{->}_{\av}"A";"C"};
{\ar@{->}_{}"C";"D"};
{\ar@{->}_{}"B";"D"};
\endxy
\]  

That is, there exists $\beta=(k,l): \Id_V \to p$ such that $\beta \circ \av= \theta$. 
\end{itemize}
\end{prop}

Using that proposition, the adjunction, and the fact that trivial fibrations are the $\I$-injective maps; we get the following. 
\begin{lem}\label{k-inj-cosegal}
Let $\F$ be an object of $\msetsu$. Then the following hold.
\begin{enumerate}
\item $\F$ is $\kb_{\msetsu}$-injective if and only if for every $s=(A_0,...,A_n)$ the map $$\F(u_s): \F(A,B) \to \F(s),$$ is a trivial fibration in $\M$. In particular $\F$ is a co-Segal category.
\item Every  a strict $\M$-category $\F$  is $\kb_{\msetsu}$-injective.
\end{enumerate}
\end{lem}

\begin{proof}
If $\F$ is $\kb_{\msetsu}$-injective, by definition, $\F$ is $\Psi_s(\av)$-injective for all generating cofibration in $\M$. And by adjunction we find that $\F(u_s)$ is $\av$-injective thanks to the previous proposition. This is equivalent to saying that any lifting problem defined by $\alpha$ and $\F(u_s)$ has a solution. Consequently $\F(u_s)$ has the RLP with respect to all maps in $\I$ and we find that $\F(u_s)$ is a trivial fibration as claimed. This proves Assertion $(1)$.

Assertion $(2)$ is a corollary of Assertion $(1)$ since categories are the constant lax diagrams, therefore $\F(u_s)$ is an identity, in particular a trivial fibration.
\end{proof}

\subsection{Localization of $\msxsueasy$}
\subsubsection{Preliminary observations}
\paragraph{A pushforward lemma}
\begin{lem}
Let $f:X \to Y$ be a function and $\sigma: \F \to \G$ be a morphism in $\msxsu$ such that for every pair $(A,B) \in X^2$ the component 
$$\sigma_{(A,B)}: \F(A,B) \to \G(A,B),$$
is a (trivial) cofibration in $\M$.\\

Then for every pair $(C,D) \in Y^2$ the component 
$$\fex(\sigma)_{(A,C)}: (\fex \F)(C,D) \to (\fex \G)(C,D),$$
is also a (trivial) cofibration in $\M$.
\end{lem}
\begin{proof}
We start by showing that:
\begin{claim}
 For any object $\F \in \msxsu$, we have  
$$(\fex \F)(C,D)= \coprod_{(A,B) \in f^{-1}(C) \times f^{-1}(D)} \F(A,B).$$
\end{claim}

To see this, we simply use the fact that for a classical diagram $F: \J \to \M$  indexed by a directed category having an initial object $e$, given any morphism $m \to F(e)$ we can define a new diagram $F': \J \to \M$ with $F'(e)=m$ and $F'(j)= F(j)$.\\

The same thing happens for lax diagrams indexed by locally direct $2$-categories like $\sxop$. But in the unital case we have to choose our map $m \to F(e)$ so that if $F(e)$ is pointed then so is $m$.\\

With that observation in mind, assume that $\fex \F$ is constructed. The unit of the adjunction $\fex \dashv \fstar$ gives a map 
$$\F(A,B) \to (\fex\F)(C,D), \quad \forall (A,B) \in f^{-1}(C) \times f^{-1}(D).$$  

And the universal property of the coproduct says that there is a unique map $$\varepsilon: \coprod_{(A,B) \in f^{-1}(C) \times f^{-1}(D)} \F(A,B) \to (\fex\F)(C,D),$$
with the obvious factorizations. Note that the coproduct is pointed when $C=D$ since in that case all $\F(A,A)$ are also pointed. So we can modify $\fex \F$ into a new object $\Ea$ of $\msysu$ by only changing the value of $(C,D)$ 
$$\Ea(C,D)= \coprod_{(A,B) \in f^{-1}(C) \times f^{-1}(D)} \F(A,B).$$

Doing this for all $(C,D)$  we find that $\Ea$ is canonically equipped with a unique map $\varepsilon:\Ea \to \fex\F$ and another map $\iota: \F \to \fstar \Ea$ such that the unit $\F \to \fstar \fex \F$ is $ \fstar(\varepsilon) \circ \iota$. This shows that $\Ea$ is as much universal as $\fex \F$ so they must be isomorphic.\\

We have a similar formula for the component of $\fex \sigma$ and thanks to the fact that (trivial) cofibrations are closed under coproduct, we get our lemma. 

\end{proof}

\paragraph{A key lemma}
The following lemma says that the left Quillen functor $$|-|: \msxsueasy \to \mcatx,$$ sends elements in $\kb_{\msxsu} $ to trivial cofibration, so in particular to weak equivalences.
\begin{lem}\label{lem-excellent}
Let $s=(A,...,B)$ be a $1$-morphism in $\sx$ and let $\alpha: U \to V$ be a generating cofibration  of $\M$.  Then the image in $\mcatx$ of the map $\Psi_s(\av): \Psi_s(\alpha) \to \Psi_s(\Id_V)$ by the functor
$$|-|: \msxsu \to \mcatx,$$
is a trivial cofibration in the local model structure on $\mcatx$. In particular it's a trivial cofibration in the local  model structure on $\mcat$. 
\end{lem}

\begin{proof}
We're going to show that $|\Psi_s(\av)|$ has the left lifting property with respect to any local fibration in $\mcat$. In fact it has the LLP with respect to any functor in $\mcat$.\\

Let $p: \Za \to \T$ be a local fibration in $\mcat$. Without loss of generality we can assume that $p$ is a morphism in $\mcatx$, since the inclusion $\mcatx \hookrightarrow \mcat$ is compatible with the respective homotopy theories. By adjunction, a lifting problem in $\mcatx$
\[
\xy
(0,20)*+{|\Psi_s(\alpha)|}="A";
(30,20)*+{\Za}="B";
(0,0)*+{|\Psi_s(\Id_V)|}="C";
(30,0)*+{\T}="D";
{\ar@{->}^{}"A";"B"};
{\ar@{->}_{|\Psi_s(\av)|}"A";"C"};
{\ar@{->}^{}"B";"D"};
{\ar@{->}^{}"C";"D"};
\endxy
\]  
is equivalent to a lifting problem in $\msxsu$:
\[
\xy
(0,20)*+{\Psi_s(\alpha)}="A";
(30,20)*+{\iota(\Za)}="B";
(0,0)*+{\Psi_s(\Id_V)}="C";
(30,0)*+{\iota(\T)}="D";
{\ar@{->}^{}"A";"B"};
{\ar@{->}_{\Psi_s(\av)}"A";"C"};
{\ar@{->}^{}"B";"D"};
{\ar@{->}^{}"C";"D"};
\endxy
\]

And one of them has a solution if and only if the other one has a solution. Again, using the adjunction 
$$\Psi_s: \M^{[1]} \leftrightarrows \msxsu: \Ev_{u_s},$$ the previous lifting problem in $\msxsu$ is equivalent to the following one in $\M^{[1]}$.  
\[
\xy
(0,20)*+{\alpha}="A";
(30,20)*+{\Za(u_s)= \Id_{\Za(A,B)}}="B";
(0,0)*+{\Id_V}="C";
(30,0)*+{\T(u_s)=\Id_{\T(A,B)}}="D";
{\ar@{->}^{}"A";"B"};
{\ar@{->}_{\av}"A";"C"};
{\ar@{->}^{}"B";"D"};
{\ar@{->}^{}"C";"D"};
\endxy
\]

As shown above $\Za(u_s)= \Id_{\Za(A,B)}$, and similarly for $\T$ since they are categories (therefore locally constant). The morphism $\alpha \to \Id_{\Za(A,B)}$ is simply given by two maps $f: U \to \Za(A,B)$ and $g: V \to \Za(A,B)$ with a factorization $f= g \circ \alpha$.  The map $\Id_V \to \Id_{\T(A,B)}$ is simply a map $h: V \to \T(A,B)$, and we have two equalities $p \circ f=  h \circ \alpha$ and $h= p \circ g$. \\

Clearly the map $\Id_V \to \Id_{\Za(A,B)}$ given by $g$ is a lifting to our problem and  $|\Psi_s(\alpha)|$ is a trivial cofibration as desired.
\end{proof}
Let $\coscan: \msxsu \to \msxsu$ be the $\kb_{\msxsu}$-injective replacement functor obtained by the small object argument. We have a natural transformation $\eta_{can}: \Id \to \coscan$ whose component $\F \to \coscan(\F)$ is a $\kb_{\msxsu}$-cell complex.\\
 
As a consequence of the previous lemma we get:
\begin{prop}
If a local model structure on $\mcatx$ exists then the image  
$$|\eta_{can}|: | \F| \to |\coscan(\F)|,$$ is a trivial cofibration of strict $\M$-categories.
\end{prop}

\subsubsection{Enlarging the cofibrations}

In the following we would like to have a left proper model structure $\msxsu$ such that the set of generating cofibrations contains the localizing set $\kb_{\msxsu}$ introduced above.\\
Denote by $\I_{\msxsu}^{+}$ the set
$$\I_{\msxsu} \bigsqcup \kb_{\msxsu}.$$

\begin{lem}\label{petit-lem-cofib}
Given any pair $(A,B) \in X^{2}$, for all $\sigma \in \I_{\msxsu}^{+}$ the component  $\sigma_{(A,B)}$ is a cofibration.
\end{lem}
\begin{proof}
The statement is clear if $\sigma \in \I_{\msxsu}$. If $\sigma \in \kb_{\msxsu}$, the component $\sigma_{(A,B)}$ is either $\alpha$, $\Id_I$ or $\Id_I \coprod \alpha$ with $\alpha \in \I$ (see Proposition \ref{prop-av}).
\end{proof}
\paragraph{The model structure}
We show below that there is a left proper combinatorial model structure on $\msxsu$ with $\I_{\msxsu}^{+}$  as the set of generating cofibrations and $\W_{\msxsueasy}$ as the class of weak equivalences.\\

We use Smith's recognition Theorem for combinatorial model categories (see for example Barwick \cite[Proposition 2.2]{Barwick_localization}). This theorem allows constructing a combinatorial model category out of two data consisting of a class $\W$ of morphisms whose elements are called \emph{weak equivalences}; and a set $\I$ of \emph{generating cofibrations}.\\

Our method is classical and the argument is present in Pellissier's PhD thesis \cite{Pel}; it is also used by Lurie \cite{Lurie_HTT}, Simpson \cite{Simpson_HTHC} and others. But in doing so, we actually reprove (implicitly) a derived version of Smith's theorem that has been outlined by Lurie \cite[Proposition A.2.6.13]{Lurie_HTT}. This version asserts that the resulting combinatorial model structure is automatically left proper. So we will just use that proposition  that we recall hereafter with the same notation as in Lurie's book. 
\begin{prop}\label{Smith-Lurie}
Let $\bf{A}$ be a presentable category. Suppose we are given a class $W$ of morphisms of A, which we will call weak equivalences, and a (small) set $C_0$ of morphisms of $\bf{A}$, which we will call generating cofibrations. Suppose furthermore that the following assumptions are satisfied:
\begin{itemize}
\item[$(1)$] The class $W$ of weak equivalences is perfect (\cite[Definition A.2.6.10]{Lurie_HTT}). 
\item[$(2)$] For any diagram
\[
\xy
(0,20)*+{X}="A";
(20,20)*+{Y}="B";
(0,0)*+{X'}="C";
(20,0)*+{Y'}="D";
%%%%%%%
(0,-20)*+{X''}="X";
(20,-20)*+{Y''}="Y";
{\ar@{->}^-{f}"A";"B"};
{\ar@{->}_-{}"A";"C"};
{\ar@{->}^-{}"B";"D"};
{\ar@{->}^-{}"C";"D"};
{\ar@{->}^-{}"X";"Y"};
{\ar@{->}^-{g}"C";"X"};
{\ar@{->}^-{g'}"D";"Y"};
%%%%%%
\endxy
\] 

in which both squares are coCartesian (=pushout square), $f$ belongs to $C_0$, and $g$ belongs $W$, the map $g'$ also below to $W$.
\item[$(3)$] If $g: X \to Y$ is a morphism in $\bf{A}$ which has the right lifting property with respect to every morphism in $C_0$, then $g$ belongs to $W$.
\end{itemize}

Then there exists a left proper combinatorial model structure on $\bf{A}$ which
may be described as follows:
\begin{itemize}
\item[$(C)$] A morphism $f : X \to  Y$ in $\bf{A}$ is a cofibration if it belongs to the weakly
saturated class of morphisms generated by $C_0$.
\item[$(W)$]  A morphism $f:X\to Y$ in $\bf{A}$ is a weak equivalence if it belongs to $W$.
\item[$(F)$] A morphism $f:X\to Y$ in $\bf{A}$ is a  fibration if it has the right lifting property with respect to every map which is both a cofibration and a weak equivalence.
\end{itemize}
\end{prop}
\begin{note}
Here \emph{perfectness} is a property of stability under filtered colimits and a generation by a small set $W_0$ (which is more often the intersection of $W$ and the set of maps between presentable objects). The reader can find the exact definition in \cite[Definition A.2.6.10]{Lurie_HTT}.
\end{note}

\begin{warn}
We've used so far the letters $f, g$ as functions so to avoid any confusion we will use $\sigma,\sigma'$ instead.
\end{warn}

Applying the previous proposition we get the following theorem.
\begin{thm}\label{enlarging-msx}
Let $\M$ be a combinatorial monoidal model category which is left proper. Then for any set $X$ there exists a combinatorial model structure on $\msxsu$ which is left proper and having the following properties. 
\begin{enumerate}
\item A map $\sigma: \F \to \G$ is a weak equivalence if it's an easy weak equivalence i.e, if it's in $\W_{\msxsueasy}$.
\item A map $\sigma: \F \to \G$  is a cofibration if it belongs to the weakly
saturated class of morphisms generated by $ \I_{\msxsu}^{+}$.
\item A morphism $\sigma: \F \to \G$ is a fibration if it has the right lifting property with respect to every map which is both a cofibration and a weak equivalence
%\item If a local model structure on $\mcatx$ exists then the adjunction $$ |-|: \msxsu \leftrightarrows \mcatx: \iota, $$
%remains a Quillen adjunction.
\end{enumerate}
We will denote this model category  by $\msxsuplus$. The identity functor $$ \Id: \msxsueasy \to \msxsuplus, $$ is a left Quillen functor.
\end{thm}
\begin{proof}
Condition $(1)$ is straight forward because $\W_{\msxsueasy}$ is the class of weak equivalence in the combinatorial model category $\msxsueasy$. We also have Condition $(3)$ since a map $\sigma$ in $\I_{\msxsu}^{+}\tx{-inj}$ is in particular in $\I_{\msxsu}\tx{-inj}$, therefore it's a trivial fibration in $\msxsueasy$ and thus an easy weak equivalence.\\

It remains to check that Condition $(2)$ is also satisfied. Consider the following diagram as in the proposition.
\[
\xy
(0,10)*+{\F}="A";
(20,10)*+{\G}="B";
(0,0)*+{\F'}="C";
(20,0)*+{\G'}="D";
%%%%%%%
(0,-10)*+{\F''}="X";
(20,-10)*+{\G''}="Y";
{\ar@{->}^-{\sigma}"A";"B"};
{\ar@{->}_-{}"A";"C"};
{\ar@{->}^-{}"B";"D"};
{\ar@{->}^-{}"C";"D"};
{\ar@{->}^-{}"X";"Y"};
{\ar@{->}^-{\theta}"C";"X"};
{\ar@{->}^-{\theta'}"D";"Y"};
%%%%%%
\endxy
\]

 If $\sigma: \F \to \G$ is in $\I_{\msxsu}^{+}$, we have from Lemma \ref{petit-lem-cofib} that each top-component 
$$\sigma_{(A,B)}: \F(A,B) \to \G(A,B),$$ 
is a cofibration in $\M$. Now as mentioned several times in the paper, colimits in $\msxsu$ are computed level-wise at each $1$-morphisms $(A,B)$. It follows that the top components in that diagram are obtained by pushout in $\M$; and since $\M$ is left proper we get that every top-component $\theta'_{(A,B)}$ is a weak equivalence, which means that $\theta'$ is an easy weak equivalence as desired.\\
\end{proof}

\subsubsection{Changing the set of objects}
Let $f: X \to Y$ be a function and consider the two model categories $\msxsuplus$ and $\msysuplus$. 
\begin{prop}\label{prop-adjunction-av-gen}
We have a Quillen adjunction
$$\fex: \msxsuplus \leftrightarrows \msysuplus: \fstar.$$
\end{prop} 

\begin{proof}
It's enough to show that we have the following inclusions
$$\fex(\I_{\msxsu}^{+}) \subseteq \I_{\msysu}^{+}, \quad \fex[\cof(\I_{\msxsu}^{+}) \cap \W] \subseteq \cof(\I_{\msysu}^{+}) \cap \W.$$ 

Using the various adjunctions, one has that for every $s$ there is an isomorphism
$$\fex [\Psi_s(\av)] \cong  \Psi_{f(s)}(\av).$$

This means that  $\fex( \kb_{\msxsu}) \subseteq \kb_{\msysu}$ and we get the first inclusion since from the old Quillen adjunction we have the inclusion
$$\fex(\I_{\msxsu}) \subseteq \I_{\msysu}.$$

For the second inclusion it's suffices to observe for any $\sigma \in \cof(\I_{\msxsu}^{+}) \cap \W$, we know already that $\fex \sigma \in \cof(\I_{\msysu}^{+})$ so it remains to show that $\fex \sigma$ is an easy weak equivalence. But for this one simply remembers that every top component $\sigma_{(A,B)}$ is a trivial cofibration and thanks to Lemma \ref{petit-lem-cofib}, we deduce that every component $\fex (\sigma)_{(A,B)}$ is also a trivial cofibration and in particular a weak equivalence.
\end{proof}

\subsubsection{The localized model category}

Let $\Sim$ be a \emph{$\kb_{\msxsu}$-localization functor} obtained by the small object argument and denote by $\eta: \Id \to \Sim$ the induced natural transformation. We refer the reader to Hirschhorn \cite{Hirsch-model-loc} for a description of such functor.
\begin{thm}\label{main-thm-1}
Let $\M$ be a combinatorial monoidal model category which is left proper. Then for any set $X$ there exists a combinatorial model structure on $\msxsu$ which is left proper and having the following properties. 
\begin{enumerate}
\item A map $\sigma: \F \to \G$ is a weak equivalence if and only if the induced map $$\Sim(\sigma):\Sim(\F) \to \Sim(\G),$$ is a level-wise weak equivalence. 
\item A map $\sigma: \F \to \G$ is cofibration if it's a cofibration in $\msxsuplus$.
\item Any fibrant object $\F$ is a co-Segal category.
%\item If a local model structure on $\mcatx$ exists then the adjunction $$ |-|: \msxsu \leftrightarrows \mcatx: \iota, $$ descends to a Quillen equivalence.
\end{enumerate}
We will denote this model category  by $\msxsuc$. The identity functor $$ \Id: \msxsuplus \to \msxsuc, $$ is a left Quillen functor.\\ 

This model structure is the left Bousfield localization of $\msxsueasy$ with the respect to the set $\kb_{\msxsu}$. \end{thm}
\begin{df}
Define a \emph{co-Segalification functor} for $\msxsu$ to be any fibrant replacement functor in $\msxsuc$.
\end{df}

\begin{proof}[Proof of Theorem \ref{main-thm-1}]
The existence of the left Bousfield localization and the left properness is guaranteed by Smith's theorem on left Bousfield localization for combinatorial model categories. We refer the reader to Barwick \cite[Theorem 4.7]{Barwick_localization} for a precise statement. This model structure is again combinatorial.\\

For the rest of the proof we will use the following facts on Bousfield localization and the reader can find them in Hirschhorn's book \cite{Hirsch-model-loc}. 
\begin{enumerate}
\item A weak equivalence in $\msxsuc$ is a \emph{$\kb_{\msxsueasy}$-local weak equivalence}; we will refer them as \emph{new weak equivalence}. And any easy weak equivalence (old one) is a new weak equivalence.
\item New cofibrations are the same as the old ones and therefore trivial fibrations are just the olds ones too. In particular trivial fibrations are easy weak equivalences. 
\item Fibrant objects are the  $\kb_{\msxsu}$-local objects that are fibrant in the original model structure.
\item Every map in $\kb_{\msxsu}$ becomes a weak equivalence in $\msxsuc$, therefore an isomorphism in the homotopy category.
\end{enumerate}

Let $\F$ be a fibrant object in $\msxsuc$, this means that the unique map $\F \to \ast$ has the RLP with respect to any trivial cofibration. But since every element in $\kb_{\msxsu}$ was an old cofibration and become a weak equivalence and therefore a new trivial cofibration. So we find that $\F$ must be $\kb_{\msxsu}$-injective. Therefore $\F$ is a co-Segal category thanks to Lemma \ref{k-inj-cosegal}; this gives Assertion $(3)$.\\

By definition of the functor $\Sim$, for every $\F$, $\Sim(\F)$ is automatically fibrant in the new model structure, therefore by the previous argument $\Sim(\F)$ is a co-Segal category for all $\F$. Now it's classical that in the localized category, a map $\sigma$ is a weak equivalence if and only if any localization of $\sigma$ is an old weak equivalence.

It follows that $\sigma$ is a new weak equivalence if and only if $\Sim(\sigma)$ is an easy weak equivalence; but easy weak equivalence between co-Segal categories is the same thing as a level-wise weak equivalence and Assertion $(1)$ follows. 
\end{proof}
\begin{thm}\label{strict-thm}
For any set $X$ and any $\F \in \msxsu$, the canonical map 
$\F \to |\F|$ is an equivalence in $\msxsuc$.   
\end{thm}
\begin{proof}
Every element $\Psi_s(\av) \in \kb_{\msxsu}$ becomes a weak equivalence in $\msxsuc$, and every easy weak equivalence is a new weak equivalence, therefore every $\Psi_s(\xi_{\alpha})$ becomes a weak equivalence too (Lemma \ref{lem-cell-complex-av}).

Now thanks to Proposition \ref{prop-av}, we know that every such $\Psi_s(\xi_{\alpha})$ is an old cofibration, therefore it's a new trivial cofibration. In particular any pushout of $\Psi_s(\xi_{\alpha})$ is a trivial cofibration and in particular a weak equivalence. The theorem follows from  Assertion $(3)$ of Proposition \ref{pdef-deux-constant}.
\end{proof}

\begin{rmk}
Let $F: \A \to \B$ be a left Quillen functor between model categories and let $K$ be a set of maps in $\A$. If the left Bousfield localization of $\A$ (resp. $\B$)  with respect to $K$ (resp.  $F(K)$) exists then then there is an induced left Quillen functor 
$$F^+: \A^+ \to \B^+,$$
where $\A^+$ and $\B^+$ are the respective Bousfield localizations.
\end{rmk}
Applying this remark to our previous Quillen adjunction $\fex \dashv \fstar$ we get:
\begin{prop}
For any function $f:X \to Y$ we have an induced Quillen adjunction 
$$\fex: \msxsuc \leftrightarrows \msysuc: \fstar.$$
\end{prop}

\subsection{Fiber wise localized model structure on $\msetsu$.}
Let $\Set_{min}$ be the model structure on $\Set$ where the weak equivalences are the isomorphisms, and every map is a cofibration (resp. fibration). 
\begin{df}
Let $\F \in \msxsu$, $\G \in \msysu$ be objects of $\msetsu$  and let $f:X \to Y$ a function. Say that a map $(\sigma, f): \F \to \G$ is: 
\begin{itemize}[label=$-$]
\item a strong weak equivalence if $f : X \to Y$ is an isomorphism and the map $\sigma: \F \to \fstar \G$ is a weak equivalence in $\msxsuc$;
\item a strong cofibration if $\fex \F \to \G$ is a cofibration in $\msysuc$;
\item a strong fibration if $\sigma: \F \to \fstar \G$ is a fibration in $\msxsuc$.
\end{itemize}
\end{df}

Clearly if $f: X \to Y$ is an isomorphism, then for any $\F$, the map $\F \to \fstar \fex \F$ is an isomorphism in particular a weak equivalence in $\msxsuc$. In fact an isomorphism of sets $f: X \to Y$ induces an isomorphism of (model) categories $f: \msxsuc \to \msysuc$; so in particular it preserves and reflects weak equivalences. Since we also have a Quillen pair $(\fex, \fstar)$ for every $f$, we can apply Stanculescu-Roig's theorem \cite{Stanculescu_bifib} and obtain the following result. 
\begin{thm}\label{fibred_localized}
Let $\M$ be a combinatorial monoidal model category which is left proper. Then  there exists a model structure on $\msetsu$ which has  the following properties. 
\begin{enumerate}
\item A map $\sigma: \F \to \G$ is a weak equivalence if it's a strong weak equivalence.
\item A map $\sigma: \F \to \G$ is a fibration if it's a strong fibration.
\item A map $\sigma: \F \to \G$ is a cofibration if it's a strong cofibration.
\item  Any fibrant object $\F$ is a co-Segal category.
\end{enumerate}
We will denote this model category  by $\msetsuc_{,fib}$. 
\end{thm}

\begin{proof}
The existence of the model structure and characterization of the three classes of maps is given by Stanculescu-Roig's theorem. 

If $\F \in \msxsuc$ is fibrant in that model category, then by definition the unique map $\F \to \ast$ going to the terminal object is a fibration; which means that the map $\F \to \fstar (\ast)$ is a fibration in $\msxsuc$. But $\fstar (\ast)$ is the terminal object in $\msxsuc$, therefore $\F$ is fibrant in $\msxsuc$. And thanks to Theorem \ref{main-thm-1}, we know that $\F$ is a co-Segal category. 
\end{proof}

\section{Tensor product of co-Segal categories}\label{sec-monoidal-coseg}
\subsection{The monoidal category $(\msetsu, \otimes_{\S}, \Un )$ }
Given a small category $\Ca$, by construction there is a degree (or length) strict $2$-functor $\degb : \S_{\Ca} \to \S_{\1}$ where $\1$ is the unit category and $\S_{\1} \cong (\Depi, +,\0)$.  If $\D$ is another category we can form the genuine fiber product of $2$-categories $\S_{\Ca} \times_{\S_{\1}} \S_{\D}$. 

\begin{prop}
There is an isomorphism of $2$-categories: $ \S_{\Ca \times \D} \cong \S_{\Ca} \times_{\S_{\1}} \S_{\D}$. In particular for any sets $X$,$Y$ we have
$$\sx\times_{\S_{\1}} \sy \cong \S_{\ol{X \times Y}} \cong \S_{\ol{X} \times \ol{Y}}.$$
\end{prop}

\begin{proof}
Exercise
\end{proof}
\subsubsection{Tensor product of $\S$-diagrams}
Let $\M=(\ul{M}, \otimes , I)$ be a symmetric monoidal (model) category. Let $\Un: (\Depi, +)^{op} \to \M$ be the discrete category with a single element i.e, $\Un= I_{\ast}$.\\

Given $\Fa: (\S_{\Ca})^{\tx{$2$-op}}  \to \M$   and $\Ga:(\S_{\D})^{\tx{$2$-op}}  \to \M$ we define 
$\Fa \otimes_{\S} \Ga: (\S_{\Ca \times \D})^{\tx{$2$-op}} \to \M $ to be the lax functor described as follows. 
\renewcommand{\theenumi}{\arabic{enumi}}
\begin{enumerate}
\item For a $1$-morphism $(s,s') \in (\S_{\Ca \times \D})$ we set $(\Fa \otimes_{\S} \Ga) (s,s'):= \Fa(s) \otimes \Ga(s')$,
\item The laxity map $\varphi_{\Fa \otimes_{\S} \Ga}: (\Fa \otimes_{\S} \Ga) (s,s') \otimes (\Fa \otimes_{\S} \Ga) (t,t') \to (\Fa \otimes_{\S} \Ga) (s \otimes t,s' \otimes t')$ is obtained as the composite:
$$\Fa(s) \otimes \Ga(s') \otimes \Fa(t) \otimes  \Ga(t')\xrightarrow{ \Id \otimes sym \otimes \Id} \Fa(s) \otimes \Fa(t) \otimes \Ga(s') \otimes  \Ga(t') \xrightarrow{\varphi_{\Fa} \otimes \varphi_{\Ga}} \Fa(s\otimes t) \otimes \Ga(s' \otimes t')$$
where $sym$ is the symmetry isomorphism in $\M$ (we have $sym:\Ga(s') \otimes \Fa(t) \xrightarrow{\cong}  \Fa(t) \otimes \Ga(s')$).
\item One easily sees that if $f: \Ca' \to \Ca  $ and $g: \D' \to \D$ then $(f\times g)^{\star} \Fa \otimes_{\S} \Ga \cong \fstar \Fa \os \gstar \Ga$.
\item If $\sigma=(\sigma,f) \in \Hom_{\mset}(\Fa, \Ga)$ and $\gamma=(\gamma,g) \in \Hom(\Fa',\gstar \Ga')$ we define 
$$\sigma \os \gamma=(\sigma \otimes \gamma, f \times g)   \in \Hom_{\mset}[\Fa \os \Ga, \Fa' \os \Ga']$$ to be the morphism whose component at $(s,s')$ is $\sigma_s \otimes \sigma_{s'}$. 
\end{enumerate}
\ \\
We leave the reader to check that:
\begin{enumerate}
\item $\os$ is a bifunctor and is associative,
\item we have a canonical symmetry: $\Fa \os \Ga \cong \Ga \os \Fa$, 
\item for any $\Fa$ we have a natural isomorphism $\Fa \os \Un \cong  \Fa$.
\end{enumerate}
From the previous discussion one has:
\begin{prop}
For any symmetric monoidal category $\M$, we have a symmetric monoidal category $(\msetsu, \os, \Un)$.
\end{prop}

\begin{rmk}
Although this tensor product is natural, it turns out that this is not the correct one for homotopy theory purposes. But for the moment we will continue our discussion with it.
\end{rmk}

Recall that given $\F \in \msxsu$, a morphism of unital precategories $\Un \to \F$ is the same thing as a function $1 \to X$, that is an element $A \in X$. Therefore the comma category $\Un \downarrow \msetsu$ is the category of \emph{marked precategories}. Instead of using a specific letter $A$ we will use the generic notation $\ast$ for the selected object.  We will write $(\F,\ast)$ for a marked precategory $\F$.

\begin{rmk}\label{mon-struct-comma}
Since $\Un \os \Un \cong \Un$, as in any monoidal category ,we have an induced monoidal structure on $\Un \downarrow \msetsu$  that takes $\Un \to \F$ and $\Un \to \G$ to:
$$\Un \cong \Un \os \Un  \to \F \os \G.$$

We have a natural fibred category $p:\Un \downarrow \msetsu \to \Set_{\ast}$, where $\Set_{\ast}$ is the category of pointed sets.
\end{rmk}
\begin{df}
Define a monoidal co-Segal precategory $\F$, as a unital $\msetsu$-precategory with one object, that is a normal lax functor
$$\F: (\Depiop,+,0) \to  (\msetsu, \os, \Un),$$
that is unital.

Say that $\F$ is \emph{a monoidal co-Segal $\M$-category} if $\F(\1) \in \mset$ is a co-Segal $\M$- category and for every $\n$, the map $\F(\1) \to \F(\n)$ is a weak equivalence of precategories in $\msetsu$ 
\end{df}
The definition is equivalent to saying that $\F$ is a co-Segal monoid of $(\msetsu, \os, \Un)$ such that $\F(\1)$ is a co-Segal $\M$-category e.g, a fibrant object in the previous model structure on $\msetsu$. 

\begin{rmk}
The previous definition is somehow too strong since the weak equivalences in the fibred model structure must have an isomorphism on the set of objects. But philosophically this is not an issue because of the hypothetical strictification theorem.
\end{rmk}

Given a function $f: 1 \to X$, we have a pullback functor $\fstar: \msxsu \to \M_{\S}([\1])_{su}$ that takes $\F$ to $\fstar(\F)$. Using these various functors, one has:
\begin{prop}
With the previous notation the following hold. 
\begin{enumerate}
\item There is a functor 
$$\End(\ast):\Un \downarrow \msetsu \to \M_{\S}([\1])_{su},$$
that takes $f: \Un \to \F$ to $\fstar \F$. 
This functor sends fibrant object to fibrant object with the respective localized model structures.
\item The functor $\End(\ast)$ extends to a monoidal functor.
\item The inclusion $\iota:\mcat \hookrightarrow \msetsu$ is a monoidal functor. 
\end{enumerate}
\end{prop}
\begin{proof}[Sketch of Proof]
Assertion $(1)$ is clear. For any function $f$, $\fstar$ is a right Quillen functor therefore it sends fibrant object to fibrant object.\\

Assertion $(2)$ is a consequence of Remark \ref{mon-struct-comma} together with the fact that
$$(f\times g)^{\star} (\Fa \otimes_{\S} \Ga) \cong \fstar \Fa \os \gstar \Ga.$$
For Assertion $(3)$, one observes that given a monoidal $\M$-category $(\C, \otimes, U)$ with unit $U$, then by definition $U$ is a morphism $\Un \to \C$. Finally it's not hard to see that given two strict $\M$-categories, we have:
$$\iota(\C \boxtimes \D) \cong \iota(\C) \os \iota(\D),$$
where $\boxtimes$ is the tensor product in $\mcat$. 
\end{proof}

\begin{df}
Define a co-Segal $2$-monoid in $\M$ to be a lax diagram
 $$\F: (\Depiop,+,0) \to (\M_{\S}([\1])_{su}, \os, \Un),$$
that is unital and such that $\F(\1)$ satisfies the co-Segal conditions and for every $\n$ the map 
$\F(\1) \to \F(\n)$ is a weak equivalence of precategories.   
\end{df}

As a corollary of the proposition we get
\begin{cor}\label{cor-deligne}
Let $\F \in  \msxsu$ be a monoidal co-Segal $\M$-category. Then  $f:\End(\ast)= \fstar \F$ is a co-Segal $2$-monoid in $\M$.
\end{cor}

\begin{proof}
$\F$ is canonically marked by its unit object $\Un \to \F$; so we get an object of $\Un \downarrow \msetsu$. $\End(A)$ is by definition the composite
$$(\Depiop,+,0) \xrightarrow{\F} (\Un \downarrow\msetsu, \os, \Id_{\Un}) \xrightarrow{\End(\ast)} (\M_{\S}([\1])_{su}, \os, \Un).$$  

Each of this functor is a (lax) monoidal functor therefore their composite is also a lax monoidal functor which is unital. 
\end{proof}

\subsubsection{Monoidal category by homotopy transfer}
Let $\C$ be a strict $\M$-category. For any pair $(A,B)$ of objects of $\C$, choose an object $\tld{\C}(A,B)$ together with a weak equivalence $\epsilon:\tld{\C}(A,B) \xrightarrow{\sim} \C(A,B)$. Assume that each unity $I \to \C(A,A)$ factorizes through the map $\epsilon: \tld{\C}(A,A) \to \C(A,A)$. We showed in Lemma \ref{lem-deux-constant} that we have a co-Segal category $\tld{\C}$ whose \emph{underlying category} is just $\C$. Moreover the various map $\epsilon$ define a canonical weak equivalence of co-Segal $\M$-categories:
$$\epsilon: \tld{\C} \xrightarrow{\sim} \C.$$  

We are in a situation of a homotopy transfer for monoids, therefore we can establish that:
\begin{lem}\label{transfer-monoidal}
With the above notation, if $\C$ is a monoidal $\M$-category, then $\tld{\C}$ is a co-Segal monoidal $\M$-category.
\end{lem}

\begin{proof}
This is a general principle of homotopy transfer in any monoidal category with weak equivalence. The proof is the same as the one of Lemma \ref{lem-deux-constant} when we replace $\M$ by $\msetsu$ and consider a single object.\\

One shows that indeed $\tld{\C}$ is a co-Segal monoid of $\msetsu$ where the \emph{quasi-multiplication} is given by the following zig-zag.

$$\tld{\C} \os \tld{\C} \xrightarrow{\boxtimes \circ (\epsilon \os \epsilon)} \C \xleftarrow[\sim]{\epsilon} \tld{\C}.$$  
\end{proof}
%%%%%%%%%%%%%%%%%%%%%%%%%%%ù Jusqu'ici ok%%%%%%%%%%%%%%%%%%%%%%%%%%%%%%%%%%%

\section{Natural transformations of co-Segal functors}\label{sec-intern-hom}
The main purpose in this section is to define an \emph{internal hom} for co-Segal precategories. This internal hom doesn't exist in an obvious way as we would expect. In fact for many cases there are two ways of establishing the existence of an internal hom in a monoidal category $(\A, \otimes)$:
\begin{itemize}[label=$-$]
\item Either we construct explicitly the internal hom; or
\item $\A$ is locally presentable and for each object $U \in \A$ , $U \otimes -$ preserves colimits. Indeed the adjoint functor functor theorem for locally presentable categories will guarantee the existence of a right adjoint which will be the internal hom.
\end{itemize}
But in our case we are unable to prove that $\F \otimes_\S-$ preserves colimits. This is one of the reasons why this product doesn't seem to be the correct one. The only option left is to construct explicitly an object that will be an approximation of the internal hom we want.\\

Although unsatisfactory, it's not surprising to have theses obstacles. Indeed, co-Segal precategories are algebraic which means that every notion should involve some explicit \emph{equation} (commutative diagram). And since we are in \emph{higher category theory}  we do not write explicit equations; but instead everything is encoded in the properties and conditions that define the objects. But still we are going to mimic the existing constructions for strict categories to continue our discussion. We begin our discussion with the definition of natural transformation we will work with.\\

Recall that given a pair $(A,B)$ of objects in a precategory $\F$ we denote by $|\F|(A,B)$ the colimit of 
$$\F_{AB}: \sx(A,B)^{op} \to \M.$$

\begin{df}
Let $\F$ and $\G$ be objects of $\msxsu$ and $\msysu$ respectively. Let $$(\sigma_i;f_i): \F \to \G, \quad i\in \{1,2,...n \},$$ be $n$ morphisms from $\F$  to $\G$ with $f_i: X \to Y$.\\

A \emph{relative natural tranformation} $\eta : \sigma_1 \nrightarrow 
\cdots \sigma_i \nrightarrow \cdots \sigma_n$ through $\sigma_2, ..., \sigma_{n-1}$ is given by the following data and axiom.\\
\ \\
\textbf{\ul{Data:}} A morphism $\eta_A: I \to \G(f_1A,...,f_nA)$ for each $A \in X$.\\
\ 
\\
\textbf{\ul{Axiom:}} For every $1$-morphism $s=(A,...,B)$, the following commutes.
\[
\xy
(-30,10)*+{\F(s)}="A";
(0,18)*+{I \otimes \F (s)}="W";
(0,0)*+{\F(s) \otimes I}="X";
(30,0)*+{\G f_n(s) \otimes  \G \alpha_A}="Y";
(30,18)*+{\G \alpha_B \otimes \G f_1(s)}="E";
(60,18)*+{\G[\alpha_B \otimes f_1(s)]}="B";
(60,0)*+{\G[f_n(s) \otimes \alpha_A]}="D";
(95,10)*+{|\G|(f_1A,f_nB)}="F";
{\ar@{->}^-{\sigma_n \otimes \eta_A}"X";"Y"};
{\ar@{->}^-{\eta_B \otimes \sigma_1}"W";"E"};
{\ar@{->}^-{\cong}"A";"W"};
{\ar@{->}^-{\cong}"A";"X"};
{\ar@{->}^-{}"Y";"D"};
{\ar@{->}^-{}"E";"B"};
{\ar@{->}^-{}"B";"F"};
{\ar@{->}^-{}"D";"F"};
\endxy
\]

In the above diagram $\alpha_A=(f_1A,...,f_nA)$ and $\alpha_B=(f_1B,...,f_nB)$.
\end{df}

\begin{df}
Define the classifying object of relative natural transformations from $\sigma_1$ to $\sigma_n$ through  $\sigma_2, ..., \sigma_{n-1}$ to be the equalizer of the following diagram. 
$$\int_{\sigma_1}^{\sigma_n} \G \alpha  \to  \prod_{A \in \Ob(\sx)} \G \alpha_A \rightrightarrows \prod_{A \xrightarrow{s} B} \HOM_{\M}[\F (s),|\G|(f_1A,f_nB)]$$
The two parallel maps being induced by the adjoint transpose of the following ones. 
$$ (\prod_{A \in \Ob(\sx)} \G \alpha_A) \otimes  \F(s) \to \G \alpha_B \otimes  \F(s) \to \G[\alpha_B \otimes f_1(s)]   \to |\G|(f_1A,f_nB)]$$
$$  \F(s) \otimes  (\prod_{A \in \Ob(\sx)} \G \alpha_A)  \to \F(s) \otimes \G \alpha_A   \to \G[ f_n(s) \otimes \alpha_A]   \to |\G|(f_1A,f_nB)]$$

\end{df}
\begin{rmk}\label{rmk-internal-hom}
\begin{enumerate}
\item Given a precategory $\F: \sxop  \to \M$ and $3$ composable $1$-morphisms $r,s,t$ the \emph{associativity axiom} on $\F$ says that the two ways of going from 
$\F(r)\otimes \F(s) \otimes \F(t)$ to $\F (r \otimes s \otimes t)$, trough the various  laxity maps, are equal.
\item If $\G: \syop \to \M$ is another precategory and $(\sigma,f): \F \to \G$ is a morphism, then the \emph{compatibility axiom} on $\sigma$ preserve the above associativity; and the various ways of going from $\F(r)\otimes \F(s) \otimes \F(t)$ to $ \G f(r \otimes s \otimes t)$ are all equal. 
\item In particular the composite $\F \to \G \to |\G|$ obeys to the same rule; therefore all different ways of going from $\F(r)\otimes \F(s) \otimes \F(t)$ to $ |\G| f(r \otimes s \otimes t)$ are all the same. This can be summarized in a big commutative diagram that we choose not to include here.
\item Now observe that $|\G|$ is a locally constant diagram obtained by taking a colimit, therefore $|\G| f(r \otimes s \otimes t)$ doesn't depend on $r \otimes s \otimes t$; neither does $|\G| f(r \otimes s)$. It turns out that if we have a morphism $\gamma: \alpha_A \to \alpha'_A$ is a $2$-morphism of $\sxop$ then we have a commutative triangle
\[
\xy
(0,18)*+{ \F(s) \otimes \G \alpha_A}="W";
(0,0)*+{\F(s) \otimes \G \alpha'_A}="X";
%(40,0)*+{V}="Y";
(40,18)*+{ |\G|(f_1A,f_nB)]}="E";
%{\ar@{->}^-{\Id}"X";"Y"};
{\ar@{->}^-{}"W";"X"};
{\ar@{->}^-{}"W";"E"};
{\ar@{->}^-{}"X";"E"};
\endxy
\]
\end{enumerate}
\end{rmk}
A direct consequence of the above is that  by the same techniques as in \cite[Lemma 6.1]{Calc_lax}, one can establish the following.
\begin{prop}
Let $(\eta_1,\alpha)=\sigma_1 \nrightarrow \cdots \sigma_n$ and $(\eta_2,\beta):\sigma_n \nrightarrow \cdots \sigma_k$ be two relative transformations. 
\begin{enumerate}
\item Then there is an induced composite $(\eta_2 \otimes \eta_1,\beta \otimes \alpha)$ whose component at $A$ is:
$$ I\cong I \otimes I \to \G(f_1A,...,f_nA) \otimes \G(f_nA,...,f_kA)\to \G(f_1A,...,f_nA,...,f_kA)  $$

\item There exists a canonical map in $\M$:
$$ (\int_{\sigma_n}^{\sigma_k} \G \beta) \otimes ( \int_{\sigma_1}^{\sigma_n} \G \alpha)  \to \int_{\sigma_1}^{\sigma_k} \G (\beta \otimes \alpha).$$
\item These maps satisfy the `associativity' coherence condition i.e, the following commutes 
\[
\xy
(-60,10)*+{(\int_{\sigma_k}^{\sigma_p}\G \gamma)  \otimes (\int_{\sigma_n}^{\sigma_k} \G \beta) \otimes ( \int_{\sigma_1}^{\sigma_n} \G \alpha)}="A";
(0,18)*+{(\int_{\sigma_n}^{\sigma_p} \G \gamma \otimes \beta) \otimes ( \int_{\sigma_1}^{\sigma_n} \G \alpha)}="W";
(0,0)*+{(\int_{\sigma_k}^{\sigma_p} \G \gamma) \otimes ( \int_{\sigma_1}^{\sigma_k} \G \beta \otimes \alpha)}="X";
(50,10)*+{\int_{\sigma_1}^{\sigma_p}\G(\gamma \otimes \beta \otimes \alpha)} ="B";
{\ar@{->}^-{}"A";"W"};
{\ar@{->}^-{}"A";"X"};
{\ar@{->}^-{}"X";"B"};
{\ar@{->}^-{}"W";"B"};
\endxy
\]

\end{enumerate}
\end{prop}

Let $\F$ and $\G$ be two precategories as above and consider the set $[\F,\G]=\Hom(\F,\G)$. Let $(\S_{\ol{[\F,\G]}})^{2\tx{op}}$ be the associated $2$-category. Recall that a $1$-morphism of $(\S_{\ol{[\F,\G]}})^{2\tx{op}}$ is a sequence $(\sigma_1, ...., \sigma_n)$ where $\sigma_i=(\sigma_i, f_i)$. For each $A \in X$ we have an induced sequence $\alpha_A=(f_1A,...,f_nA)$ that corresponds to a $1$-morphism in $\syop$. A $2$-morphism $$(\sigma_1, ...., \sigma_n) \to (\sigma_1, ...., \sigma_n),$$
 in $(\S_{\ol{[\F,\G]}})^{2\tx{op}}$ is an operation that insert new transformations $\sigma$ to form a new sequence; and these insertions are governed by simplicial coface maps. 

If we denote by $\alpha'_A=(f_1A,...,f_nA)$  the new sequence, then we have an induced $2$-morphism $\alpha_A \to \alpha'_A$ in $\syop$. Thanks to Remark \ref{rmk-internal-hom} above and the universal property of the equalizer we can establish that:
\begin{lem}
For any $2$-morphism $(\sigma_1, ...., \sigma_n) \to (\sigma_1, ...., \sigma_n)$ there is a a unique map in $\M$
$$ \int_{\sigma_1}^{\sigma_n} \G \alpha \to  \int_{\sigma_1}^{\sigma_n} \G \alpha'.$$

These maps are compatible and form a functor  $\S_{\ol{[\F,\G]}}(\sigma_1, \sigma_n)^{op} \to \M$, that takes $(\sigma_1, ...., \sigma_n)$ to 
$ \int_{\sigma_1}^{\sigma_n} \G \alpha$.\\

We will denote this functor by $\Homco(\F,\G)$; it's an object of $\M_\S([\F,\G])$.
\end{lem}
\begin{rmk}
It's not hard to see that if $\F$ and $\G$ are strict $\M$-categories then $\Homco(\F,\G)$ is a locally constant diagram; that is a category. 
\end{rmk}

\begin{df}
Let $\F$ and $\G$ be two precategories that satisfy the co-Segal conditions.  Define thhe\footnote{Drinfeld called `thhe' the homotopy version of `the'} co-Segal category of morphism from $\F$ to $\G$ to be a fibrant replacement $\rhom$ of $\Homco(\F,\G)$ in the model category $\M_\S([\F,\G])^{\cb}$. 
\end{df}
\section{Distributors}\label{sec-distrib}
We discuss here the corresponding notion of \emph{distributor} also called \emph{bimodule} or even simply \emph{module}. Among the distributors we have the ones that correspond to the Yoneda functor associated to an object. The idea behind our definition is the notion of \emph{join of categories} that we've considered already in \cite{SEC1} under the terminology \emph{bridge}. The join construction is also behind the definition of comma categories for quasicategories  given by Joyal \cite{Joyal_qcat} and used by Lurie \cite{Lurie_HTT}.

\begin{df}
Let $\F \in \msx$ and $\G \in \msy$ be two precategories.  A \emph{distributor} $\E: \F \nrightarrow \G$ or \emph{$(\F, \G)$-bimodule} is an object 
$\E \in \M_\S(X \sqcup Y)$ such that:  
\begin{enumerate}
\item If $A \in X$ and $U \in Y$ then the diagram
 $$\E: \S_{\ol{X \sqcup Y}}(A,U)^{op} \to \M,$$
is the constant diagram of value the initial object $\emptyset$ of $\M$; and
\item $(i_X)^{\star}\E \cong \F$ and $(i_Y)^{\star}\E \cong \G$, where
$$i_X: X \hookrightarrow X \sqcup Y, \quad \tx{and} \quad i_Y: Y \hookrightarrow X \sqcup Y,$$ 
are the canonical inclusions
\end{enumerate}
An $(\F,\Un)$-bimodule $\E \in \M_\S(X \sqcup \{\ast\})$ will be simply called an \emph{$\F$-module}. If $\F$ and $\G$ satisfy the co-Segal conditions then \emph{a co-Segal distributor} is a distributor satisfying also the co-Segal conditions.
\end{df}

The first definition of an $(\F,\G)$-bimodule that we considered was a lax functor
 $$\E: (\S_{\ol{X} < \ol{Y}})^{2\tx{-$op$}} \to \M,$$ such that it restriction to $X$ (resp. $Y$) is $\F$ (resp. $\G$). Here $\ol{X} < \ol{Y}$ represents the join category: there is no morphism whose source is an element of $Y$; and there is exactly a single morphism from an element of $X$ to an element of $Y$.\\
But the only problem is that $\ol{X} < \ol{Y}$ is not of the form $\ol{Z}$ which means that $\E$ is not an object of $\mset$. However we will use this vision to establish some of the results. 

\begin{rmk}
There is a canonical functor $\ol{X} < \ol{Y} \to \ol{X \sqcup Y}$ and we get an induced functor 
$$ \Lax[(\S_{\ol{X \sqcup Y}})^{2\tx{-$op$}}, \M] \to \Lax[(\S_{\ol{X} < \ol{Y}})^{2\tx{-$op$}}, \M].$$
This functor has a left adjoint and any bimodule viewed as an object of $\Lax[(\S_{\ol{X} < \ol{Y}})^{2\tx{-$op$}}, \M]$ can also be viewed as a bimodule in the sense of the previous definition.
\end{rmk}
\subsection{Yoneda module}
Let $\F \in \msx$ be a precategory and let $A$ be an object of $\F$, and let  $\ol{X}<1$ the join category.  We have two canonical embeddings:

\begin{itemize}[label=$-$]
\item  $i_X: \sxop \hookrightarrow \sxdop$
\item $i_1: (\S_{\1})^{2\tx{-op}} \hookrightarrow  \sxdop$
\end{itemize} 
For any lax functor $\E: \sxdop \to \M$ we have two restrictions 
$$ (i_X)^{\star} \E: \sxop \to \M  \ \ \ \tx{and} \ \ \ \ \ \ (i_1^{\star}) \E: (\S_{\1})^{2\tx{-op}} \to \M.$$

In what follows we define a co-Segal module  that we interpret as being the Yoneda functor `$\F(-,A)$'. Since we know that $\F_A$ satisfies the boundary conditions: ${\F_A}_{|\sxop}= \F$ and ${\F_A}_{|\Depiop}=\Un$ then the only data we need to specify are :
\begin{itemize}[label=$-$]
\item the components $\F_{A, B \ast}: \S_{\xdiez}(B,\ast)^{op} \to \M$;
\item the laxity maps : $\F_{A, BC}  \otimes \F_{A, C\ast} \to \F_{A, B \ast}$
\end{itemize}

\paragraph{Description of $\S_{\xdiez}(B, \ast)$} In the category $\S_{\xdiez}(B, \ast)$ there a two type of morphisms:
\begin{enumerate}
\item the ones of the form $(B,...,C, \ast)$ with only one instance of $\ast$ (i.e $C \neq \ast$);
\item and the ones which are degeneracies of the previous i.e with multiple $\ast$ e.g $(B,..., C, \ast ,\ast, ..., \ast)$. 
\end{enumerate}
Define the \emph{support} of a morphism $(B,..., C, \ast ,\ast, ..., \ast)$ to be $(B,..., C, \ast)$ that is the chain where we only keep a single $\ast$ at the end. Taking the support defines a function from the set of morphisms of the form $(2)$ above, to the set of morphisms of  the form $(1)$. We define $\F_{A, B \ast}$ on the objects by declaring that the image of a chain $(B,..., C, \ast ,\ast, ..., \ast)$ is the image of its support, that is:
$$\F_A(B,..., C, \ast, \ast,...,\ast)= \F_A(B,...,C, \ast):= \F(B,...,C,A).$$

It remains to define $\F_{A, B \ast}$ on morphisms of $\S_{\xdiez}(B, \ast)$ and check that it is a functor. To define properly this we need to observe some  facts about $\sx$ and $\S_{\xdiez}$.
\begin{rmk}\ \
\begin{enumerate}
\item We leave the reader to check that in $\S_{\xdiez}(B, \ast)$, the full subcategory of chains of the form $(1)$ is isomorphic to $\sx(B,A)$. The isomorphism is clear: `replace $A$ by $\ast$ and vice versa'.  We therefore have an embedding $\iota: \sx(B,A) \hookrightarrow \S_{\xdiez}(B,\ast)$.
\item Now as $\S_{\xdiez}(\ast, \ast) \cong \Depi$,  if we use the composition in $\S_{\xdiez}$ and the previous embedding $\iota$ we get a functor:  
$$\sx(B,A) \times \Depi \xhookrightarrow{} \S_{\xdiez}(B,\ast) \times \S_{\xdiez}(\ast, \ast) \xrightarrow{comp} \S_{\xdiez}(B,\ast).$$  
The composition in both $\sx$ and $\S_{\xdiez}$ is the concatenation of chains side by side and is governed by the ordinal addition in $\Depi$. It's not hard to see that the ordinal addition in $\Depi$ is faithful (injective on morphisms) on the two variables. Consequently the composition functor in both   $\sx$ and $\S_{\xdiez}$ is also a faithful  bifunctor. 
\end{enumerate}
\end{rmk} 

% Recall that in $\Depi$ the morphism are generated by the coface maps $\sigma_n^i: \n+\1 \to \n$ ($\sigma_n^i(i)=\sigma_n^i(i+1)$, $i\leq n-1$). Since the $2$-morphisms in $\sx$ and $\S_{\xdiez}$ are parametrized by the one of $\Depi$, it turns out the morphism in 

Going back to our original problem to define $\F_{A, B \ast}$ on morphisms of $\S_{\xdiez}(B,\ast)$, we proceed as follows. First we need to keep in mind  that the morphisms in $\S_{\xdiez}(B,\ast)$ are parametrized by the one of $\Depi$; they consist to delete letters in chains while keeping the extremities $B$ and $\ast$ fixed. \ \\

But in $\Depi$ the morphism are generated by the coface maps $\sigma_n^i: \n+\1 \to \n$ ($\sigma_i^n(i)=\sigma_i^n(i+1)$, $i\leq n-1$). These maps govern the morphisms in $\S_{\xdiez}(B,\ast)$ that consist to delete a single letter or $\ast$ in a chain; moreover they generate by composition all the other maps. 
Consequently it's sufficient to define $\F_{A, B \ast} (\sigma_i^{op})$ where $\sigma_i:  (B,..., ..., \ast,...,\ast) \to (B,..., ..., \ast,...,\ast)$ delete exactly one letter or one $\ast$.  We distinguish few cases below.

\begin{enumerate}
\item If $\sigma_i$ is in the image of the faithful functor $\sx(B,A) \times \Depi \hookrightarrow\S_{\xdiez}(B,\ast)$:
\begin{itemize}
\item then either $\sigma_i$ deletes a letter, hence $\sigma_i= \iota(\sigma_l') \otimes \Id_{\k}$ with $\sigma_l':(B,...,C,A) \to (B,..., C', A)$ a (unique) morphism in $\sx(B,A)$. In that case we set:
$$\F_{A, B \ast} (\sigma_i^{op}):= \F(\sigma_l'^{op}): \F(B,...,C',A) \to \F(B,...,C,A).$$

\item or $\sigma_i$ deletes an $\ast$, hence $\sigma_i= \iota(\Id_{(B,...,C,A)}) \otimes f $ with $f \in \Depi$. In that case we set:
 $$\F_{A, B \ast} (\sigma_i^{op}) := \Id_{\F(B,...,C,A)}.$$
\end{itemize} 
\item If $\sigma_i$ is not in the image of $\sx(B,A) \times \Depi \hookrightarrow\S_{\xdiez}(B,\ast)$ then this means that $\sigma_i$ deletes a $\ast$ which is between a letter and another $\ast$, e.g $(B, \ast, \ast) \to (B,\ast)$.  In that case we set:

$$\F_{A, B \ast} (\sigma_i^{op}) := \Id_{\F(B,...,C,A)}$$ 
where  $(B,...,C,\ast)$ is the support of the source and target of $\sigma_i$
\end{enumerate}
\ \\
For a general morphism $\alpha: (B,...,\ast) \to (B, ..., \ast)$ which is parametrized by a morphism $\Le(\alpha): \n \to \m$ of $\Depi$, one can write
in a unique way $\alpha =\sigma_{j_1} \circ \cdots \circ \sigma_{j_{n-m}}$; where the string of subscripts $j$ satisfy:
$$0 \leq j_1 < \cdots < j_{m-n}< n-1.$$ 
This follows from the fact that each morphism $f: \n \to \m$ of $\Depi$ has a \emph{unique presentation} $f= \sigma_{j_1} \circ \cdots \circ \sigma_{j_{n-m}}$  with 
$0 \leq j_1 < \cdots < j_{m-n}< n-1.$ (see \cite[p.177]{Mac}). \ \\

With these notations we define $\F_{A, B \ast} (\alpha^{op}):= \F_{A, B \ast}(\sigma_{j_{n-m}}^{op}) \circ \cdots \circ \F_{A, B \ast}(\sigma_{j_{1}}^{op}).$
\begin{lem}
The above data define a functor $\F_{A, B \ast}: \S_{\xdiez}(B,\ast)^{op} \to \M$.
\end{lem}

\begin{proof}[Sketch of proof]
We have to check that $\F_{A, B \ast}((\alpha \circ \beta)^{op})= \F_{A, B \ast}(\beta^{op}) \circ \F_{A, B \ast}(\alpha^{op})$. But this boils down to checking that $\F_{A, B \ast}$ respect the simplicial identities `$\sigma_j \circ \sigma_i = \sigma_i \circ \sigma_{j+1} (i\leq j)$' i.e., 
$$ \F_{A, B \ast}(\sigma_i^{op}) \circ  \F_{A, B \ast}(\sigma_j^{op})= \F_{A, B \ast}(\sigma_{j+1}^{op}) \circ \F_{A, B \ast}(\sigma_i^{op}).$$
This is an easy exercise and is left to the reader. The only thing one needs to use is the fact that the component $\F_{BA}: \sx(B,A) \to \M$ is a functor, thus distributes over the composition; and the fact that we have an equality $h \circ \Id= \Id \circ h$ for any morphism $h$ of $\M$. 
\end{proof}

The laxity maps $\F_A(B, ..., C) \otimes \F_A(C, ..., D, \ast, ...,\ast) \to \F_A(B, ..., D, \ast, ...,\ast)$ are given by the one of $\F$:

$$\F(B, ..., C) \otimes \F(C, ..., D, A) \to \F(B,...,D,A)= \F_A(B, ..., D, \ast,...,\ast).$$

It's easy to check that the laxity map satisfy the coherence conditions. Furthermore since  $\F_{A, B \ast}$ clearly takes its values in the subcategory of weak equivalences (as $\F_{BA}$ does), then we've just proved that 
\begin{prop}
$\F_A: \sxdop \to \M$ is a lax functor that satisfies the co-Segal conditions.
\end{prop}

\begin{df}
Define the Yoneda module of $\F$ at $A$ to be the image of $\F_A$ in $\M_\S(\ol{X \sqcup Y})$ by the left adjoint of 
$$\M_\S(\ol{X \sqcup \{\ast\}})  \to \Lax[(\S_{\ol{X} < 1})^{2\tx{-$op$}}, \M].$$
\end{df}

Some modules arise as morphisms $\E: \F^{op} \to \M$ where $\M$ is considered as strict co-Segal category enriched over itself. This class of modules will play an important role in the upcoming papers.
\newpage
\appendix
\section{Lemmata}
We include here some of the proofs of the paper. 
\subsection{Proof of Proposition \ref{limit-msxsu}}\label{ap-limit-msxsu}
\begin{proof}
For the three assertions we need to establish that the following commutes.
\begin{equation}
\begin{gathered}
\xy
(0,18)*+{I \otimes F_\infty(s)}="W";
(0,0)*+{F_\infty(A,A) \otimes F_\infty(s) }="X";
(50,0)*+{F_\infty[(A,A) \otimes s]}="Y";
(50,18)*+{F_\infty(s)}="E";
{\ar@{->}^-{\varphi_\infty}"X";"Y"};
{\ar@{->}^-{I_A \otimes \Id}"W";"X"};
{\ar@{->}_-{\cong}^{l}"W";"E"};
{\ar@{->}^-{F_\infty(\sigma)}"E";"Y"};
\endxy
\end{gathered}
\end{equation}

This means that we want to prove the following equality.
\begin{equation}\label{eq2}
F_\infty(\sigma) \circ l= \varphi_\infty \circ (I_A \otimes \Id).
\end{equation}

Since $F_j$ is unital we know that we have this equality:
\begin{equation}
F_j(\sigma) \circ l= \varphi_j \circ (I_A \otimes \Id).
\end{equation}
For Assertion $(1)$ we use the fact that in $\msx$, and $\msxpt$ limits a are computed object-wise. In fact this is true in any category of lax diagrams in general. Let $p_j: F_\infty \to F_j$ be the canonical projection for each $j$.
By construction the laxity map $\varphi_\infty$ is the \ul{unique} map that gives the following equality for every $j$.
\begin{equation}\label{eq5}
\varphi_j \circ (p_j \otimes p_j)= p_j \circ \varphi_\infty
\end{equation}

In order to guide the reader we display only once the meaning of this type of equality. That equality is equivalent to saying that the commutative diagram hereafter commutes.

\begin{equation} 
\begin{gathered}
\xy
(0,18)*+{F_\infty(A,A) \otimes F_\infty(s)}="W";
(0,0)*+{F_j(A,A) \otimes F_j(s) }="X";
(50,0)*+{F_j[(A,A) \otimes s]}="Y";
(50,18)*+{F_\infty[(A,A) \otimes s]}="E";
{\ar@{->}^-{\varphi_j}"X";"Y"};
{\ar@{->}^-{p_j \otimes p_j}"W";"X"};
{\ar@{-->}^-{\varphi_\infty}"W";"E"};
{\ar@{->}^-{p_j}"E";"Y"};
\endxy
\end{gathered}
\end{equation}
The unity $I_A: I \to F_j(A,A)$ induces a unique map $I_A: I \to F_\infty(A,A)$ with the obvious factorizations $p_j \circ I_A= I_A$. Since there is a danger of confusion we will write $I_j$ and $I_\infty$ the respective unity. Then with this notation, that factorization becomes
\begin{equation}
p_j \circ \Idinf= I_j.
\end{equation}

Similarly we will write $l_j$ and $l_\infty$ the components of the natural isomorphism $$l:I \otimes- \xrightarrow{\cong} \Id.$$  

From the bifunctoriality of $\otimes$ and  $l$, we have some commutative squares that we represent by the equalities hereafter.
\begin{equation}\label{eq3}
p_j \circ \linf= l_j \circ (\Id \otimes p_j) 
\end{equation}
\begin{equation}\label{eq1}
p_j \otimes p_j \circ (\Idinf \otimes \Id) = (I_j \otimes \Id) \circ (\Id \otimes p_j) 
\end{equation}

Finally we have a third commutative diagram given by the naturality of $p_j$. We represent it by the following equality.
\begin{equation}\label{eq4}
F_j(\sigma) \circ p_j= p_j \circ \Finf(\sigma) 
\end{equation}

We form a compatible cone starting at $I \otimes F_\infty(s)$ and that projects to each $F_j[(A,A) \otimes s]$ by the map $\pi_1:I \otimes F_\infty(s) \to F_j[(A,A) \otimes s]$ given by
$$\pi_1= \varphi_j \circ (p_j \otimes p_j) \circ (\Idinf \otimes \Id).$$
In order to see what's happening we display that map below.
$$I \otimes F_\infty(s) \to F_\infty(A,A) \otimes F_\infty(s) \xrightarrow{p_j \otimes p_j} F_j(A,A) \otimes F_j(s) \xrightarrow{\varphi_j} F_j[(A,A) \otimes s].$$ 

From the universal property of the limit, there is a unique map 
$$\varepsilon_\infty: I \otimes F_\infty(s) \to   \Finf[(A,A) \otimes s],$$ such that 
$\pi_1= p_j \circ \varepsilon_\infty$ where  $p_j: \Finf[(A,A) \otimes s] \to F_j[(A,A) \otimes s]$ is the canonical projection. 

\begin{claim}
The two maps $F_\infty(\sigma) \circ l_\infty$  and  $\varphi_\infty \circ (I_\infty \otimes \Id)$ gives this factorization. Consequently they are equal by uniqueness and $\Finf$ is strongly unital.
\end{claim}
To see why the claim holds, we use the equalities $(\ref{eq1})$, $(\ref{eq2})$, $(\ref{eq3})$, $(\ref{eq4})$, in this order from the top to the bottom and we establish the following.

\begin{equation*}
\begin{split}
\pi_1&= \varphi_j \circ \underbrace{(p_j \otimes p_j) \circ (\Idinf \otimes \Id)}_{=(I_j \otimes \Id) \circ (\Id \otimes p_j)}\\
&= \underbrace{\varphi_j \circ (I_j \otimes \Id )}_{=F_j(\sigma) \circ l_j}\circ (\Id_I \otimes p_j)\\
&= F_j(\sigma) \circ (\underbrace{l_j \circ(\Id_I \otimes p_j)}_{=p_j \circ l_\infty})  \\
&= [\underbrace{F_j(\sigma) \circ p_j}_{=p_j  \circ \Finf(\sigma)} ] \circ  l_\infty \\
&= [ p_j  \circ \Finf(\sigma)] \circ  l_\infty \\
&= p_j  \circ [\Finf(\sigma) \circ  l_\infty] \quad \checkmark \\
\end{split}
\end{equation*}

The other equality is easily seen to hold using the equality $(\ref{eq5})$.

\begin{equation*}
\begin{split}
\pi_1&= \underbrace{\varphi_j \circ (p_j \otimes p_j)}_{=p_j \circ \varphi_\infty} \circ (\Idinf \otimes \Id)\\
&=p_j \circ [\varphi_\infty \circ (\Idinf \otimes \Id)] \quad \checkmark
\end{split}
\end{equation*}
This completes the proof of Assertion $(1)$.\\

For Assertion $(2)$ the idea is the same. Since $\M$ is monoidal closed, colimits distribute over $\otimes$; in particular $I \otimes -$ preserves colimits. Therefore $I \otimes \Finf(s)$ is the colimit (in $\M$) of the diagram of the $I \otimes F_j(s)$. Note that by hypothesis $\Finf[(A,A) \otimes s]$ is the colimit of the $F_j[(A,A) \otimes s]$ . We will denote by $\rho_j: F_j \to \Finf$ the canonical map.\\

We form a compatible cocone ending at $\Finf[(A,A) \otimes s]$ where the canonical map for each $j$, is the map $\gamma_j: I \otimes F_j(s) \to \Finf[(A,A) \otimes s]$ defined as the following composite.
$$I \otimes F_j(s)\to F_j(A,A)\otimes F_j(s) \xrightarrow{\varphi_j} F_j[(A,A) \otimes s] \xrightarrow{\rho_j} \Finf[(A,A) \otimes s].$$

From the universal property of the colimit, there is a unique map
$$\eta_1: I \otimes \Finf(s) \to  \Finf[(A,A) \otimes s]$$ such that for each $j$ the following factorization holds. 
$$\gamma_j= \eta_1 \circ (\Id_I \otimes \rho_j).$$

Just like before we have
\begin{claim}
The two maps $F_\infty(\sigma) \circ l_\infty$  and  $\varphi_\infty \circ (I_\infty \otimes \Id)$ gives this factorization. And by uniqueness they are equal, consequently $\Finf$ is strongly unital 
and each $\rho_j$ becomes a map of unital precategories. Furthermore $\Finf$ if the colimit in $\msxsu$ of the original diagram. 
\end{claim}
One establishes the claims just like in the previous using the various commutative diagrams that come with the definition of each morphism $\rho_j: F_j \to \Finf$, and the functoriality of $\otimes$, $l$. We give the different steps to get the desired equalities below and leave the reader to check the details. In each step the expression in the square brackets is the one that we replace to get the next line below.\\
\begin{equation*}
\begin{split}
\gamma_j&= \rho_j \circ [\varphi_j \circ (I_j \otimes \Id)]\\
&=[\rho_j \circ F_j(\sigma)] \circ l_j \\
&=\Finf(\sigma) \circ [\rho_j \circ l_j] \\
&=\Finf(\sigma) \circ [l_\infty \circ (\Id_I \otimes \rho_j)] \\
&=[\Finf(\sigma) \circ l_\infty] \circ (\Id_I \otimes \rho_j) \quad \checkmark\\
\end{split}
\end{equation*}
\\
Similarly we establish the following. 
%%%%%%%
\begin{equation*}
\begin{split}
\gamma_j&= [\rho_j \circ \varphi_j] \circ (I_j \otimes \Id)\\
&= \varphi_\infty \circ [(\rho_j \otimes \rho_j) \circ (I_j \otimes \Id)]\\
&= \varphi_\infty \circ [(\Idinf \otimes \Id) \circ (\Id_I \otimes \rho_j)]\\
&= [\varphi_\infty \circ (\Idinf \otimes \Id)] \circ (\Id_I \otimes \rho_j) \quad \checkmark\\
\end{split}
\end{equation*}
%%%%%%%
Assertion $(3)$ is a corollary of Assertion $(2)$ because we showed in \cite{COSEC1}, based on ideas of Linton \cite{Linton-cocomplete} and Wolff \cite{Wo}, that for lax diagrams in general, coequalizer of reflexive pairs are computed level-wise. The same holds for directed colimits. Therefore given a reflexive pair in $\msxsu$, the coequalizer exists in $\msx$ and is computed level-wise.
\end{proof}

\subsection{Proof of Proposition \ref{lem-creat-colimit-trois}}\label{ap-lem-creat-colimit-trois}
To prove the Proposition we will use the language of locally Reedy $2$-categories (see \cite{These_Hugo,COSEC1}). 
\begin{proof}
Let $\Ea: \J \to \msxsu$ be a diagram in $\msxsu$. Thanks to Proposition \ref{limit-msxsu}, we know already that the colimit exists and we will denote it by $\Ea_\infty$. Below we will construct a unital precategory $\Za$ that satisfies the universal property of the colimit, this way we will have $\Ea_\infty \cong \Za$.\\

Let $\zn$ be the colimit of the truncated diagram $\Ea_{\leq \n}= \tau_{\n}(\Ea)$ and let $\rho_j: \tau_\n(\Ea_j) \to \zn$ be the canonical map for each $j$. The idea is to extend inductively $\zn$ into a $\znplus$ and so on. \\

Let $s \in \sx(A,B)^{op}$ be a $1$-morphism of degree $\n+\1$. Let $T= \laxlatch(\zn, s)$ be the \emph{lax-latching object} of  $\zn$ at $s$ (see \cite{These_Hugo,COSEC1}). Each canonical map $\rho_j$ induces by functoriality of the lax-latching object, a map as follows.
$$\rho_j: \laxlatch(\tau_\n(\Ea_j), s) \to  \laxlatch(\zn, s).$$

We also have a canonical map $\iota_j:  \laxlatch(\tau_\n(\Ea_j), s) \to \Ea_j(s)$ which is functorial in $j$. This means  that we have a functor $\iota: \J \to \Ar(\M)$ that takes $j$ to $\iota_j$.
%%%%%%%%%%% Rajouter quelque chose ici %%%%%%%%%%%%%
Using the previous maps, one forms a diagram consisting of spans (or pushout data) as follows.
$$\Ea_j(s)  \xleftarrow{\iota_j} \laxlatch(\tau_\n(\Ea_j), s) \xrightarrow{\rho_j} \laxlatch(\zn, s).$$

These pushout data are connected from one $j$ to another by the functor $\iota$ and the canonical cocone that defines the colimit $\zn$. Let $m_s$ be the colimit of that diagram.\\
When we arrive here we need to distinguish four cases for $s$ as follows. 
\begin{enumerate}
\item Either $s= (A,A) \otimes t$, for some $t=(A,...,B)$;
\item or $s=t \otimes (B,B)$; 
\item or $s$ is both i.e, $s=(A,A) \otimes t \otimes (B,B)$ for some $t$;
\item or finally $s$ has none of the previous form.
\end{enumerate}

If we are in the fourth case, we set right away $\znplus(s)=m_s$. If we are in the first or second case, we need to force the unity conditions. And since the method is the same for both cases we will only outline the operation in the first case.\\

By definition, the object $m_s$ comes equipped with a canonical map $can_\sigma: \zn(t) \to  m_s$, as $\sigma$ runs through the set of all morphisms from $t$ to $s$. The morphisms $\sigma$ are objects of the usual latching category of $\sx(A,B)^{op}$ at $s$; and this latching category is in turn part of the lax-latching category of $\sxop$ at $s$. We also have a (unique) laxity map which is just the canonical map going to the colimit $$\varphi : \zn(A,A)\otimes \zn(t) \to m_s.$$

Then we define  $\znplus(s)$ to be the (simultaneous) coequalizer of the following parallel maps for all $\sigma: t \to s$.
 $$I \otimes \zn(t) \to \zn(A,A) \otimes \zn(t) \xrightarrow{\varphi} m_s;$$ 
$$I \otimes \zn(t) \to \zn(t) \xrightarrow{can_\sigma} m_s.$$

We have a canonical map $\delta_{A,t}: m_s \to \znplus(s)$. If we are in the second case we will have a map $\delta_{t,B}: m_s \to \znplus(s)$.  If we precompose each map with the respective map $can_\sigma: \zn(t) \to m_s$, we find a map $ \zn(t) \to \znplus (s)$ that we define to be $\znplus(\sigma)$. We have also a laxity map:
$$ \zn(A,A) \otimes \zn(t) \xrightarrow{\varphi} \znplus(s).$$

Now if we are the third case, we proceed like before but this time we don't define yet $\znplus(s)$ as the coequalizer. Denote by $\delta_{A,t}: m_s \to P_{t,A}$ and  $\delta_{t,B}: m_s \to P_{B,t}$, the respective canonical maps going to the respective \emph{coequalizing-object}. Then we define $\znplus(s)$ as the pushout object obtained by forming the following pushout.
$$P_{A,t} \xleftarrow{\delta_{A,t}} m_s \xrightarrow{\delta_{t,B}} P_{t,B}.$$
 
Proceeding like this for all $s$ of degree $\n+\1$ we find an object $\znplus \in \Lax[\sxopnp, \M]_{su}$ and by induction one finds an object $\Za_{\infty} \in \msxsu$. It's clear from the construction that every level $\zn$ satisfies the universal property of  the colimit of $\tau_\n(\Ea)$ therefore $\Za_\infty$ is indeed the colimit of $\Ea$ as desired. This proves Assertion $(1)$.\\

Assertion $(2)$ is a consequence of Assertion $(1)$ together with the fact that $\Lax[\sxopun, \M]_{su}$ is isomorphic to the category of pointed $\M$-graphs with vertices $X$:
 $$\Lax[\sxopun, \M]_{su} \cong \Lax[\sxopun, \M]_{\ast} \cong I_X \downarrow \mgraphx.$$
\end{proof}

\subsection{Proof of Lemma \ref{lem-mon-proj-left}}\label{ap-lem-mon-proj-left}
\begin{proof}
The proof is based on the following crucial facts.
\begin{itemize}[label=$-$]
\item Any functorial operation in $\msx, \msxpt$ and $\msxaaub$ is done level-wise at the one morphism $(A,B)$ e.g colimits. As we said previously this is because the $1$-morphism $(A,B)$ is not submitted to any algebraic constraint so that $\F(A,B)$ does \ul{not} receive any laxity map.
\item Directed colimits in $\msx, \msxpt$ and in  $\msxaaub$ are computed level-wise (and also limits).
\end{itemize}

Assertion $(1)$ is a consequence of Proposition \ref{eval-invar-ab}. In fact if we write down the pushout defining $\F_1$, and look at the diagram at the $1$-morphism $(A,B)$, one gets the following.
\[
\xy
%%%%%%%%%%%face arriere%%%%%
(-10,20)*+{\emptyset}="A";
(10,20)+(15,0)*+{\F(A,B)}="B";
(-10,0)*+{\emptyset}="C";
%(10,0)+(15,0)*+{\F_k[(A,A)\otimes s]}="D";
{\ar@{->}^-{}"A";"B"};
{\ar@{->}_-{\Upsilon_{[(A,A)\otimes s]}^{\Id}(j_k)}"A"+(0,-3);"C"};
\endxy
\xy
%%%%%%%%%%%face arriere%%%%%
(-10,20)*+{I}="A";
(10,20)+(15,0)*+{\F(A,A)}="B";
(-10,0)*+{I}="C";
%(10,0)+(15,0)*+{\F_k[(A,A)\otimes s]}="D";
{\ar@{->}^-{}"A";"B"};
{\ar@{->}_-{\Upsilon_{[(A,A)\otimes s]}(j_k)}^{\Id}"A"+(0,-3);"C"};
\endxy
\]

It's clear that in both cases the pushout is just the object $\F(A,B)$ itself (or isomorphic to it). So the component $\delta_1: \F(A,B) \to \F_1(A,B)$, is an isomorphism; and by induction one has that every component $\delta_k: \F_k(A,B)\to \F_{k+1}(A,B)$ is an isomorphism which gives the assertion. Assertion $(2)$ is a tautology, and is a corollary of Assertion $(1)$.\\

To prove Assertion $(3)$ we proceed in the following manner. Thanks to the \emph{crossing lemma} (Lemma \ref{cross-lem}) we know that $m_k$ and $\F_k[(A,A)\otimes s]$ have the same colimit in the sense that the maps $j_k: \F_k[(A,A)\otimes s]\to m_k $ and $\xi_k: m_k \to \F_{k+1}[(A,A)\otimes s]$ induce maps between the respective colimit that are inverse each other:
$$j_{\infty}: \colim \F_k[(A,A)\otimes s] \to \colim m_k,$$ 
$$ \xi_{\infty}: \colim m_k \to \colim  \F_k[(A,A)\otimes s] \ \ \ \tx{with} \ \ \  \xi_{\infty} = (j_{\infty})^{-1}.$$
 
 But since directed colimits are computed level-wise we have precisely that 
 $$P_{s}^l(\F)[(A,A)\otimes s]= \colim  \F_k[(A,A)\otimes s],$$
 and similarly 
 $$P_{s}^l(\F)(s)= \colim  \F_k(s).$$
 
If we put this in the diagram in Step $6$ of the construction, we get the following commutative diagram.
 \[
\xy
%%%%%%%%%%%face arriere%%%%%
%(-10,20)*+{P_{s}^l(\F)(s)}="A";
(10,13)+(3,0)*+{\colim m_k}="B";
(-10,0)*+{P_{s}^l(\F)[(A,A)\otimes s]}="C";
(20,0)+(20,0)*+{P_{s}^l(\F)[(A,A)\otimes s]}="D";
%{\ar@{->}^-{\alpha_\infty}"A";"B"};
%{\ar@{->}_-{P_{s}^l(\F)(\sigma)}"A"+(0,-3);"C"};
{\ar@{->}^-{\xi_\infty}_{\cong}"B";"D"};
{\ar@{->}_-{\Id}"C";"D"};
{\ar@{->}^-{j_\infty}_{\cong}"C";"B"};
\endxy
\]

On the other hand if we take the colimit in the coequalizer diagram of Step $2$ we get the following commutative square.
\[
\xy
(0,18)*+{I \otimes P_{s}^l(\F)(s)}="W";
(0,0)*+{P_{s}^l(\F)(A,A) \otimes P_{s}^l(\F)(s) }="X";
(110,0)*+{\colim m_k}="Y";
(50,18)*+{P_{s}^l(\F)(s)}="E";
(60,0)*+{P_{s}^l(\F)[(A,A)\otimes s]}="G";
(110,18)*+{P_{s}^l(\F)[(A,A)\otimes s]}="F";
{\ar@{.>}^-{j_\infty}"G";"Y"};
{\ar@{->}^-{I_A \otimes \Id}"W";"X"};
{\ar@{->}^-{\cong}"W";"E"};
{\ar@{->}^-{\varphi}"X";"G"};
{\ar@{.>}^-{j_\infty}"F";"Y"};
{\ar@{->}^-{P_{s}^l(\F)(\sigma)}"E";"F"};
\endxy
\]

It's now clear that if we extend that commutative square with the map 
$$ \xi_\infty: \colim m_k \to P_{s}^l(\F)[(A,A)\otimes s],$$
 we get our desired commutative diagram.
This means that $P_{s}^l(\F) \in \msxaaub$.\\

It remains to show that this is indeed a left adjoint.  Let $h: \F \to \G$ be a map in $\msxpt$ with\footnote{We should write $\Ub\G$ instead of $\G$} $\G \in \msxaaub$. Being a map in $\msxpt$, implies that if we use the component $h: \F[(A,A)\otimes s] \to \G[(A,A)\otimes s]$,
the two parallel maps below become equal in $\G[(A,A)\otimes s]$.
$$I \otimes \F_k(s) \xrightarrow{\cong} \F_k(s)\xrightarrow{\F_k(\sigma)} \F_k[(A,A)\otimes s],$$
$$I \otimes \F_k(s) \xrightarrow{I_A \otimes \Id} \F_k(A,A) \otimes \F_k(s) \xrightarrow{\varphi} \F_k[(A,A)\otimes s].$$

Therefore there is a unique map $ p: m_0 \to G[(A,A)\otimes s]$ such that $h= p \circ j_0$ where $$j_0: \F[(A,A)\otimes s] \to m_0,$$ is the canonical map going to the coequalizer. But this factorization of $h$ yields by adjunction the following commutative square. 
\[
\xy
%%%%%%%%%%%face arriere%%%%%
(-10,20)*+{\Upsilon_{[(A,A)\otimes s]}(\F[(A,A)\otimes s])}="A";
(10,20)+(15,0)*+{\F}="B";
(-10,0)*+{\Upsilon_{[(A,A)\otimes s]}(m_0)}="C";
(10,0)+(15,0)*+{\G}="D";
{\ar@{->}^-{}"A";"B"};
{\ar@{->}_-{\Upsilon_{[(A,A)\otimes s]}(j_0)}"A"+(0,-3);"C"};
{\ar@{->}^-{h}"B";"D"};
{\ar@{->}_-{}"C";"D"};
\endxy
\]

Now since $\F_1$ is the pushout of
 $$ \Upsilon_{[(A,A)\otimes s]}(m_0) \leftarrow \Upsilon_{[(A,A)\otimes s]}(\F[(A,A)\otimes s]) \to \F,$$
 there is a unique map $h_0: \F_1 \to \G$ such that $h= h_1 \circ \delta_0$. Proceeding by induction, one finds a unique map $h_k: \F_{k+1} \to \G$ such that
$h_{k-1}= h_{k} \circ \delta_k$. The universal property of the colimit gives a unique $h_\infty: P_{s}^l(\F) \to \G$ such that $h_k= h_\infty \circ can_k$, where $can_k: \F_k \to P_{s}^l(\F)$ is the canonical map (in particular $can_0= \eta$). And we have $h= h_\infty \circ \eta$ and Assertion $(3)$ follows.\\

We shall now check that Assertion $(4)$ holds. First, we know that $P_{s}^l(\F)$ is an object of $\msxaaub$; therefore to prove that we have a monadic projection, we have to show that  for $\F \in \msxaaub$, the maps $\eta: \F \to P_{s}^l(\F)$ is an isomorphism and that the map $P_{s}^l(\eta)$ is also an isomorphism.\\

But if $\F \in \msxaaub$ then the coequalizer in Step $(2)$ is just the identity i.e, $m_0= \F[(A,A)\otimes s]$ and $j_0= \Id$. From there it's clear that taking the pushout defining $\F_1$ doesn't change anything, so that $\delta_0:\F \to  \F_1$ is an isomorphism (and can be chosen to be the identity).  This means that the process stops because $\delta_k: \F_k \to \F_{k+1}$ is an isomorphism for all $k$.\\

The fact that $P_{s}^l(\eta)$ is an isomorphism is a combination of two facts. The first one is that since $P_{s}^l(\F)$ is in $\msxaaub$ then by the previous argument,  the construction stops for $P_{s}^l(\F)$ i.e, $P_{s}^l(P_{s}^l(\F))\cong P_{s}^l(\F)$.

The other reason is that the construction is functorial therefore we can apply it to the morphism $\eta$ and we have a commutative diagram below. 
\begin{align*}
\xy
(-60,0)+(-24,0)*+{\F}="X";
(-60,0)+(24,0)*+{P_{s}^l(\F)}="A";
{\ar@{>}^-{\eta}"X";"A"};
%%%%%% dot dot dot 
(-84,-20)*+{\F_1=push(m_0,\F)}="S";
(-36,-20)*+{push(m_0,P_{s}^l(\F))}="T";
{\ar@{->}^-{}"X";"S"};
{\ar@{->}^-{\cong}"A";"T"};
{\ar@{-->}^-{}"S";"T"};
(-84,-40)+(0,-5)*+{P_{s}^l(\F)}="L";
(-15,0)+(-33,-40)+(12,-5)*+{P_{s}^l(P_{s}^l(\F))}="M";
%(20,0)+(0,0)+(-20,-40)+(0,-5)*+{Q(Q(\F))}="N";
%%%%%%%%%%%%% morphisms fictifs
{\ar@{.>}^-{}"S";"L"};
{\ar@{.>}^-{}"T";"M"};
{\ar@{-->}^-{P_{s}^l(\eta)}"L";"M"};
{\ar@/_4.8pc/@{.>}_{\eta_{\F}}"X"; "L"};
{\ar@/^4.8pc/@{.>}^{\eta_{P_{s}^l(\F)}}"A"; "M"};
\endxy
%%%%%%%%%%%%%%%%%%%%% dernier morphisms verticaux
\end{align*}

Now observe that on the one hand the map $\eta_{P_{s}^l(\eta_F)}$ is an isomorphism since $P_{s}^l(F)$ is in $\msxaaub$. On the other hand, the resulting map 
$h:\F \to P_{s}^l(P_{s}^l(\F))$ in the diagram, is a map in $\msxpt$ going to an object of $\msxaaub$; therefore the universality of the adjoint says that there is a \ul{unique} map $h_\infty:  P_{s}^l(\F) \to P_{s}^l(P_{s}^l(\F))$ inside $\msxaaub$ such that $h=h_\infty \circ \eta$.\\

But since $\msxaaub$ is a \ul{full} subcategory of $\msxpt$, the two maps 
$P_{s}^l(\eta)$ and $\eta_{P_{s}^l(\F)}$ of $\msxpt$ are also maps in $\msxaaub$ and both give a factorization of $h$ through $\eta$; thus they are equal by uniqueness. From the above, $\eta_{P_{s}^l(\F)}$ is an isomorphism  and we get the result. The fact that $P_{s}^l$ preserves directed colimits is straightforward since every operation appearing in its construction preserves directed colimits: coequalizer, pushout, directed colimits, etc.
\end{proof}

\bibliographystyle{plain}
\bibliography{Bibliography_These}
\end{document}